\documentclass[12pt,a4paper]{amsart}
\usepackage{amsmath}

\usepackage{amscd}

\usepackage{amsmath}

\usepackage{latexsym}

\usepackage{amsfonts}

\usepackage{amssymb}

\usepackage{amsthm}

\usepackage{graphicx}

\usepackage{verbatim}
\usepackage{extarrows}
\usepackage{mathrsfs}

\usepackage{enumerate}

\usepackage{hyperref}

\usepackage{enumerate}
\usepackage{color}
\newtheorem{definition}{Definition}[section]
\newtheorem{theorem}{Theorem}[section]
\newtheorem{lemma}[theorem]{Lemma}
\newtheorem{corollary}[theorem]{Corollary}
\newtheorem{proposition}[theorem]{Proposition}
\theoremstyle{remark}
\newtheorem{remark}[theorem]{Remark}
\numberwithin{equation}{section}
\numberwithin{theorem}{section}
\numberwithin{figure}{section}

\newcommand{\ZZ}{\mathbb{Z}}

\newcommand{\NN}{\mathbb{N}}

\newcommand{\RR}{\mathbf{R}}

 \newcommand{\ee}{\mathbf{e}}
 \newcommand{\vv}{\mathbf{V}}
  \newcommand{\uu}{\mathbf{U}}
   \newcommand{\ww}{\mathbf{W}}
    \newcommand{\uuu}{\mathbf{u}}
  \newcommand{\xx}{\mathbf{x}}
    \newcommand{\yy}{\mathbf{y}}
        \newcommand{\zz}{\mathbf{z}}
   \newcommand{\ff}{\mathbf{F}}

\DeclareMathOperator{\Div}{div}
\allowdisplaybreaks
\begin{document}
%%%%%%%%%%%%%%%%%%%%%%%%%%%%%%%%%%%%%%%%%%%%%%%%%%%%%%%%%%%%%%%%%%%%%%%%%%%%%%%%%%%%%%%%%%%%%%%%%%%%
\title[Regularity of the generalized Leray equations ]{Global regularity and  decay behavior for  Leray equations  with critical-dissipation and Its Application to Self-similar Solutions}
\author[C. Miao]{Changxing Miao}

\address{Institute of Applied Physics and Computational Mathematics, P.O. Box 8009, Beijing 100088, P.R. China }

\email{miao\_{}changxing@iapcm.ac.cn}
\author[X. Zheng]{Xiaoxin Zheng}

\address{ School of Mathematical Sciences, Beihang University,  Beijing 100191, P.R. China}
\email{xiaoxinzheng@buaa.edu.cn}

\date{\today}
\subjclass[2010]{35Q30; 35B40; 76D05. }
\keywords{}
\begin{abstract}
 In this paper, we  show the global regularity and the optimal decay of weak solutions to
the   generalized Leray problem with critical dissipation. Our method is based on the maximal smoothing effect,  $L^{p}$-type elliptic regularity  of linearization, and the action of the heat semigroup generated by the fractional powers of Laplace operator on distributions with Fourier transforms supported in an annulus. As a by-product, we shall construct a self-similar solution to the three-dimensional incompressible Navier-Stokes equations, and more importantly, prove the   global regularity and the optimal decay without additional requirement  of existing literatures.
\end{abstract}
\maketitle
%%%%%%%%%%%%%%%%%%%%%%%%%%%%%%%%%%%%%%%%%%%%%%%%%%%%%%%%%%%%%%%%%%%%%%%%%%%%%%%%%%%%%%%%%%%%%%%%%%%
\section{Introduction}\label{INTR}
%%%%%%%%%%%%%%%%%%%%%%%%%%%%%%%%%%%%%%%%%%%%%%%%%%%%%%%%%%%%%%%%%%%%%%%%%%%%%%%%%%%%%%%%%%%%%%
\setcounter{section}{1}\setcounter{equation}{0}
In  this paper, we study the regularity and  decay behavior of weak solutions to the generalized Leray problem in $\RR^{3}$ governed by
\begin{align}\label{E-forwad-leray}
\left\{
\begin{aligned}
&(-\Delta)^{\alpha}\uu- \frac{2\alpha-1}{2\alpha}\uu-\frac{1}{2\alpha}\xx\cdot \nabla \uu+\nabla P+\uu\cdot\nabla \uu=\operatorname{div}\ff, \\
&\mbox{div}\, \uu=0.\\
\end{aligned}\ \right.
\end{align}
It arises from studying the forward self-similar solutions of Cauchy problem to the three-dimensional incompressible fractional Navier-Stokes equations with $\alpha\in(0,1]$, which reads as follows
\begin{equation}\label{NS}
\left\{\begin{aligned}
&\uuu_t+\uuu\cdot\nabla \uuu+(-\Delta)^\alpha \uuu+\nabla \pi=0,\\
&\operatorname{div}\uuu=0,
\end{aligned}\right.
\end{equation}
complemented with the initial condition
\begin{equation}\label{NS-i}
\uuu(\xx,0)=\uuu_{0}(\xx)\ \ \ \ \ \mbox{in}\ \ \ \RR^{3}.
\end{equation}
When $0<\alpha<1$, the fractional Laplacian $(-\Delta)^{\alpha}$ is also called the L\'{e}vy operator
$$
(-\Delta)^{\alpha} \uuu(\xx)= c_{\alpha} {\text{p.v.}}\int_{\RR^{3}}\frac{\uuu(\xx)-\uuu(\yy)}{|\xx-\yy|^{3+2\alpha}}\;{\rm d}\yy,
\;\;\; c_{\alpha}
=\frac{\alpha(1-\alpha)4^{\alpha}\Gamma(\frac{3}{2}+\alpha)}{\Gamma(2-\alpha)\pi^{\frac{3}{2}}}.$$
To illustrate the self-similar solutions to the system \eqref{NS} firstly raised by Leray in \cite{Leray}, let us review the well-known scaling property, that is, if $(\uuu,\pi)$ is a solution to the system~\eqref{NS}, then so is $(\uuu_\lambda,\pi_\lambda)$ for each $\lambda>0,$ where
$$
\uuu_\lambda(\xx,t)= \lambda^{2\alpha-1} \uuu\big(\lambda \xx,\lambda^2t\big) \ \ \text{and}\ \   \pi_\lambda(\xx,t)=\lambda^{2(2\alpha-1)} \pi\big(\lambda \xx,\lambda^2t\big).
$$
It has been shown  that the scaling invariance plays an essential  role in the regularity theory of the incompressible Navier-Stokes equations.  This scaling indicates that the ratio $\frac{|\xx|^2}{t}$ is important and suggests that we search for a special solution  which is scale invariant with respect to the natural scaling, \[\uuu_{\lambda}=\uuu \quad\text{and}\quad  \pi_{\lambda}=\pi \quad\text{for each}\quad  \lambda>0, \] which are often called self-similar solutions. Taking $\lambda(t)=t^{-\frac{1}{2\alpha}}$ with  $t>0$ (or $\lambda(t)=(-t)^{-\frac{1}{2\alpha}}$ with  $t<0$),
 we obtain the forward self-similar solutions (or backward self-similar solutions, resp.) to the system \eqref{NS} having the form
 $$
\uuu(\xx,t)= \lambda^{2\alpha-1}(t)\uuu\big(\lambda(t)\xx,1\big)\quad\text{and}\ \ \  \pi(\xx,t)=\lambda^{2(2\alpha-1)}(t)\pi\big(\lambda(t)\xx,1\big).
$$
Denoting $\uu(\xx)=\uuu(\xx,1)$ and $P(\xx)=\pi(\xx,1)$, one easily verifies that the profile pair $(\uu,P)$  solves either the generalized Leray equations in $\RR^3$
\begin{align}\label{E-forwad}
\left\{
\begin{aligned}
&(-\Delta)^{\alpha}\uu- \frac{2\alpha-1}{2\alpha}\uu-\frac{1}{2\alpha}\xx\cdot \nabla \uu+\nabla P+\uu\cdot\nabla \uu=0, \\
&\operatorname{div}\, \uu=0,\\
\end{aligned}\ \right.
\end{align}
or
\begin{align}\label{E-back}
\left\{
\begin{aligned}
&(-\Delta)^{\alpha}\uu+\frac{2\alpha-1}{2\alpha}\uu+\frac{1}{2\alpha}\xx\cdot \nabla \uu+\nabla P+\uu\cdot\nabla \uu=0, \\
&\operatorname{div}\, \uu=0,\\
\end{aligned}\ \right.
\end{align}
Existence results for the backward self-similar solutions to the system \eqref{E-back}  were obtained in, see, e.g.,  \cite{JMV,Tsai-book}. In this paper, we are interested in the forward self-similar solutions to the problem \eqref{E-forwad}. Due to the singularity arising from  self-similarity, we cannot directly construct a solution   in the Sobolev space. Thus, one decomposes
$$
\uu=\uu_{0}+\vv
$$
with $\uu_0(\xx)=e^{-(-\Delta)^\alpha}\uuu_0(\xx)$ and  the difference $\vv$  satisfies in $\RR^{3}$
\begin{align}\label{E}
\left\{
\begin{aligned}
&(-\Delta)^{\alpha}\vv- \frac{2\alpha-1}{2\alpha}\vv-\frac{1}{2\alpha}\xx\cdot \nabla \vv+\nabla P=-\vv\cdot\nabla \vv +L_{\uu_0}(\vv)+\operatorname{div}\ff_0, \\
&\operatorname{div}\,\vv=0,
\end{aligned}\ \right.
\end{align}
where $\ff_0=-\uu_{0}\otimes \uu_{0}$ and
\[L_{\uu_0}\vv=- \uu_{0} \cdot\nabla \vv-\vv\cdot\nabla \uu_{0}.\]
For such system with $\alpha=1$, existence of weak solutions $\vv\in H^1(\RR^3)$  was firstly shown by  Korobkov and Tsai in \cite{KB}. But they said they did not  prove regularity and decay estimate for the weak solutions.
This problem was later solved in the following theorem by Lai, Miao and Zheng in \cite{LMZ19}.
\begin{theorem}[Theorem 3.6, \cite{LMZ19}]\label{thm-1}
Let $\alpha\in(\frac58,1]$ and $\uuu_{0}=\frac{\sigma(\xx)}{|\xx|^{2\alpha-1}} $ with $\sigma(\xx)=\sigma(\xx/|\xx|)\in C^{1,0}(\mathbb{S}^{2})$.	Then  there exists $C=C\big(\|\sigma\|_{W^{1,\infty}(\mathbb{S}^2)}\big)>0$ such that the system \eqref{E-2} admits at least one weak  solution $\vv\in H^{\alpha}(\RR^{3})$  with
$$
\|\vv\|_{H^{\alpha}(\RR^{3})}\leq C\big(\|\sigma\|_{W^{1,\infty}(\mathbb{S}^2)}\big).
$$
In particular,   for $\alpha=1$,
\begin{equation*}
|\vv|(\xx)\leq C\big(\|\sigma\|_{W^{1,\infty}(\mathbb{S}^2)}\big)(1+|\xx|)^{-3}\log(e+|\xx|).
\end{equation*}
\end{theorem}
In Theorem \ref{thm-1}, it is also proved that the existence of weak solutions $\vv\in H^{\alpha}(\RR^{3})$  to the system  \eqref{E-2} with $\alpha\in(\frac58,1)$ by using the blowup argument and Leray's degree theory. Recently, we further develop some estimates in the weighted framework, and then we obtain the regularity and decay estimate of weak solutions for the case where $\alpha\in(\frac56,1)$, but it is not optimal.
\begin{theorem}[\cite{LMZ21}]\label{thm1.2}
	    Let $\alpha\in(\frac56,1]$ and $\uuu_{0}=\frac{\sigma(\xx)}{|\xx|^{2\alpha-1}} $ with $\sigma(\xx)=\sigma(\xx/|\xx|)\in C^{2,0}(\mathbb{S}^{2})$. Assume that  $\vv\in H^{\alpha}(\RR^{3})$ is the weak solution  established in Theorem~\ref{thm-1}.
		Then we have that $\vv\in H^{1+\alpha} (\RR^3)$ satisfies
		\begin{equation*}
	\big	\|\sqrt{1+|\cdot|}\vv\big\|_{H^{1+\alpha}(\RR^{3})}\leq C\big(\|\sigma\|_{W^{1,\infty}(\mathbb{S}^2)}\big).
		\end{equation*}
Moreover, one has
 $$|\vv(\xx)|\leq  C\big(\|\sigma\|_{W^{1,\infty}(\mathbb{S}^2)}\big)(1+|\xx|)^{-\frac{1}{2}}.$$
 \end{theorem}
\begin{remark}
Now we try to explain why $ \alpha>5/6$ by the scaling analysis, namely, if $\vv$ is the solution to the system \eqref{E}, then so does $\vv_\lambda$ for each $\lambda>0,$ where
\[\vv_{\lambda}(\xx)=\lambda^{2 \alpha-1} \vv(\lambda \xx).\]
A direct calculation yields
\[\left\|\vv_{\lambda}\right\|_{\dot{H}^{\alpha}\left(\mathbf{R}^{3}\right)}=\lambda^{2 \alpha-1+\alpha-\frac{3}{2}}\|\vv\|_{\dot{H}^{\alpha}\left(\mathbf{R}^{3}\right)}=\lambda^{3 \alpha-\frac{5}{2}}\|\vv\|_{\dot{H}^{\alpha}\left(\mathbf{R}^{3}\right)},\]
which implies $\alpha=\frac56$ is ``the critical index". Thus the problem \eqref{E} is classified as ``critical" in the fractional powers of Laplace operator $\alpha=\frac56$.
\end{remark}
Therefore, the goal of the present paper is to study the global regularity and decay behaviour  of  weak solution to the   generalized Leray problem with critical dissipation. To do this,  it makes us think of  a very similar problem, the stationary solutions to the four-dimensional Navier-Stokes equations in $\Omega\subset\RR^4$
\begin{align}\label{E-NS}
\left\{
\begin{aligned}
 & -\Delta \vv+ \vv\cdot \nabla \vv+\nabla P=  \ff,  \\
& \operatorname{div}\,\vv=0.
\end{aligned}\right.
\end{align}
In 1979,  Gerhardt \cite{ref-Ger} proved the regularity of weak solutions to the system \eqref{E-NS} and his proof strongly relies on two ingredients.
 One is  the so-called ``compactness lemma" that if $\uu\in L^2(\Omega)$ and $\vv\in H^{1}(\Omega)$, we have
\[\int_{\Omega}|\uu| \cdot|\vv|^{2} \,\mathrm{d} \xx \leq \varepsilon  \|\vv\|_{H^{1 }(\Omega)}^{2}+C_{\varepsilon}  \|\vv\|_{L^2(\Omega)}^{2}\quad\text{for each}\,\,\varepsilon>0.\]
Another is $L^p$-theory of Stokes equations, which is the linear part of the model \eqref{E-NS}, that if $\vv$ satisfies the Dirichlet problem of Stokes equations
\begin{align}\label{E-NS-stokes}
\left\{
\begin{aligned}
  -\Delta \vv+\nabla P=&   \ff, \quad\xx\in\Omega, \\
  \vv=&0,\quad\,\xx\in\partial\Omega,\\
 \operatorname{div}\,\vv=& 0, \quad \,\xx\in\Omega,
\end{aligned}\ \right.
\end{align}
we have by using bounds for   Calder\'on-Zygmund integral singular operators \cite{[Nirenberg]} that
\begin{equation}\label{E-w2p}
\|\vv\|_{W^{2,p}(\Omega)}+\|P\|_{W^{1,p}(\Omega)}\leq C(p,\Omega)\|\ff\|_{L^p(\Omega)}\quad\text{for each}\,\,p\in(1,+\infty).
\end{equation}
Based on the mathematical similarity between \eqref{E-forwad-leray} and \eqref{E-NS-stokes}, we would like to borrow this idea to apply to the weak solutions to the Leray problem \eqref{E-forwad-leray} to get  global  regularity, but it does not work. The first difficulty we meet is that $L^p$ elliptic regularity is no longer available for solution of linearization
\begin{align}\label{E-forwad-leray-linear}
\left\{
\begin{aligned}
&(-\Delta)^{\alpha}\vv- \frac{2\alpha-1}{2\alpha}\vv-\frac{1}{2\alpha}\xx\cdot \nabla \vv+\nabla P = \ff, \\
&\mbox{div}\, \vv=0,\\
\end{aligned}\ \right.
\end{align}
due to the coefficient $\xx$ is unbounded over whole space $\RR^3$. Indeed, if one estimates the solution $\vv$ to the system \eqref{E-forwad-leray-linear}  directly by making use of \eqref{E-w2p} for each $p\in(1,+\infty),$
 \begin{equation}\label{E-w2p-pertubated}
\|\vv\|_{W^{2\alpha,p}(\RR^3)}+\|P\|_{W^{1,p}(\RR^3)}\leq C(p,\alpha)\left(\|\ff\|_{L^p(\RR^3)}+\|\vv\|_{L^p(\RR^3)}+\|\xx\cdot\nabla\vv\|_{L^p(\RR^3)}\right),
\end{equation}
it is not clear how to control the last term. Hence, to proceed  further, we need to look for new methods to obtain $W^{2\alpha,p}(\RR^3)$-regularity for weak solutions to the problem~\eqref{E-forwad-leray-linear}. Viewed from the perspective of microlocal analysis, it can be  thought of  as the maximal smoothing effect in the $L^p(\RR^3)$ framework. This observation suggests us to consider localization of  solution in frequency space. Performing $\dot{\Delta}_q$ to both sides of the first equation of \eqref{E-forwad-leray-linear}, we easily find that  the couple $(\vv_q,P_q):=(\dot{\Delta}_q\vv,\dot{\Delta}_qP)$ is  smooth and solves the following system in $\RR^3$
\begin{align}\label{E-q-L}
\left\{
\begin{aligned}
 &(-\Delta)^{\alpha}\vv_q-\frac{2-\alpha}{2\alpha}\vv_q-\frac{1}{2\alpha}\xx\cdot \nabla \vv_q+\nabla P_q=\ff_q+\mathbf{R}_q,\\
 &\textnormal{div}\,\vv_q=0,
\end{aligned} \right.
\end{align}
where the commutator  $[\dot{\Delta}_q,\xx\otimes ]\vv$ is defined by
\[[\dot{\Delta}_q,\xx\otimes ]\vv=\dot{\Delta}_q\left(\xx\otimes \vv\right)-\xx\otimes\dot{\Delta}_q \vv.\]
Next, taking the standard $L^p$-estimate of $\vv_q$ with $p\geq2$, one has
\begin{equation}\label{eq-L-Lp-intr}
\begin{split}
&\int_{\mathbf{R}^3}(-\Delta)^{\alpha}\vv_q\,|\vv_q|^{p-2}\vv_q\,\mathrm{d}\xx-\frac{2\alpha -1}{ 2\alpha}\left\|\vv_q\right\|_{L^p(\RR^3)}^p\\&-\frac{1}{2\alpha}\int_{\mathbf{R}^3}\xx\cdot \nabla \vv_q\,|\vv_q|^{p-2}\vv_q\,\mathrm{d}\xx\\
=&\int_{\mathbf{R}^3}\left(-\nabla P_q +\ff_q \right)\,|\vv_q|^{p-2}\vv_q\,\mathrm{d}\xx+ \int_{\mathbf{R}^3}\mathbf{R}_q\,|\vv_q|^{p-2}\vv_q\,\mathrm{d}\xx.
\end{split}
\end{equation}
Formally, we will meet a new difficulty  caused by the commutator, excluding  the difficulty arising from  $\xx\cdot\nabla\vv$. In order to overcome this new difficulty, we see that   commutator can be rewritten as
\begin{align*}
[\dot{\Delta}_q,\xx\otimes]\vv=&\int_{\RR^3}\varphi_q(\xx-\yy)\yy\otimes \vv(\xx-\yy)\,\mathrm{d}\yy-\xx\otimes\int_{\RR^3}\varphi_q(\xx-\yy)\vv(\yy)\,\mathrm{d}\yy\\
=&-\int_{\RR^3}(\xx-\yy)\varphi_q(\xx-\yy)\otimes \vv(\yy)\,\mathrm{d}\yy.
\end{align*}
Taking Fourier transform of commutator and using the following property of Fourier transform which is the key point to overcome the first difficulty
 $$(\xx\cdot f)^\wedge(\xi)=i\operatorname{div}_\xi \hat{f}(\xi),$$ we immediately get
\begin{align*}
\mathscr{F}\left([\dot{\Delta}_q,\xx\otimes]\vv \right)(\xi)
= i2^{-q}\nabla_{\xi} \varphi\left(2^{-q} {\xi}  \right) \otimes\widehat{\vv}(\xi)
=   i2^{-q}\nabla_{\xi} \varphi\left({2^{-q}{\xi}} \right) \otimes\widehat{\widetilde{\dot{\Delta}}_q\vv}(\xi),
\end{align*}
which means that
\begin{equation*}
[\dot{\Delta}_q,\xx\otimes]\vv=2^{-q}\bar{\varphi}_q\ast\left(\widetilde{\dot{\Delta}}_q\vv\right)(\xx)\quad\text{with}\quad \bar{\varphi}(\xx)=\xx \varphi(\xx)\in\mathscr{D}(\RR^3).
\end{equation*}
This equality  enables us to bound the commutator  as follows
\begin{equation}\label{eq-L-Lp-comm-I-intr}
\int_{\mathbf{R}^3}\mathbf{R}_q\,|\vv_q|^{p-2}\vv_q\,\mathrm{d}\xx
\leq \frac{C}{2\alpha} \left\|\widetilde{\dot{\Delta}}_q\vv \right\|_{L^p(\RR^3)}\left\|\vv_q\right\|_{L^p(\RR^3)}^{p-1}.
\end{equation}
Next, we deal with  integral term involving $\xx\cdot\nabla\vv$. Integrating by parts, it can bounded as follows
\begin{align*}
-\frac{1}{2\alpha}\int_{\mathbf{R}^3}\xx\cdot \nabla \vv_q\,|\vv_q|^{p-2}\vv_q\,\mathrm{d}\xx=&-\frac{1}{2p\alpha}\int_{\mathbf{R}^3}\xx\cdot \nabla  \,|\vv_q|^{p} \,\mathrm{d}\xx\\
=& \frac{3}{2p\alpha}\int_{\mathbf{R}^3}   \,|\vv_q|^{p} \,\mathrm{d}\xx.
\end{align*}
It is surely a basic point that we employ  the localization method so as to make this happen. This can help us to make best use of the available structure of $\xx\cdot\nabla \vv$ in the equations \eqref{E-forwad-leray-linear}, not as perturbation  in \eqref{E-w2p-pertubated}.

Inserting the estimate \eqref{eq-L-Lp-comm-I-intr} and  the following lower bound concerning  dissipation  from the Bernstein inequality in Lemma \ref{lem-bern-new}
\begin{equation*}
\int_{\mathbf{R}^3}(-\Delta)^{\alpha}\,|\vv_q|^{p-2}\vv_q\,\mathrm{d}\xx\geq c_p2^{2q\alpha} \left\|\vv_q\right\|^p_{L^p(\RR^3)},\quad\forall p\in[2,+\infty)
\end{equation*}
into the equality \eqref{eq-L-Lp-intr} leads to
\begin{equation}\label{eq-L-Lp-sec-int}
 c_p2^{2q\alpha} \left\|\vv_q\right\|_{L^p(\RR^3)}
\leq \left\|\nabla P_q\right\|_{L^p(\RR^3)}+\left\|\ff_q\right\|_{L^p(\RR^3)}+C \left\|\widetilde{\dot{\Delta}}_q\vv \right\|_{L^p(\RR^3)}.
\end{equation}
Since
\begin{equation*}
-\Delta P_q=\operatorname{div}\ff_q\quad\text{in }\RR^3,
\end{equation*}
we have the classical elliptic estimate that for each $p\in[2,\infty),$
\[\left\|P_q\right\|_{\dot{W}^{2,p}(\RR^3)}\leq C\|\operatorname{div}\ff_q\|_{L^p(\RR^3)}. \]
This estimate together with  the estimate \eqref{eq-L-Lp-sec-int} yields  the $L^p$-type elliptic regularity that for each $p\in[2,+\infty),$
\begin{equation}\label{eq.2ptype}
\left\|\vv\right\|_{\dot{B}^{2\alpha}_{p,p}(\RR^3)}+\left\|P\right\|_{\dot{B}^{1}_{p,p}(\RR^3)}\leq C_{p,\alpha}\left(\left\|\ff\right\|_{\dot{B}^{s }_{p,p}(\RR^3)}+\left\|\vv\right\|_{\dot{B}^{0}_{p,p}(\RR^3)}\right).
\end{equation}
 In addition, with the help of Bony's paraproduct decomposition,   we develop a Besov type compactness lemma: for the divergence free vector field $\uu\in  \dot{H}^{\frac{5}{6}}(\RR^3)$ and  $\vv\in {B}^{\frac53}_{p,p}(\RR^3)$ with $2\leq p<\frac92$, we have
\begin{equation}\label{eq-KY-Fractional-BS-intr}
\|\operatorname{div}\left(\uu\otimes \vv\right)\|_{\dot{B}^{0}_{p,p}(\RR^3)}\leq \varepsilon\|\vv\|_{\dot{B}^{\frac{5}{3}}_{p,p}(\RR^3)}+C_{\varepsilon,\,p,\,\uu}\|\vv\|_{\dot{B}^0_{p,p}(\RR^3)}\quad\text{for each}\,\,\varepsilon>0.
\end{equation}
With both estimates \eqref{eq.2ptype} and \eqref{eq-KY-Fractional-BS-intr} in hand,  we establish our first   results stated as follows:
\begin{theorem}\label{thm-leray}
Let $\alpha\in(0,1]$ and $\ff\in \dot{H}^{-1}(\RR^3)$.
Then the problem \eqref{E-forwad-leray}  admits at least one   weak solution $(\uu,P)\in H^\alpha(\RR^3)\times L^{\frac{3}{3\alpha-2}}(\RR^3)$ such that
\begin{itemize}
  \item [(i)] for each the divergence-free vector field $ {\phi}\in \mathscr{D}(\RR^3)$,
   \begin{equation*}
\begin{split}
&\int_{\RR^3}\Lambda^\alpha \uu\cdot\Lambda^\alpha \phi\,\mathrm{d}\yy+\frac{2-\alpha}{\alpha}\int_{\RR^3}\uu\cdot\phi\,\mathrm{d}\yy+\frac{1}{2\alpha}\int_{\mathbf{R}^3}\yy\cdot \nabla\varphi\cdot \uu\,\mathrm{d}\yy\\=&\int_{\RR^3}P\operatorname{div}\phi\,\mathrm{d}\yy+\int_{\RR^3}\ff:\nabla\phi\,\mathrm{d}\yy;
\end{split}
\end{equation*}
  \item [(ii)]  when $p\in[5/6,1],$ we have $(\uu,P)\in H^{2\alpha}(\RR^3)\times H^{1 }(\RR^3)$. Moreover,\medskip

  \begin{enumerate}
    \item[\rm (1)]   if  $\ff\in \dot{B}_{p,p}^{s-1}(\RR^3)$ with $s\geq0$ and $p\in(2,+\infty),$ we have $$(\uu,P)\in B_{p,p}^{s+2\alpha}(\RR^3)\times B_{p,p}^{s+1}(\RR^3).$$

    \item[\rm (2)]  if $\ff\in H _{\langle \xx\rangle^{2\beta}}^{s+1}(\RR^3)$ with $s\geq0$ and $\beta\in(0,\alpha)$,
  we have $$(\uu,P)\in H^{s+\alpha}_{\langle \xx\rangle^{2\beta}}(\RR^3)\times H_{\langle \xx\rangle^{2\beta}}^{s+\alpha}(\RR^3).$$
   \item[\rm (3)]  if $\displaystyle
\sup_{\xx\in\RR^3}\left|| \xx|^{4\alpha-2}\mathbf{F}(\xx)\right|+\sup_{\xx\in\RR^3}\left||\xx|^{4\alpha-1}\nabla \mathbf{F}(\xx)\right|<+\infty,
$ we  have
\begin{equation}\label{eq.decay-loss}
|\uu|(\xx)\leq C\langle \xx\rangle^{-(4\alpha-1)}\log\langle \xx\rangle\ \ \text{for each}\ \  \xx\in\RR^{3}.
\end{equation}
Moreover, if there exists  some  $\gamma\in(0,1)$ such that
\begin{equation*}
 \left\||\cdot|^{4\alpha-1+\gamma}\operatorname{div}\ff\right\|_{\dot{C}^\gamma (\RR^3)}<\infty,
 \end{equation*}
  we  have
\begin{equation}\label{eq.decay-loss-no}
|\uu|(\xx)\leq C\langle \xx\rangle^{-(4\alpha-1)} \ \ \text{for each}\ \  \xx\in\RR^{3}.
\end{equation}
  \end{enumerate}
\end{itemize}
\end{theorem}
\begin{remark}
The existence of weak solutions can be obtained by performing the same argument such as used in \cite{LMZ21,Tsai-book}, so we omit it and mainly focus on the global regularity and the optimal decay estimate through our paper.
\end{remark}
\begin{remark}
In the proof of  the decay estimate \eqref{eq.decay-loss}, we firstly find that the solution $\uu$ can be expressed by
\begin{equation}\label{eq.fundamental-solution-l-intr}
\uu(\xx)= \int_{0}^{1}e^{-(-\Delta)^{\alpha}(1-s)}\mathbb{P}{\rm div}_{\xx}\left(s^{\frac{1}{\alpha}-2}\mathbf{G}\big(\cdot/s^{\frac{1}{2\alpha}}\big)\right)\,{\rm d}s,
\end{equation}
where $\mathbf{G}=\uu\otimes\uu+\mathbf{F}$ and $\mathbb{P}$ is the Leray projector. Secondly, we develop the action of the heat semigroup generated by the fractional powers of Laplace operator on distributions with Fourier transforms supported in an annulus
\begin{equation} \label{eq-edecay}
\left\|\big(|\cdot|^m D^\gamma \dot{\Delta}_q\mathbb{P}G_\alpha\big)\ast f\right\|_{L^p(\RR^3)}\leq C2^{q(|\gamma|-m)}e^{-ct2^{2q\alpha}}\big\| f \big\|_{L^p(\RR^3)}.
\end{equation}
\end{remark}
\begin{remark}
Without the additional assumptions, the  logarithmic loss in  decay estimate~\eqref{eq.decay-loss} is caused by   Leray projector~$\mathbb{P},$ which seems inevitable.
\end{remark}
Similar to the problem \eqref{E-forwad-leray}, we can show the global regularity and decay behaviour for weak solutions to the problem \eqref{E}   following from Theorem \ref{thm-leray} that
\begin{theorem}\label{thm-leray-per}
 Let $\alpha\in[\frac56,1]$ and $\uuu_{0}(\xx)=\frac{\sigma(\xx)}{|\xx|^{2\alpha-1}} $ with $\sigma(\xx)=\sigma(\xx/|\xx|)\in C^{1,0}(\mathbb{S}^{2}).$ Assume that   $(\vv,P)\in H^\alpha(\RR^3)\times L^{\frac{3}{3-2\alpha}}(\RR^3) $ is a weak solution to the problem  \eqref{E} constructed in Theorem \ref{thm-1}. Then we have that for each $s\geq0 $ and $p\in(2,+\infty),$
$$(\vv,P)\in \left(B_{p,p}^{s}(\RR^3)\times B_{p,p}^{s}(\RR^3)\right)\bigcap\big(H^{s}_{\langle \xx\rangle^{2\beta}}(\RR^3)\times H_{\langle \xx\rangle^{2\beta}}^{s}(\RR^3)\big).$$
More importantly, we have the optimal decay estimate
\begin{equation}\label{eq.decay-noloss}
|\vv|(\xx)\leq C\langle \xx\rangle^{-(4\alpha-1)} \ \ \,\, \text{for each}\ \ \xx\in\RR^{3}.
\end{equation}
 \end{theorem}
\begin{remark}
Compared with the force in  \eqref{E-forwad-leray},  the problem \eqref{E}  has the very special force $\operatorname{div}(\uu_{0}\otimes \uu_{0})$ with $\uu_0(\xx)=e^{-(-\Delta)^\alpha}\uuu_0(\xx)$.  This symmetric structure and the action of the heat semigroup \eqref{eq-edecay} can help us to obtain the optimal decay estimate \eqref{eq.decay-noloss} without logarithmic loss.
\end{remark}
 Since the forward self-similar solution $\uuu(\xx,t)$ takes the form
\[\uuu(\xx, t) = {{t}^{-\frac{2\alpha-1}{2\alpha}}}\uu\big(\xx/ {t}^{\frac{1}{2\alpha}}\big)={{t}^{-\frac{2\alpha-1}{2\alpha}}}\left( \vv +e^{\Delta}\uuu_0\right)\big(\xx/ {t}^{\frac{1}{2\alpha}}\big),\]
we immediately get by using Theorem \ref{thm-leray-per} that
\begin{corollary}\label{coro-critical}
 Let $\alpha\in[\frac56,1]$ and $\uuu_{0}(\xx)=\frac{\sigma(\xx)}{|\xx|^{2\alpha-1}} $ with $\sigma(\xx)=\sigma(\xx/|\xx|)\in C^{1,0}(\mathbb{S}^{2}).$ Assume that  the  $\uuu\in \mathrm{BC}_{\rm w}\big ([0,+\infty),\,L^{ \frac{3}{2\alpha-1},\infty }(\RR^3)\big)$ is the forward self-similar solution to Cauchy problem \eqref{NS}-\eqref{NS-i} established in \cite{LMZ19}. Then
\begin{itemize}
\item [(1)]  For $\forall p\in  [2,+\infty)$ and   $\forall s\in[0,+\infty)$, we have that $$\uuu(\xx,t)-e^{-t(-\Delta)^{\alpha}}\uuu_{0}(\xx)\in B^{s }_{p,p}(\RR^3)$$  and   $\forall\,k\in\NN,$
$$\sup_{|\alpha|=k}\big\|D^\alpha\left(\uuu(t)-e^{-t(-\Delta)^{\alpha}}\uuu_{0}\right)\big\|_{L^{\infty}(\RR^{3})}\leq Ct^{\frac{2\alpha+k-1}{2\alpha}},$$
which implies that the $\uuu(\xx,t)$ is smooth in $\RR^{3}\times (0,+\infty);$
\item [(2)]    we have the following pointwise estimates
\begin{equation}\label{E1.6-new}
|\uuu(\xx,t)|\leq\frac{C}{|\xx|^{2\alpha-1}+ {t}^{\frac{1}{2\alpha}}}\quad\text{and}\quad \left|\uuu(\xx,t)-e^{t\Delta}\uuu_{0}\right|\leq \frac{C }{|\xx|^{4\alpha-1}+t^{\frac{4\alpha-1}{2\alpha}}}
\end{equation}
for all $(\xx,t)\in \RR^{3}\times(0,+\infty)$.
\end{itemize}
 \end{corollary}
 \begin{remark}
 \begin{itemize}
   \item[(i)] For the subcritical case $\alpha\in(5/6,1]$, we here establish the global regularity $$\uuu(\xx,t)-e^{-t(-\Delta)^{\alpha}}\uuu_{0}(\xx)\in B^{s}_{p,p}(\RR^3)$$ for each $s\geq0$ and $p\geq2.$ This improve the known global regularity $$\uuu(\xx,t)-e^{-t(-\Delta)^{\alpha}}\uuu_{0}(\xx)\in H^{1+\alpha} (\RR^3)$$ shown in \cite{LMZ19,LMZ21}.
   \item[(ii)]  By developing the $L^p$-type elliptic regularity and ``compactness lemma", we firstly prove global regularity and decay behaviour of the forward self-similar solution for the critical problem where $\alpha=\frac56$.
   \item [(iii)]Let us point out that we establish the optimal  decay estimate \eqref{eq.decay-noloss} under the natural   assumption $\sigma(\xx)\in C^{1,0}(\mathbb{S}^{2})$, and we drop   the additional condition  in the existing literature, for example, that $\sigma(\xx)\in C^{1,\gamma}(\mathbb{S}^{2})$ with $\gamma>0$  in \cite{ZP22}.
 \end{itemize}
 \end{remark}
 \noindent\textbf{Notations.}
We first agree that   $\langle \xx\rangle =(e+|\xx|^2)^{\frac12}$, $ {B}_r(\xx_0)=\{\xx\in\RR^3|\,|\xx-\xx_0|<r_2\}$ and
\[\mathscr{C}(\xx_0;r_1,r_2)=\{\xx\in\RR^3|\,r_1<|\xx-\xx_0|<r_2\}.\]
For $s\in\RR,$ we denote $s=[s]+\{s\}$, where $[s] $ is the integer   part and $\{s\}$ is the decimal part. Let $\mathscr{D}\left(\mathbf{R}^{n}\right),\,\mathscr{S}\left(\mathbf{R}^{n}\right)$  and $\mathscr{S}^\prime\left(\mathbf{R}^{n}\right)$ be the space of smooth functions with compact support, the Schwartz space of rapidly decreasing smooth functions and the space of tempered distributions on $\RR^n$ respectively.

The rest of our paper is organized as follows: In the next section, we give some preliminary
lemmas and $B^{2\alpha}_{p,p}(\RR^3)$-regularity of linearization to  the Leray problem which is the key point in our proof. In Section \ref{HighR}, we are devoted to proving the global regularity of weak solutions to the Leray problem.
 Finally, the optimal decay estimates of  weak solutions to Leray problem and the forward self-similar solutions to the Navier-Stokes equations are established in Section \ref{DECAY}.

%%%%%%%%%%%%%%%%%%%%%%%%%%%%%%%%%%%%%%%%%%%%%%%%%%%%%%%%%%%%%%%%%%%%%%%%%%%%%%%%%%%%%%%%%%%%%%%%%%%
\section{Preliminary and $L^p$-theory  of linearization}\label{PRE}
%%%%%%%%%%%%%%%%%%%%%%%%%%%%%%%%%%%%%%%%%%%%%%%%%%%%%%%%%%%%%%%%%%%%%%%%%%%%%%%%%%%%%%%%%%%%%%
\setcounter{section}{2}\setcounter{equation}{0}
%%%%%%%%%%%%%%%%%%%%%%%%%%%%%%%%%%%%%%%%%%%%%%%%%%%%%%%%%%%%%%%%%%%%%%%%%%%%%%%%%%%%%%%%%%%%%%
\subsection{Littlewood-Paley theory}
In this subsection, we will review some analysis statement including the Lttilewood-Paley theory, see for example \cite{BCD,BL76,CWZ-book}. First of all, let us   recall Fourier transform that for each $f\in \mathscr{S}^{\prime}(\mathbf{R}^n)$,
$$
\hat{f}(\xi):=\mathscr{F}(f)(\xi)=\frac{1}{(2 \pi)^{n / 2}} \int_{\mathbf{R}^{n}} e^{-i \xx \cdot \xi} f(\xx)\,\mathrm{d} \xx, \quad\forall\xi \in \mathbf{R}^{n},
$$
and its inverse Fourier transform
$$
\check{f}(\xx):=\mathscr{F}^{-1}(f)(\xx)=\frac{1}{(2 \pi)^{n / 2}} \int_{\mathbf{R}^{n}} e^{i \xx \cdot \xi} f(\xi) \,\mathrm{d} \xi,\quad \forall \xx \in \mathbf{R}^{n}.
$$
 There exist two smooth radial functions   $\hat{h}(\xi) \in \mathscr{D}\big(B_{\frac{4}{3}}(0)\big) $ and $ \hat{\varphi}(\xi) \in \mathscr{D}(\mathscr{C})$  with $\mathscr{C}=\mathscr{C}(0;\frac34,\frac{8}{3})$ such that $0 \leqslant \hat{h}(\xi),\, \hat{\varphi}(\xi) \leqslant 1$ and
\begin{equation*}\label{eq-1-1}
\hat{h}(\xi)+\sum_{j \geqslant 0} \hat{\varphi}\left(2^{-j} \xi\right)=1, \quad \forall \xi \in \mathbf{R}^{n},
\end{equation*}
\begin{equation*}\label{eq-1-2}
 \sum_{j \in \mathbb{Z}} \hat{\varphi}\left(2^{-j} \xi\right)=1, \quad \forall \xi \in \mathbf{R}^{n} \backslash\{0\}.
\end{equation*}
Under the framework of dyadic partition of unity in  frequency-space domain, we introduce some notations concerning the Littlewood-Paley theory.
\begin{itemize}
  \item  Homogeneous decomposition: the dyadic blocks
  \[\dot{\Delta}_{j} f(\xx)= 2^{j n} \int_{\mathbf{R}^{n}} \varphi\left( 2^j y\right) f(\xx-\yy) \,\mathrm{d} \yy, \quad \widetilde{\dot{\Delta}}_j=\sum_{i=-1}^1\dot{\Delta}_{j+i}, \quad\forall j \in \mathbb{Z} \]
  and the low-frequency cut-off operators
   \[\dot{S}_{j} f(\xx) =2^{j n} \int_{\mathbf{R}^{n}} h\left(2^{j} \yy\right) f(\xx-\yy) \,\mathrm{d} \yy, \quad \forall j \in \mathbb{Z}.\]
  \item  Inhomogeneous decomposition:
  \[\Delta_{-1} f(\xx)=\dot{S}_{0} f(\xx)=  \int_{\mathbf{R}^{n}} h\left(  \yy\right) f(\xx-\yy) \,\mathrm{d} \yy, \]
 \[\Delta_{j} f(\xx)=\dot{\Delta}_{j} f(\xx)\quad\text{and}\quad S_{j} f(\xx)=\dot{S}_{j}f(\xx), \quad\forall j \geqslant 0.\]

\end{itemize}
Let us remark that for each $f \in \mathscr{S}^{\prime}\left(\mathbf{R}^{n}\right), $
$$
S_{j} f(\xx)=\sum_{j^{\prime} \leqslant j-1} \Delta_{j^{\prime}} f(\xx), \quad \forall j \in \mathbb{Z},
$$
but for each polynomial $f \in \mathscr{S}^{\prime}\left(\mathbf{R}^{n}\right) $ whose compact support of $\hat{f}(\xi)$ in $\{0\}$,
$$
  \dot{S}_{j} f(\xx) \neq \sum_{j^{\prime} \leqslant j-1} \dot{\Delta}_{j^{\prime}} f(\xx), \quad \forall j \in \mathbb{Z}.
$$
In order to remedy this defect, we introduce a new space of tempered distributions  $$\mathscr{S}_{h}^{\prime}\left(\mathbf{R}^{n}\right):=\mathscr{S}^{\prime}\left(\mathbf{R}^{n}\right)\backslash\mathscr{P},$$ which entails  that for each $f \in \mathscr{S}_h^{\prime}\left(\mathbf{R}^{n}\right) $, we have
$$
\dot{S}_{j} f(\xx)=\sum_{j^{\prime} \leqslant j-1} \dot{\Delta}_{j^{\prime}} f(\xx), \quad \forall j \in \mathbb{Z}.
$$
Therefore, we can write the following   Littlewood-Paley decompositions:
$$
I_d=\sum_{j\geq-1} \Delta_{j} \quad \text { and } \quad I_d=\sum_{j\in\ZZ} \dot{\Delta}_{j},
$$
which make sense in $\mathscr{S}^{\prime}\left(\mathbf{R}^{n}\right)$ and $\mathscr{S}_h^{\prime}\left(\mathbf{R}^{n}\right)$, respectively.  Also, we note that the above dyadic block and low-frequency cut-off operator map $L^{p}\left(\mathbf{R}^{n}\right)$ into $L^{p}\left(\mathbf{R}^{n}\right)$ with norms independent of $j$ and $p$, which is contained in the following Bernstein lemma.
\begin{lemma}[\cite{BCD,CWZ-book}, Bernstein inequality]\label{lem-bern} Let $1 \leq p \leq q\leq\infty$ and $ f \in L^{p}\left(\mathbf{R}^{n}\right)$.
Then there exists a constant $C$ such that for all $ j\in\ZZ $ and $ k \in \mathbb{N}$,
\begin{equation*}\label{eq-bern-1}
\sup _{|\alpha|=k}\big\|D^{\alpha} \dot{S}_{j} f\big\|_{L^{q}\left(\mathbf{R}^{n}\right)} \leq C^{k} 2^{j\left(k+n\left(\frac{1}{p}-\frac{1}{q}\right)\right)}\big\|\dot{S}_{j} f\big\|_{L^{p}\left(\mathbf{R}^{n}\right)},
\end{equation*}
\begin{equation*}\label{eq-bern-2}
C^{-k} 2^{ jk}\big\|\dot{\Delta}_{j} f\big\|_{L^{q}\left(\mathbf{R}^{n}\right)} \leq \sup _{ |\alpha|=k}\big\|D^{\alpha} \dot{\Delta}_{j} f\big\|_{L^{q}\left(\mathbf{R}^{n}\right)} \leq C^{k} 2^{j\left(k+n\left(\frac{1}{p}-\frac{1}{q}\right)\right)}\big\|\dot{\Delta}_{j} f\big\|_{L^{p}\left(\mathbf{R}^{n}\right)}.
\end{equation*}
\end{lemma}
\begin{remark}
The proof of the inequality \eqref{eq-bern-1} and  the right inequality of \eqref{eq-bern-2} heavily relies on the Young inequality, that is, $\forall\, p,q,r\in[1,\infty],$
\begin{equation}\label{eq.Young}
\|f \ast g\|_{L^{p }\left(\mathbf{R}^{n}\right)} \leq C\|f\|_{L^{q }\left(\mathbf{R}^{n}\right)}\|g\|_{L^{r}\left(\mathbf{R}^{n}\right)},
\end{equation}
where $1+\frac{1}{p}=\frac{1}{q}+\frac{1}{r}.$ This inequality could also be applied in the Lorentz spaces firstly developed by R. O'Neil in \cite{Neil}, and a more general Young inequality in the Lorentz spaces such as in \cite{Grafakos}, whose effect is obviously  better than the classical Young inequality \eqref{eq.Young}.
\begin{lemma}[\cite{Grafakos}, Theorem 1.2.13]
 Let $1\leq\ell \leq s \leq \infty,\, 1 \leq r<\infty$ and $1<p, q<\infty$ with
$$
1+\frac{1}{p}=\frac{1}{q}+\frac{1}{r},
$$
If $f \in L^{q, \ell}\left(\mathbf{R}^{n}\right)$ and $g \in L^{r}\left(\mathbf{R}^{n}\right)$,
then $f * g \in L^{p, s}\left(\mathbf{R}^{n}\right)$, and there exists a constant $C$ depending only on $r, q, s$ and $\ell$ such that
$$
\|f \ast g\|_{L^{p, s}\left(\mathbf{R}^{n}\right)} \leq C\|f\|_{L^{q, \ell}\left(\mathbf{R}^{n}\right)}\|g\|_{L^{r}\left(\mathbf{R}^{n}\right)}.
$$
\end{lemma}
Applying the above generalized Young inequality instead of \eqref{eq.Young} in the proof of the Bernstein inequality, we would have the generalized Bernstein inequality: Assume $1<p,q<\infty$ and $p<q$, then
\begin{equation} \label{eq.bern-g-infty}
\sup _{|\alpha|=k}\big\|D^{\alpha}\dot{\Delta}_j f\big\|_{L^{\infty}\left(\mathbf{R}^{n}\right)} \leq C 2^{j(k+\frac{n}{p})}\|\dot{\Delta}_jf\|_{L^{p, \infty}\left(\mathbf{R}^{n}\right)},
\end{equation}
\begin{equation} \label{eq.bern-g-non-infty}
\sup _{|\alpha|=k}\big\|D^{\alpha}\dot{\Delta}_j f\big\|_{L^{q, 1}\left(\mathbf{R}^{n}\right)} \leq C 2^{j\left(k+n\left(\frac{1}{p}-\frac{1}{q}\right)\right)}\|\dot{\Delta}_jf\|_{L^{p, \infty}\left(\mathbf{R}^{n}\right)},
\end{equation}
so does the low-frequency cut-off operator $\dot{S}_j.$

Next, we define the homogeneous Besov spaces in terms of the Bernstein inequality.
\end{remark}
\begin{definition} [Homogeneous  Besov space] Assume $s \in \mathbf{R}, 1 \leqslant p, r \leqslant \infty$, one defines
$$
\dot{B}_{p, r}^{s}(\RR^n)=\left\{f \in \mathscr{S}_{h}^{\prime}\left(\mathbf{R}^{n}\right)\, \big|\,\|f\|_{\dot{B}_{p, r}^{s}(\RR^n)}<\infty\right\},
$$
where
\[\|f\|_{\dot{B}_{p, r}^{s}\left(\mathbf{R}^{n}\right)} := \left\|\left(2^{j s}\big\|\dot{\Delta}_{j}f\big\|_{p}\right)_{j \in \mathbb{Z}}\right\|_{\ell^{r}(\ZZ)}.\]
\end{definition}
Similar to the standard function spaces (e.g., Sobolev spaces $H^s(\RR^n)$ or $L^p(\RR^n)$ spaces), the homogeneous Besov spaces have  some properties illustrated by the following proposition.
\begin{proposition}
\begin{itemize}
  \item[(i)] For each $p \geqslant 1$,   $ \dot{B}_{p, r}^{s}(\RR^n) $ with $s<\frac{n}{p}$ or  $\dot{B}_{p, 1}^{ {n}/{p}}(\RR^n)$ is   Banach space.
  \item[(ii)] Inclusion: For each  $0\leq r_{1} \leqslant r_{2}\leq\infty$, \[\|f\|_{\dot{B}_{p, r_2}^{s}(\RR^n)} \leqslant C\|f\|_{\dot{B}_{p, r_{1}}^{s}(\RR^n)}.\]
  \item [(iii)] Embedding: For each $1\leq p_{1} \leqslant p_{2}\leq\infty,$
   $$
\|f\|_{\dot{B}_{p_{2}, r}^{s}(\RR^n)} \leqslant C\|f\|_{ \dot{B}_{p_{1}, r}^{s+n\left(\frac{1}{p_{1}}-\frac{1}{p_{2}}\right)}(\RR^n)}.
$$
\item [(iv)]  The  interpolation inequality: If $s_{2}\geq s_{1}$, there exists $\theta\in[0,1]$ such that
$$
\|f\|_{\dot{B}_{p, r}^{\theta s_{1}+(1-\theta) s_{2}}(\RR^n)} \leqslant C\|f\|_{\dot{B}_{p, r}^{s_{1}}(\RR^n)}^{\theta}\|f\|_{\dot{B}_{p, r}^{s_{2}}(\RR^n)  \quad}^{1-\theta}.
$$
\item[(v)] The sharp interpolation inequality: If $s_{2}> s_{1}$, there exists $\theta\in(0,1)$ such that
$$
\|f\|_{\dot{B}_{p, 1}^{\theta s_{1}+(1-\theta) s_{2}}(\RR^n)} \leqslant \frac{C}{s_{2}-s_{1}}\left(\frac{1}{\theta}+\frac{1}{1- {\theta}}\right)\|f\|_{\dot{B}_{p, \infty}^{s_{1}}(\RR^n)}^{\theta}\|f\|_{\dot{B}_{p, \infty}^{s_{2}}(\RR^n)}^{1-\theta}.
$$
\item [(vi)]For $s\in\RR,$ $\dot{H}^{s}(\RR^n)=\dot{B}_{2,2}^{s}(\RR^n),$  and their norms are equivalent
$$
\frac{1}{C^{|s|+1}}\|f\|_{\dot{B}_{2,2}^{s}(\RR^n)} \leqslant\|f\|_{\dot{H}^{s}(\RR^n)} \leqslant C^{|s|+1}\|f\|_{\dot{B}_{2,2}^{s}(\RR^n)}.
$$
\end{itemize}
\end{proposition}
Next, we will give the definition of the nonhomogeneous Besov spaces, which enjoys most properties which have been proven   for the homogeneous Besov spaces.
\begin{definition}[Inhomogeneous Besov space]
Let $s \in \mathbf{R}, \,p, \,r \geqslant 1$, the inhomogeneous Besov space $B_{p, r}^{s}\left(\mathbf{R}^{n}\right)$ defined as follows
$$
B_{p, r}^{s}\left(\mathbf{R}^{n}\right)=\left\{f(\xx) \mid f \in \mathscr{S}^{\prime}\left(\mathbf{R}^{n}\right) \text{ satisfies }\|f\|_{B_{p, r}^{s}\left(\mathbf{R}^{n}\right)} <\infty\right\},
$$
where
\[\|f\|_{B_{p, r}^{s}\left(\mathbf{R}^{n}\right)} := \left\|\left(2^{j s}\left\|\Delta_{j}f\right\|_{p}\right)_{j \geq-1  }\right\|_{\ell^{r}(\mathbb{N}\cup\{-1\})}.\]
\end{definition}
Now we can   state some  useful properties of the nonhomogeneous Besov spaces.
\begin{proposition}
Assume $s \in \mathbf{R}, 1 \leqslant p, r \leqslant \infty$,  then there hold
\begin{enumerate}
   \item [\rm(i)] $ (B_{p, r}^{s}(\mathbf{R}^{n}),\|\cdot\|_{B_{p, r}^{s}(\mathbf{R}^{n})} )$  is the  Banach space.
   \item [\rm(ii)] The interpolation inequality:  If $s_{1}<s_{2}$,  there exists a constant $C>0$ such that for each $\theta \in(0,1), $
   $$
\|f\|_{B_{p, 1}^{\theta s_{1}+(1-\theta) s_{2}}(\mathbf{R}^{n})} \leqslant \frac{C}{\left(s_{2}-s_{1}\right) \theta(1-\theta)}\|f\|_{B_{p, \infty}^{s_{1}}(\mathbf{R}^{n})}^{\theta}\|f\|_{B_{p, \infty}^{s_{2}}(\mathbf{R}^{n})}^{1-\theta}.
$$
   \item[\rm(iii)] Density: For $p, r<\infty,$ $\mathscr{D}(\RR^n)$ is dense  in  $B_{p, r}^{s}\left(\mathbf{R}^{n}\right).$
       \item [\rm(iv)] Duality: For each $ 1 \leqslant p, r<\infty,$ we have $ B_{p^{\prime}, r^{\prime}}^{-s}\left(\mathbf{R}^{n}\right)=\left(B_{p, r}^{s}\right)^{\prime}\left(\mathbf{R}^{n}\right)$.
       \item [\rm(v)] For $\forall s>0$,  we have from \cite[Chapter 6]{BL76} that
       \[B_{p, r}^{s}\left(\mathbf{R}^{n}\right)=L^{p}\left(\mathbf{R}^{n}\right) \cap \dot{B}_{p, r}^{s}\left(\mathbf{R}^{n}\right).\]
               \item [\rm(vi )]  The general inclusions:
$$
\begin{aligned}
&B_{p, r}^{s}(\mathbf{R}^{n}) \hookrightarrow B_{p, r_{1}}^{s_{1}}(\mathbf{R}^{n}), \quad s_{1}<s \,\,\text{ or }\,\, s_{1}=s, \,r_{1} \geq r, \\
&B_{p, r}^{s} (\mathbf{R}^{n})\hookrightarrow B_{p_{1}, r}^{s-n\left(\frac{1}{p}-\frac{1}{p_{1}}\right)}(\mathbf{R}^{n}), \quad p_{1} \geq p, \\
&B_{\infty, 1}^{0}(\mathbf{R}^{n}) \hookrightarrow C_{b}(\mathbf{R}^{n})=C(\mathbf{R}^{n}) \cap L^{\infty}(\mathbf{R}^{n}).
\end{aligned}
$$
 \end{enumerate}
\end{proposition}
After stating the Besov type spaces, we will discuss the simpler H\"older spaces which are defined as follows.
\begin{definition}[H\"older space $C^{k, \gamma}\left(\Omega\right)$]
 Assume $k \in \mathbb{N}, \,\gamma \in(0,1)$ and $\Omega\subset\RR^n$. Let us define  the H\"older space $C^{k, \gamma}\left(\Omega\right)$ as follows:
\[C^{k, \gamma}\left(\Omega\right)=\left\{f(\xx) \in C_{b}^{k}\left(\Omega\right) \mid\|f\|_{C^{k, \gamma}(\Omega )}<\infty\right\}, \]
where
\[\|f\|_{C^{k, \gamma}\left(\Omega\right)}=\|f\|_{C^{k}\left(\Omega\right)}+\sup _{\xx \neq \yy} \sup _{|\alpha|=k} \frac{\left|D^{\alpha} f(\xx)-D^{\alpha} f(\yy)\right|}{|\xx-\yy|^{\gamma}} =:\|f\|_{C^{k}\left(\Omega\right)}+\|f\|_{\dot{C}^{k, \gamma}\left(\Omega\right)}. \]
\end{definition}
\begin{remark}
Let us point out that if $\gamma=r-[r]>0$, we have $$\frac{1}{C}\| f \|_{\dot{B}_{\infty, \infty}^{r}\left(\mathbf{R}^{n}\right)}\leq\| f \|_{\dot{C} ^{[r], \gamma}\left(\mathbf{R}^{n}\right)} \leqslant C_{r} \| f \|_{\dot{B}_{\infty, \infty}^{r}\left(\mathbf{R}^{n}\right)}, $$
where $C_{r}=C_{[r]}\big(\frac{1}{\gamma}-\frac{1}{1-\gamma}\big).$
\end{remark}
Next, we introduce the definition of a weighted Hilbert space.
\begin{definition} Let $s\geq0$  and $\textnormal{w}(\xx)$ be the nonzero weighted function. Then we define the weighted space $H_{\rm w}^{s}\left(\mathbf{R}^{n}\right)$ as
$$
H_{\rm w}^{s}\left(\mathbf{R}^{n}\right) :=\left\{f \in \mathscr{S}^{\prime}\left(\mathbf{R}^{n}\right) \mid\|f\|_{H_{\rm w}^{s}\left(\mathbf{R}^{n}\right)}< \infty\right\},
$$
where
$$
\|f\|_{H_{\rm w}^{s}\left(\mathbf{R}^{n}\right)}:=\left(\int_{\RR^n} |f|^2(\xx) \mathrm{w}(\xx)\,\mathrm{d}\xx\right)^{\frac12}+\left(\int_{\RR^n} |(-\Delta)^{\frac{s}{2}}  f|^2(\xx) \mathrm{w}(\xx)\,\mathrm{d}\xx\right)^{\frac12}.
$$
\end{definition}
Lastly we recall Bony's paraproduct algorithm \cite{Bony}, which is one of the most powerful tools of paradifferential calculus. In terms of the  Littlewood-Paley decomposition,
$$
f=\sum_{j\in\ZZ} \dot{\Delta}_{j} f  \quad\text{and}\quad
 g=\sum_{j\in\ZZ} \dot{\Delta}_{j} g,
$$
so that, we can write formally,
$$
f g=\sum_{j\in\ZZ}\left(\dot{S}_{j+1} f \dot{S}_{j+1} g-\dot{S}_{j} f \dot S_{j} g\right).
$$
After some simplifications, we readily get
$$
\begin{aligned}
f g &=\sum_{j\in\ZZ}\left(\dot{\Delta}_{j} f \dot S_{j} g+\dot{\Delta}_{j} g \dot S_{j} f+\dot{\Delta}_{j} f \dot{\Delta}_{j} g\right) \\
&=\sum_{j\in\ZZ} \dot{\Delta}_{j}f \dot{S}_{j-1} g+\sum_{j\in\ZZ} \dot{\Delta}_{j} g \dot{S}_{j-1} f+\sum_{\left|j-j^{\prime}\right| \leqslant 1}\dot{\Delta}_{j^{\prime}} f \dot{\Delta}_{j} g.
\end{aligned}
$$
As a result, the product of two tempered distributions is decomposed into two paraproducts, respectively, the paraproduct term
$$
\dot{T}_fg=\sum_{j\in\ZZ} \dot{\Delta}_{j} f \dot{S}_{j-1} g
$$
and the reminder term
$$
\dot{R}(f,g)=\sum_{\left|j-j^{\prime}\right| \leqslant 1}\dot{\Delta}_{j^{\prime}} f \dot{\Delta}_{j} g=\sum_{j\in\ZZ}\widetilde{\dot{\Delta}}_{j } f \dot{\Delta}_{j} g.
$$
In  summary, the product of two tempered distributions can be split into three parts as follows
\[fg=\dot{T}_fg+\dot{T}_g f+\dot{R}(f,g).\]
\subsection{Some useful lemmas }
%%%%%%%%%%%%%%%%%%%%%%%%%%%%%%%%%%%%%%%%%%%%%%%%%%%%%%%%%%%%%%%%%%%%%%%%%%%%%%%%%%%%%%%%%%%%%%
In this subsection, we are devoted to showing some useful lemmas which play an important role in our proof. Let us begin with a compactness lemma which is a vital cornerstone  for the critical case where $\alpha=\frac56.$
\begin{lemma}[Compactness Lemma]\label{lem-SYB}
Assume the divergence free vector field $\uu\in  \dot{H}^{\frac{5}{6}}(\RR^3)$ and  $\vv\in {B}^{\frac53}_{p,r}(\RR^3)$ with $2\leq p<\frac92$ and $1\leq r\leq\infty$. Then, for any $\varepsilon>0,$
there exists a constant $C>0$ depending only on $\varepsilon, p$ and $\uu$ such that
\begin{equation}\label{eq-KY-Fractional-BS}
\|\operatorname{div}\left(\uu\otimes \vv\right)\|_{\dot{B}^{0}_{p,r}(\RR^3)}\leq \varepsilon\|\vv\|_{\dot{B}^{\frac{5}{3}}_{p,r}(\RR^3)}+C_{\varepsilon,\,p,\,\uu}\|\vv\|_{\dot{B}^0_{p,r}(\RR^3)}.
\end{equation}
\end{lemma}
\begin{proof}
According to the Bony paraproduct decomposition, we split $\uu\otimes\vv$ into three parts as follows
\[\uu\otimes\vv=\dot{T}_{\uu}\vv+\dot{T}_{\vv}\uu+\dot{R}\big(\uu,\,\vv\big).\]
By the H\"older inequality, we can infer that for $2\leq p<\frac{9}{2},$
\begin{equation}\label{eq-l-1-1}
\left\|
\operatorname{div}\dot{T}_{\uu}\vv\right\|_{\dot{B}^{0}_{p,r}(\RR^3)}\leq C \left\|\|\dot{S}_{k-1}\uu\|_{L^{\infty}(\RR^3)}\|\dot{\Delta}_k \vv\|_{L^{p}(\RR^3)}\right\|_{\ell^r(\ZZ)}.
\end{equation}
Inserting
\[\uu=\sum_{k<N}\dot{\Delta}_k\uu+\sum_{k\geq N}\dot{\Delta}_k\uu=: \uu^\sharp +\uu^\natural,\]
into \eqref{eq-l-1-1} gives
\begin{equation}\label{eq-l-1-1-1-2}
\begin{split}
\left\|
\operatorname{div}\dot{T}_{\uu}\vv\right\|_{\dot{B}^0_{p,r}(\RR^3)}\leq &C \left\| 2^k \|\dot{S}_{k-1}\uu^\sharp\|_{L^{\infty}(\RR^3)}\|\dot{\Delta}_k \vv\|_{L^{p}(\RR^3)}\right\|_{\ell^r(\ZZ)}\\
&+C\left\| 2^k\|\dot{S}_{k-1}\uu^\natural\|_{L^{\infty}(\RR^3)}\|\dot{\Delta}_k \vv\|_{L^{p}(\RR^3)}\right\|_{\ell^r(\ZZ)}.
\end{split}
\end{equation}
On one hand,
\begin{align*}
\left\|2^k \|\dot{S}_{k-1}\uu^\sharp\|_{L^{\infty}(\RR^3)}\|\dot{\Delta}_k \vv\|_{L^{p}(\RR^3)}\right\|_{\ell^r(\ZZ)}\leq &C\|\dot{S}_{N}\uu\|_{L^{\infty}(\RR^3)}\left\|2^k\|\dot{\Delta}_k \vv\|_{L^{p}(\RR^3)}\right\|_{\ell^r(\ZZ)}\\
\leq&C2^{\frac{2N}{3}}\|\uu\|_{\dot{B}^{-\frac23}_{\infty,\infty}(\RR^3)}\|\uu\|_{\dot{B}^1_{p,r}(\RR^3)}.
\end{align*}
By the sharp interpolation inequality, one has
\begin{equation*}\label{eq-sharp-interpolation}
\|\uu\|_{\dot{B}^1_{p,1}(\RR^3)}\leq \|\uu\|^{\frac25}_{\dot{B}^0_{p,\infty}(\RR^3)}\|\uu\|^{\frac35}_{\dot{B}^\frac{5}{3}_{p,\infty}(\RR^3)}.
\end{equation*}
  We furthermore have by using  Young's inequality that
\begin{equation}\label{eq-l-1-1-1-3}
\begin{split}
 &\left\|2^k \|\dot{S}_{k-1}\uu^\sharp\|_{L^{\infty}(\RR^3)}\|\dot{\Delta}_k \vv\|_{L^{p}(\RR^3)}\right\|_{\ell^r(\ZZ)}\\
\leq&C2^{\frac{2N}{3}}\|\uu\|_{\dot{B}^{\frac{5}{6}}_{2,\infty}(\RR^3)}\|\uu\|^{\frac25}_{\dot{B}^0_{p,2}(\RR^3)}
\|\uu\|^{\frac35}_{\dot{B}^\frac{5}{3}_{p,2}(\RR^3)}\\
\leq&C_\varepsilon 2^{\frac{5N}{3}}\|\uu\|^{\frac52}_{\dot{B}^{\frac{5}{6}}_{2,\infty}(\RR^3)}\|\uu\|_{\dot{B}^0_{p,2}(\RR^3)}
+\frac{\varepsilon}{2}
\|\uu\|_{\dot{B}^\frac{5}{3}_{p,2}(\RR^3)}.
\end{split}
\end{equation}
On the other hand,
\begin{equation}\label{eq-l-1-1-1-4}
\begin{split}
\left\| 2^k\|\dot{S}_{k-1}\uu^\natural\|_{L^{\infty}(\RR^3)}\|\dot{\Delta}_k \vv\|_{L^{p}(\RR^3)}\right\|_{\ell^r(\ZZ)}
\leq &C\left\|\uu^\natural\right\|_{\dot{B}^{-\frac23}_{\infty,\infty}(\RR^3)}\|\uu\|_{\dot{B}^{\frac{5}{3}}_{p,r}(\RR^3)}\\
\leq &C\sup_{k\geq N}2^{\frac{5k}{6}}\|\dot{\Delta}_k\uu \|_{L^2(\RR^3)}\|\uu\|_{\dot{B}^{\frac{5}{3}}_{p,r}(\RR^3)}.
\end{split}
\end{equation}
Plugging \eqref{eq-l-1-1-1-3} and \eqref{eq-l-1-1-1-4} into \eqref{eq-l-1-1-1-2}, we obtain
\begin{equation}\label{eq-l-1-1-1-5}
\begin{split}
\left\|
\operatorname{div}\dot{T}_{\uu}\vv\right\|_{\dot{B}^0_{p,r}(\RR^3)}\leq& C\sup_{k\geq N}2^{\frac{5k}{6}}\|\dot{\Delta}_k\uu \|_{L^2(\RR^3)}\|\uu\|_{\dot{B}^{\frac{5}{3}}_{p,r}(\RR^3)}\\
&+C_\varepsilon 2^{\frac{5N}{3}}\|\uu\|^{\frac52}_{\dot{B}^{\frac{5}{6}}_{2,\infty}(\RR^3)}\|\uu\|_{\dot{B}^0_{p,r}(\RR^3)}
+\frac{\varepsilon}{2}
\|\uu\|_{\dot{B}^\frac{5}{3}_{p,r}(\RR^3)}.
\end{split}
\end{equation}
The divergence-free  condition allows us to write
\[\operatorname{div}\dot{T}_{\vv}\uu=\dot{T}_{\nabla\vv}\uu.\]
Moreover, we have by the H\"older inequality that
\begin{equation}\label{eq-l-1-2}
\begin{split}
\left\|\operatorname{div}(\dot{T}_{\vv}\uu)\right\|_{\dot{B}^0_{p,r}(\RR^3)}\leq& C \left\|\|\dot{\Delta}_k \uu^\sharp\|_{L^{p}(\RR^3)}\|\dot{S}_{k-1}\nabla \vv\|_{L^{\infty}(\RR^3)}\right\|_{\ell^r(\ZZ)}\\
&+C \left\|\|\dot{\Delta}_k \uu^\natural\|_{L^{p}(\RR^3)}\|\dot{S}_{k-1}\nabla \vv\|_{L^{\infty}(\RR^3)}\right\|_{\ell^r(\ZZ)}.
\end{split}
\end{equation}
The Bernstein inequality in Lemma \ref{lem-bern} can be applied to tackle the low-frequency part
\begin{equation}\label{eq-l-1-2-1}
\begin{split}
  \left\|\|\dot{\Delta}_k \uu^\sharp\|_{L^{p}(\RR^3)}\|\dot{S}_{k-1}\nabla \vv\|_{L^{\infty}(\RR^3)}\right\|_{\ell^r(\ZZ )}\leq& C\|\uu^\sharp\|_{\dot{B}^{1+\frac{3}{p}}_{p,\infty}(\RR^3)}  \|\nabla\vv\|_{\dot{B}^{-(1+\frac{3}{p})}_{\infty,r}(\RR^3)}  \\
  \leq& C2^{\frac{5N}{3} }\|\uu \|_{\dot{B}^{\frac{5}{6}}_{2,\infty}(\RR^3)}  \| \vv\|_{\dot{B}^{0}_{p,r}(\RR^3)}.
\end{split}
\end{equation}
By the H\"older inequality, we obtain
\begin{equation}\label{eq-l-1-2-2}
\begin{split}
&\left\|\|\dot{\Delta}_k \uu^\natural\|_{L^{p}(\RR^3)}\|\dot{S}_{k-1}\nabla \vv\|_{L^{\infty}(\RR^3)}\right\|_{\ell^r(\ZZ)}\\\leq& C\|\uu^\natural\|_{\dot{B}^{\frac{5}{6}+3\left(\frac{1}{p}-\frac{1}{2}\right)}_{p,\infty}(\RR^3)}  \|\nabla\vv\|_{\dot{B}^{-\frac{5}{6}-3\left(\frac{1}{p}-\frac{1}{2}\right)}_{\infty,r}(\RR^3)}
  \leq  C\sup_{k\geq N}2^{\frac{5k}{6}}\|\dot{\Delta}_k\uu \|_{L^2(\RR^3)}  \| \vv\|_{\dot{B}^{\frac{5}{3}}_{p,r}(\RR^3)}.
\end{split}
\end{equation}
Putting \eqref{eq-l-1-2-1} and \eqref{eq-l-1-2-2} into \eqref{eq-l-1-2}  to get
\begin{equation}\label{eq-l-1-2-3}
\begin{split}
\left\|\operatorname{div}(\dot{T}_{\vv}\uu)\right\|_{\dot{B}^0_{p,r}(\RR^3)}\leq& C\sup_{k\geq N}2^{\frac{5k}{6}}\|\dot{\Delta}_k\uu \|_{L^2(\RR^3)}  \| \vv\|_{\dot{B}^{\frac{5}{3}}_{p,r}(\RR^3)}\\
&+ C2^{\frac{5N}{3} }\|\uu \|_{\dot{B}^{\frac{5}{6}}_{2,\infty}(\RR^3)}  \| \vv\|_{\dot{B}^{0}_{p,r}(\RR^3)}.
\end{split}
\end{equation}
For the remainder term, the Bernstein inequality in Lemma \ref{lem-bern} and the H\"older inequality help us to conclude that
\begin{align*}
&\left\|\operatorname{div}\big(\dot{R}(\uu,\vv)\big)\right\|_{\dot{B}^0_{p,r}(\RR^3)} \\
\leq&C\left\| 2^{\frac{5k}{2} }\|\dot{\Delta}_{k} \uu^\sharp\|_{L^{2}(\RR^3)}\left\|\widetilde{\dot{\Delta}}_k \vv\right\|_{L^{p}(\RR^3)}\right\|_{\ell^r(\ZZ)}+C\left\| 2^{\frac{5k}{2} }\|\dot{\Delta}_{k} \uu^\natural\|_{L^{2}(\RR^3)}\left\|\widetilde{\dot{\Delta}}_k \vv\right\|_{L^{p}(\RR^3)}\right\|_{\ell^r(\ZZ)}\\
\leq&C2^{\frac{5N}{3} }\|\uu \|_{\dot{B}^{\frac{5}{6}}_{2,\infty}(\RR^3)}  \| \vv\|_{\dot{B}^{0}_{p,r}(\RR^3)}+C\sup_{k\geq N}2^{\frac{5k}{6}}\|\dot{\Delta}_k\uu \|_{L^2(\RR^3)}  \| \vv\|_{\dot{B}^{\frac{5}{3}}_{p,r}(\RR^3)}.
\end{align*}
This estimate together with \eqref{eq-l-1-1-1-5} and \eqref{eq-l-1-2-3}, we immediately obtain that
\begin{equation}\label{eq-collecting}
\begin{split}
\|\operatorname{div}\left(\uu\otimes \vv\right)\|^2_{\dot{B}^{0}_{p,2}(\RR^3)}\leq &\frac{\varepsilon}{2}
\|\uu\|_{\dot{B}^\frac{5}{3}_{p,2}(\RR^3)}+C\sup_{k\geq N}2^{\frac{5k}{6}}\|\dot{\Delta}_k\uu \|_{L^2(\RR^3)}\|\uu\|_{\dot{B}^{\frac{5}{3}}_{p,r}(\RR^3)}\\
&+C_\varepsilon 2^{\frac{5N}{3}}\|\uu\|^{\frac52}_{\dot{B}^{\frac{5}{6}}_{2,\infty}(\RR^3)}\|\uu\|_{\dot{B}^0_{p,r}(\RR^3)}.
  \end{split}
\end{equation}
Since $\ell^r(\ZZ)\subset\ell^\infty(\ZZ)$  for each $r\in[1,\infty)$ and $\uu\in \dot{H}^{\frac{5}{6}}(\RR^3)$, we have that, for any $\varepsilon>0,$  there exists a suitable integer $N$ such that
\[\sup_{k\geq N}2^\frac{5k}{6}\|\dot{\Delta}_k \uu\|^2_{L^{2}(\RR^3)}\leq\bigg(\sum_{k\geq N}2^\frac{5k}{3}\|\dot{\Delta}_k \uu\|^2_{L^{2}(\RR^3)}\bigg)^{\frac12}\leq \frac{\varepsilon}{2C},\]
which together with \eqref{eq-collecting} implies the desired result \eqref{eq-KY-Fractional-BS} in Lemma \ref{lem-SYB}.
\end{proof}
The action of the heat semigroup $e^{-t(-\Delta)^\alpha}$ on distributions with Fourier transforms supported in an annulus is described in following lemma.
\begin{lemma}\label{lem-exe-decay}
Let $\alpha\in(0,1]$ and  $p,\,r\in[1,\infty]$. Then there exists two positive constants $c$  and $C$ such that for $p\geq r,$
\begin{equation}\label{eq-exe-decay}
\left\|\big(|\cdot|^m D^\gamma O_{\alpha,q}\big)\ast f\right\|_{L^p(\RR^3)}\leq C2^{q(|\gamma|-m+3(1/r-1/p))}e^{-ct2^{2q\alpha}}\big\| f \big\|_{L^r(\RR^3)},
\end{equation}
where $G_{\alpha,q} $ is the heat kernel with fractional diffusion $(-\Delta)^\alpha$ and  $O_{\alpha,q}=\dot{\Delta}_q\mathbb{P}G_\alpha$.
\end{lemma}
\begin{proof}
First of all,  we take  the Fourier transform  of $D^\gamma O_{\alpha,q}f$ yielding
\[\mathscr{F}\left(D^\gamma O_{\alpha,q}f\right)=2^{q|\gamma|}\mathscr{F}\left(\mathbb{P}D^\gamma \varphi\right)(2^{q}\xi)e^{-t|\xi|^{2\alpha}}\hat{f}.\]
Denoting $\phi(\xx):=\mathbb{P}D^\gamma\varphi$, we easily verify  that the $\phi(\xx)$ is a smooth function with Fourier transform compactly supported in the ring $\mathscr{C}(0;3/4,8/3)$. Hence we have
\begin{equation}\label{eq.idenity}
D^\gamma O_{\alpha,q}f=2^{q|\gamma|}2^{3q}g\left(2^{2\alpha q}t,\,2^{q}\cdot\right)\ast f,
\end{equation}
where $$g(t,\xx)=\mathscr{F}^{-1}\left(\hat{\phi}(\xi)e^{-t|\xi|^{2\alpha}}\right).$$
When $m=2k_0$ is nonnegative even number, we take the Fourier transform to obtain
\begin{align*}
|\xx|^m g(t,\xx)=&\frac{1}{(1+|\xx|^2)^3}
\frac{1}{(2\pi)^{\frac32}}\int_{\mathbf{R}^3}e^{i\xx\cdot\xi}(I_d-\Delta_\xi)^3(\Delta_\xi)^{k_0}\left(\hat{\phi}(\xi)e^{-t|\xi|^{2\alpha}}\right)\,\mathrm{d}\xi.
\end{align*}
Thanks to the Leibiniz rule, one has
\begin{equation*}
(\Delta_\xi)^{k_0}\left(\hat{\phi}(\xi)e^{-t|\xi|^{2\alpha}}\right)= \sum_{j=1}^3\sum_{i=0}^{m}C_{m}^{i}\left(\partial_{x_j}^{m-i} \hat{\phi}(\xi)\right)\left(\partial^{i}_{x_j}e^{-t|\xi|^{2\alpha}} \right).
\end{equation*}
Moreover, we have
\begin{align*}
&(I_d-\Delta_\xi)^3(\Delta_\xi)^{k_0}\left(\hat{\phi}(\xi)e^{-t|\xi|^{2\alpha}}\right)\\
=&\sum_{\beta_1\leq\beta,\,|\beta|\leq6}C_{\beta}^{\beta_1}\bigg(\sum_{j=1}^3\sum_{i=0}^{m}C_{m}^{i}\partial^{\beta-\beta_1}\left(\partial_{x_j}^{m-i} \hat{\phi}(\xi)\right)\partial^{\beta_1}\left(\partial^{i}_{x_j}e^{-t|\xi|^{2\alpha}} \right)\bigg).
\end{align*}
On the other hand, we have by the Fa\'a-di-Bruno formula that
\begin{equation*}
e^{t|\xi|^{2\alpha}}\partial^{\beta_1}\left(\partial^{i}_{x_j}e^{-t|\xi|^{2\alpha}} \right) =\sum_{\substack{\beta_{11}+\cdots+\beta_{1\ell}=|\beta_1|\\|\beta_{1k}|\geq0}}(-t)^{|\beta_1|+i}\prod_{k=1}^{\ell}\partial^{\beta_{1k} }\partial^i_{x_j}\big( |\xi|^{2\alpha}\big).
\end{equation*}
Since $\operatorname{supp}\hat{\phi}(\xi)\subset\mathscr{C}(0;3/4,8/3)$, we have that for each $\xi\in \operatorname{supp}\hat{\phi}(\xi), $
\begin{align*}
 \left|\partial^{\beta-\beta_1}\left(\partial_{x_j}^{m-i} \hat{\phi}(\xi)\right)\partial^{\beta_1}\left(\partial^{i}_{x_j}e^{-t|\xi|^{2\alpha}} \right)\right|
\leq&C\left(1+t\right)^{|\beta_1|+i}e^{-t|\xi|^{2\alpha}}\\ \leq& C\left(1+t\right)^{|\beta_1|+i}e^{-\frac{9}{16}t }.
\end{align*}
Since $$\lim_{t\to+\infty}\left(1+t\right)^{6+m}e^{-\frac{9}{32}t }=0,$$
we immediately have the following pointwise estimate that for each even nonnegative integer  $m=2k_0$,
\begin{equation}\label{eq-even-pe}
|\xx|^m g(t,\xx)\leq C e^{-ct}\frac{1}{(1+|\xx|^2)^3}.
\end{equation}
If the real number $m$ is not an even integer, there exists $\theta\in(0,1)$ such that
\[m=\theta\overline{m}+(1-\theta)\big(\overline{m}+2\big),\]
where $\overline{m}=2[m/2].$

Furthermore, we get by the interpolation inequality and \eqref{eq-even-pe} that for each $m\geq0,$
\begin{equation*}
|\xx|^m g(t,\xx)=\left(|\xx|^{\overline{m}} g(t,\xx)\right)^{\theta } \left(|\xx|^{\overline{m}+2} g(t,\xx)\right)^{(1-\theta) }  \leq C e^{-ct}\frac{1}{(1+|\xx|^2)^3}.
\end{equation*}
From this estimate, it follows that
\begin{equation}\label{eq-all-pe}
\big\||\cdot|^m g(t,\cdot)\big\|_{L^p(\RR^3)}\leq Ce^{-\frac{9}{32}t}\quad\text{for each}\,\,p\in[1,\infty].
\end{equation}
Recall from \eqref{eq.idenity}, one writes
\begin{align*}
\big(|\cdot|^m D^\gamma O_{\alpha,q}\big)\ast f=&\left(|\cdot|^m 2^{q|\gamma|}2^{3q}g\left(2^{2\alpha q}t,\,2^{q}\cdot\right)\right)\ast f\\
=&2^{q(|\gamma|-m)}2^{3q}\left(|\cdot|^m g\right)\left(2^{2\alpha q}t,\,2^{q}\cdot\right)\ast f.
\end{align*}
This equality together with \eqref{eq-all-pe} enables us to conclude that for each $1\leq r\leq p\leq \infty,$
\begin{equation*}
\left\|\big(|\cdot|^m D^\gamma O_{\alpha,q}\big)\ast f\right\|_{L^p(\RR^3)}\leq C2^{q(|\gamma|-m+3(1/r-1/p))}e^{-ct2^{2q\alpha}}\big\| f \big\|_{L^r(\RR^3)}.
\end{equation*}
So we finish the proof of Lemma \ref{lem-exe-decay}.
\end{proof}
\begin{lemma}[\cite{ref-CMZ07}, New Bernstein's inequality]\label{lem-bern-new}
Let $p \in[2, +\infty)$ and $\alpha \in(0,1].$ Then there exist two positive constants $c_{p}$ and $C_{p}$ such that for any $f \in \mathscr{S}^{\prime}(\RR^n)$ and $j \in \mathbb{Z}$,
$$
c_{p} 2^{\frac{2\alpha j}{p}}\big\|\dot{\Delta}_{j} f\big\|_{L^p(\RR^n)} \leq\left\|\Lambda^{\alpha}\left(\big|\dot{\Delta}_{j} f\big|^{\frac{p}{2}}\right)\right\|_{L^2(\RR^n)}^{\frac{2}{p}} \leq C_{p} 2^{\frac{2 \alpha j}{p}}\big\|\dot{\Delta}_{j} f\big\|_{L^p(
\RR^n)}.
$$
\end{lemma}
\begin{lemma}[\cite{LMZ21}]\label{lem-Comm}
	Let $\alpha\in(0,1),\, f\in L^2(\RR^3)$ and $\phi\in \dot{C}^{0,\beta}(\RR^3)\cap \dot{W}^{1,\infty}(\RR^3)$. Then there exists a constant $C>0$ such that for each $ \beta\in(0,\alpha),$
	\[\big\|\left[\Lambda^\alpha,\,\phi\right]f\big\|_{L^2(\RR^3)}\leq C\max\left\{\|\phi\|_{\dot{C}^{\beta}(\RR^3)},\,\|\phi\|_{\dot{W}^{1,\infty}(\RR^3)}\right\}\|f\|_{L^2(\RR^3)}.\]
\end{lemma}
\begin{lemma}\label{lem-com-two}
Let $s\in(0,1)$, $\beta\in(0,1]$ and $f\in L^2_{\langle\cdot\rangle^{2\beta}}(\RR^3)$. Then we have
\begin{itemize}
  \item [(i)] for each $\beta<s,$ \[\left\|\left[\langle\cdot\rangle^\beta, \Lambda^s\right] f\right\|_{L^2(\RR^3)}\leq C\left\|  f\right\|_{L^2(\RR^3)};\]
  \item [(ii)]for each $\beta\geq s,$  \[\left\|\left[\langle\cdot\rangle^\beta, \Lambda^s\right] f\right\|_{L^2(\RR^3)}\leq C\left\|\langle\cdot\rangle^{\beta} f\right\|_{L^2(\RR^3)}.\]
\end{itemize}
\end{lemma}
\begin{proof}
If $\beta<s,$  we immediately have by Lemma \ref{lem-Comm} that
 \[\left\|\left[\langle\cdot\rangle^\beta, \Lambda^s\right] f\right\|_{L^2(\RR^3)}\leq C\left\|f\right\|_{L^2(\RR^3)}.\]
 Now, it remains for us to consider the case where $\beta\geq s.$
Since $\Lambda=\sqrt{-\Delta},$ we have in terms of definition of fractional Laplacian that
\begin{align*}
\left[\langle\xx\rangle^\beta, \Lambda^s\right] f(\xx)=-c_s\int_{\RR^3}\frac{\langle\xx\rangle^\beta-\langle\yy\rangle^\beta}{|\xx-\yy|^{3+s}}f(\yy)\,\mathrm{d}\yy.
\end{align*}
We decompose it into the following two parts
\begin{align*}
\int_{\RR^3}\frac{\langle\xx\rangle^\beta-\langle\yy\rangle^\beta}{|\xx-\yy|^{3+s}}f(\yy)\,\mathrm{d}\yy=&
\left(\int_{B_{100}(\xx)} +\int_{\RR^3\backslash B_{100}(\xx)}\right)\frac{\langle\xx\rangle^\beta-\langle\yy\rangle^\beta}{|\xx-\yy|^{3+s}}f(\yy)\,\mathrm{d}\yy\\
=:&{\rm I}+{\rm II}.
\end{align*}
The first term ${\rm I}$ can be bounded by
\begin{align*}
|{\rm I}|\leq \beta\int_{B_{100}(\xx)}\frac{1}{|\xx-\yy|^{2+s}}|f|(\yy)\,\mathrm{d}\yy,
\end{align*}
which together with the Young inequality yields
\begin{equation}\label{eq-2022-I}
\|{\rm I}\|_{L^2(\RR^3)}\leq C\|f\|_{L^2(\RR^3)}.
\end{equation}
As for term ${\rm II}$,  we split it into two parts as follows
\begin{align*}
{\rm II}=&
\bigg(\int_{ B_{\frac{|\xx|}{2}}(0)\backslash B_{100}(\xx) } +\int_{B^c_{\frac{|\xx|}{2}}(0)\backslash B_{100}(\xx)}\bigg)\frac{\langle\xx\rangle^\beta-\langle\yy\rangle^\beta}{|\xx-\yy|^{3+s}}f(\yy)\,\mathrm{d}\yy\\
=:&{\rm II}_1+{\rm II}_2,
\end{align*}
where $B^c_{\frac{|\xx|}{2}}(0)=\RR^3\backslash B_{\frac{|\xx|}{2}}(0).$

Since
\begin{align*}
\left|\langle\xx\rangle^\beta-\langle\yy\rangle^\beta \right|\leq \left||\xx|^2-|\yy|^2\right|^{\frac{\beta}{2}}
\leq  \left|\xx\right|^\beta+\left |\yy\right|^\beta,
\end{align*}
  we see that
\begin{align*}
| {\rm II}_2|\leq \left(1+2^\beta\right)\int_{\RR^3\backslash B_{100}(\xx)}\frac{1}{|\xx-\yy|^{3+s }}\langle\yy\rangle^{\beta}|f|(\yy)\,\mathrm{d}\yy.
\end{align*}
Moreover, by the Young inequality again, we get
\begin{equation}\label{eq-2022-II-2}
\|{\rm II}_2\|_{L^2(\RR^3)}\leq C\left\|\langle\cdot\rangle^{\beta} f\right\|_{L^2(\RR^3)}.
\end{equation}
Now, we turn to bound the term ${\rm II}_1.$ For $\yy\in  B_{\frac{|\xx|}{2}}(0)$, we observe by the triangle inequality that
\begin{equation}\label{eq-triangle-11}
\frac12|\xx|\leq|\xx|-|\yy|\leq|\xx-\yy|\leq |\xx|+|\yy|\leq \frac32|\xx|.
\end{equation}
From this inequality, we have
\begin{align*}
\left|{\rm II}_1\right|\leq 4\int_{\RR^3\backslash B_{100}(\xx)}\frac{1}{|\xx-\yy|^{3+s-\beta }}|f|(\yy)\,\mathrm{d}\yy.
\end{align*}
By the H\"older inequality, the Young inequality and the fact $3+s-\beta>2$, one obtains
\begin{equation}\label{eq-2022-II-1-s}
\left\|{\rm II}_1\right\|_{L^2(B_1(0))}\leq C\left\|{\rm II}_1\right\|_{L^\infty(B_1(0))}\leq C\|f\|_{L^2(\RR^3)}.
\end{equation}
In terms of \eqref{eq-triangle-11}, we find that
\begin{align*}
\left|{\rm II}_1\right|\leq& 4\int_{B_{\frac{|\xx|}{2}}(0)}\frac{1}{|\xx-\yy|^{3+s-\beta }}|f|(\yy)\,\mathrm{d}\yy\leq \frac{C}{|\xx|^{3+s-\beta }}\int_{B_{\frac{|\xx|}{2}}(0)}|f|(\yy)\,\mathrm{d}\yy.
\end{align*}
By the Cauchy-Schwarz inequality, one gets
\begin{align*}
\left|{\rm II}_1\right|\leq& \frac{C}{|\xx|^{3+s-\beta }}\int_{B_{\frac{|\xx|}{2}}(0)}\frac{1}{|\yy|^\beta}\left(|\yy|^\beta|f|\right)(\yy)\,\mathrm{d}\yy\\
\leq& \frac{C}{|\xx|^{3+s-\beta }} \bigg(\int_{B_{\frac{|\xx|}{2}}(0)}\frac{1}{|\yy|^{2\beta}}\,\mathrm{d}\yy\bigg)^{\frac12}\left\|\langle\cdot\rangle^{\beta} f\right\|_{L^2(\RR^3)}\leq\frac{C}{|\xx|^{\frac32+s  }}\left\|\langle\cdot\rangle^{\beta} f\right\|_{L^2(\RR^3)}.
\end{align*}
Hence, we have
\begin{equation} \label{eq-2022-II-1-l}
\left\|{\rm II}_1\right\|_{L^2(\RR^3\backslash B_1(0))}\leq C\left\|\langle\cdot\rangle^{\beta} f\right\|_{L^2(\RR^3)}\bigg(\int_{\RR^3\backslash B_1(0)}\frac{1}{|\xx|^{ 3 +2s  }}\,\mathrm{d}\xx\bigg)^{\frac12} \leq C\left\|\langle\cdot\rangle^{\beta} f\right\|_{L^2(\RR^3)}.
\end{equation}
Collecting the above estimates \eqref{eq-2022-I}, \eqref{eq-2022-II-2}, \eqref{eq-2022-II-1-s} and \eqref{eq-2022-II-1-l}, we eventually obtain
\[\left\|\left[\langle\cdot\rangle^\beta, \Lambda^s\right] f\right\|_{L^2(\RR^3)}\leq C\left\|\langle\cdot\rangle^{\beta} f\right\|_{L^2(\RR^3)},\]
which implies the second desired estimate of Lemma \ref{lem-com-two}.
\end{proof}
\begin{lemma}\label{weighted-h-equiv}
Let $\beta\in(0,1]$ and $f\in\mathscr{D}(\RR^3)$. There exists a constant  $C>0$  such that
\begin{equation}\label{eq-w20220817}
\frac{1}{C }\big\|\langle\cdot\rangle\Lambda f\big\|_{L^2(\RR^3)}\leq \big\|\langle\cdot\rangle\nabla f\big\|_{L^2(\RR^3)}\leq C \big\|\langle\cdot\rangle\Lambda f\big\|_{L^2(\RR^3)}.
\end{equation}
\end{lemma}
\begin{proof}
According to the fact  that $\displaystyle\sum_{i=1}^3\frac{(-i\xi)(i\xi)}{|\xi|^2}=1$, we  write
\[\Lambda f=-\sum_{i=1}^3\mathscr{R}_i\partial_{x_i}   f,\]
where $\mathscr{R}_i$ is the Riesz   operator.

 Recall from \cite[Chapter V-6.4]{Stein} that  the function $|\xx|^a$ belong to the $A_p$ class with $p>1,$ if and only if $-n<a<n(p-1).$
Thus,  the function $|\xx|^{2\beta}$ belongs to the $A_2$ class as long as $\beta\in(0,1]$, and then we immediately have by the weighted inequality for singular integral established in  \cite[Chapter V-4.2]{Stein}  that
 \begin{equation}\label{eq-w20220817-1}
 \left\|\langle\cdot\rangle^{\beta}\Lambda f\right\|_{L^{2}(\RR^3)}\leq \sum_{i=1}^3\left\|\langle\cdot\rangle^{\beta}\mathscr{R}_i\partial_{x_i}f\right\|_{L^{2}(\RR^3)}
 \leq C\sum_{i=1}^3\left\|\langle\cdot\rangle^{\beta} \partial_{x_i}f\right\|_{L^{2}(\RR^3)},
 \end{equation}
 which gives the first inequality in \eqref{eq-w20220817}.

To see the inverse inequality, we write $\partial_{x_i}f$ in terms of Fourier transform as
\[\mathscr{F}(\partial_{x_i}f)(\xi)=-i\xi_i\hat{f}(\xi)=\frac{-i\xi_i}{|\xi|}|\xi|\hat{f}(\xi)=-\mathscr{F}\left(\mathscr{R}_i\Lambda f\right)(\xi) \]
for  each $i\in\{1,2,3\}.$

In parallel with \eqref{eq-w20220817-1}, we conclude that for  each $i\in\{1,2,3\},$
\begin{align*}
\left\|\langle\cdot\rangle^{\beta}\partial_{x_i} f\right\|_{L^{2}(\RR^3)}\leq \left\|\langle\cdot\rangle^{\beta}\mathscr{R}_i\Lambda f\right\|_{L^{2}(\RR^3)}
 \leq C\left\|\langle\cdot\rangle^{\beta} \Lambda f\right\|_{L^{2}(\RR^3)}.
\end{align*}
The sum of all $i$ is the second inequality in \eqref{eq-w20220817}.
\end{proof}
%%%%%%%%%%%%%%%%%%%%%%%%%%%%%%%%%%%%%%%%%%%%%%%%%%%%%%%%%%%%%%%%%%%%%%%%%%%%%%%%%%%%%%%%%%%%%%
\subsection{$L^p$-type theory of the linearized  Leray problem}
%%%%%%%%%%%%%%%%%%%%%%%%%%%%%%%%%%%%%%%%%%%%%%%%%%%%%%%%%%%%%%%%%%%%%%%%%%%%%%%%%%%%%%%%%%%%%%
In this subsection, we are going to investigate the high regularity of the Leray operator in $\RR^3$
\begin{align}\label{E-L}
\left\{
\begin{aligned}
& (-\Delta)^{\alpha}\vv- \frac{2\alpha-1}{2\alpha}\vv-\frac{1}{2\alpha}\xx\cdot \nabla \vv
+\uu\cdot\nabla \vv+\nabla P=\ff, \\
& \textnormal{div}\,\vv=0.
\end{aligned}\ \right.
\end{align}
Let the couple $(\vv,P)\in H^\alpha(\RR^3)\times L^{\frac{3}{3-2\alpha}}(\RR^3)$ be a weak solution to the system \eqref{E-L} which satisfies
\begin{equation}\label{eq.weak}
\begin{split}
&\int_{\RR^3}\Lambda^\alpha \vv\cdot\Lambda^\alpha \phi\,\mathrm{d}\yy+\frac{2-\alpha}{\alpha}\int_{\RR^3}\vv\cdot\phi\,\mathrm{d}\yy+\frac{1}{2\alpha}\int_{\mathbf{R}^3}\yy\cdot \nabla\varphi\cdot \vv\,\mathrm{d}\yy\\=&\int_{\RR^3}P\operatorname{div}\phi\,\mathrm{d}\yy+\int_{\RR^3}\ff\cdot\phi\,\mathrm{d}\yy+\int_{\RR^3}(\uu\cdot\nabla)\phi\cdot\vv\,\mathrm{d}\yy
\end{split}
\end{equation}
for each  divergence-free vector field $ {\phi}\in \mathscr{D}(\RR^3)$.
\begin{theorem}\label{thm-V-H}
Let $s\in\RR,\,p\in[2,\infty),\,r\in[1,\infty]$ and $\uu\equiv0$.
Assume that $(\vv,P)\in H^\alpha(\RR^3)\times L^{\frac{3}{3-2\alpha}}(\RR^3)$ is a weak solution to the system~\eqref{E-L} with $\alpha\in(0,1]$. Then there exists a constant $C=C_{p,\alpha}>0$ such that
\begin{equation}\label{eq-V-H}
\left\|\vv\right\|_{\dot{B}^{s+2\alpha}_{p,r}(\RR^3)}+\left\|P\right\|_{\dot{B}^{s+1}_{p,r}(\RR^3)}\leq C_{p,\alpha}\left(\left\|\ff\right\|_{\dot{B}^{s }_{p,r}(\RR^3)}+\left\|\vv\right\|_{\dot{B}^{s}_{p,r}(\RR^3)}\right).
\end{equation}
\end{theorem}
\begin{proof}
Taking $\phi(\yy)=\ee_i \varphi_q(\yy)$ with $ \varphi_q(\yy)=2^{3q}\varphi(2^q\yy)$ in the above equality and denoting $f_q=\dot{\Delta}_qf$ for each $f$, it yields by a simple calculation that for each $i=1,2,3,$
\begin{equation*}
 (-\Delta)^{\alpha} \vv^{(i)}_q-\frac{2-\alpha}{2\alpha}\dot{\Delta}_q\vv^{(i)}-\frac{1}{2\alpha}\xx\cdot \nabla \vv^{(i)}_q+\partial_{x_i}P_q= \ff^{(i)}_q+\mathbf{R}^{(i)}_q.
\end{equation*}
This implies that the couple $(\vv_q,P_q):=(\dot{\Delta}_q\vv,\dot{\Delta}_qP)$ is  smooth and solves the following system in the whole space
\begin{align}\label{E-q-L}
\left\{
\begin{aligned}
& (-\Delta)^{\alpha}\vv_q-\frac{2-\alpha}{2\alpha}\vv_q-\frac{1}{2\alpha}\xx\cdot \nabla \vv_q+\nabla P_q=\ff_q+\mathbf{R}_q,\\
& \textnormal{div}\,\vv_q=0,
\end{aligned} \right.
\end{align}
where the commutator  $R_q$ defined by
\[R_q=[\dot{\Delta}_q,\xx\otimes ]\vv=\dot{\Delta}_q\left(\xx\otimes \vv\right)-\xx\otimes\dot{\Delta}_q \vv.\]
Since the Fourier transform $\mathscr{F}:\,\mathscr{S}\left(\mathbf{R}^{3}\right)\to \mathscr{S}\left(\mathbf{R}^{3}\right)$ is a bounded linear  operator,  we know that  each term in \eqref{E-q-L}  is a Schwarz function.
So, multiplying the first equation of \eqref{E-q-L} by $|\vv_q|^{p-2}\vv_q$ with $p\geq2$ and then integrating the resulting equality in space variable over the whole space, we get
\begin{equation}\label{eq-L-Lp}
\begin{split}
&\int_{\mathbf{R}^3}(-\Delta)^{\alpha}\vv_q\,|\vv_q|^{p-2}\vv_q\,\mathrm{d}\xx-\frac{2\alpha -1}{ 2\alpha}\left\|\vv_q\right\|_{L^p(\RR^3)}^p\\&-\frac{1}{2\alpha}\int_{\mathbf{R}^3}\xx\cdot \nabla \vv_q\,|\vv_q|^{p-2}\vv_q\,\mathrm{d}\xx\\
=&\int_{\mathbf{R}^3}\left(-\nabla P_q +\ff_q \right)\,|\vv_q|^{p-2}\vv_q\,\mathrm{d}\xx+ \int_{\mathbf{R}^3}\mathbf{R}_q\,|\vv_q|^{p-2}\vv_q\,\mathrm{d}\xx.
\end{split}
\end{equation}
The new Bernstein inequality in Lemma \ref{lem-bern-new} enables us to infer the lower bound of the first quantity in the left side of equality \eqref{eq-L-Lp} that for each $p\in[2,\infty),$
\begin{equation}\label{eq-L-Lp-diff-lower}
\int_{\mathbf{R}^3}(-\Delta)^{\alpha}\,|\vv_q|^{p-2}\vv_q\,\mathrm{d}\xx\geq c_p2^{2q\alpha} \left\|\vv_q\right\|^p_{L^p(\RR^3)}.
\end{equation}
Integrating by parts yields
\begin{equation}\label{eq-L-Lp-self}
\begin{split}
&-\frac{2\alpha -1}{2 \alpha}\left\|\vv_q\right\|_{L^p(\RR^3)}^p-\frac{1}{2\alpha}\int_{\mathbf{R}^3}\xx\cdot \nabla \vv_q\,|\vv_q|^{p-2}\vv_q\,\mathrm{d}\xx\\
=&-\frac{2\alpha-1}{ 2\alpha}\left\|\vv_q\right\|_{L^p(\RR^3)}^p+\frac{3}{2p\alpha}\left\|\vv_q\right\|_{L^p(\RR^3)}^p
=\frac{3+p-2p\alpha}{2p\alpha}\left\|\vv_q\right\|_{L^p(\RR^3)}^p.
\end{split}
\end{equation}
By the H\"older inequality, one has
\begin{equation}\label{eq-L-Lp-f}
\int_{\mathbf{R}^3}\left(-\nabla P_q +\ff_q \right) \,|\vv_q|^{p-2}\vv_q\,\mathrm{d}\xx
\leq \left(\left\|\nabla P_q\right\|_{L^p(\RR^3)}+\left\|\ff_q\right\|_{L^p(\RR^3)}\right)\left\|\vv_q\right\|_{L^p(\RR^3)}^{p-1}.
\end{equation}
Now  we turn to deal with the integral term involving commutator which takes
\begin{align*}
[\dot{\Delta}_q,\xx\otimes]\vv=&\int_{\RR^3}\varphi_q(\xx-\yy)\yy\otimes \vv(\xx-\yy)\,\mathrm{d}\yy-\xx\otimes\int_{\RR^3}\varphi_q(\xx-\yy)\vv(\yy)\,\mathrm{d}\yy\\
=&-\int_{\RR^3}(\xx-\yy)\varphi_q(\xx-\yy)\otimes \vv(\yy)\,\mathrm{d}\yy.
\end{align*}
Letting $\bar{\varphi}(\xx)=\xx \varphi(\xx)\in\mathscr{D}(\RR^3),$   taking the Fourier transform and using the support property of $\varphi$, we find
\begin{align*}
\mathscr{F}\left([\dot{\Delta}_q,\xx\otimes]\vv \right)(\xi)=&2^{-q}\widehat{ \bar{\varphi}_q }(\xi)\widehat{\vv}(\xi)
= i2^{-q}\operatorname{div}_{\xi} \varphi\left( {2^{-q}}{\xi} \right) \otimes\widehat{\vv}(\xi)\\
= & i2^{-q}\operatorname{div}_{\xi} \varphi\left( {2^{-q}}{\xi} \right) \otimes\widehat{\widetilde{\dot{\Delta}}_q\vv}(\xi),
\end{align*}
which means that
\begin{equation}\label{eq-L-Lp-comm}
[\dot{\Delta}_q,\xx\otimes]\vv=2^{-q}\bar{\varphi}_q\ast\left(\widetilde{\dot{\Delta}}_q\vv\right)(\xx).
\end{equation}
With the equality \eqref{eq-L-Lp-comm} in hand, we infer by the H\"older inequality and the Bernstein inequality in Lemma \ref{lem-bern} that
\begin{equation}\label{eq-L-Lp-comm-I}
\begin{split}
\int_{\mathbf{R}^3}\mathbf{R}_q\,|\vv_q|^{p-2}\vv_q\,\mathrm{d}\xx
\leq&\frac{1}{2\alpha}2^{-q}\left\|\bar{h}_q\ast\nabla\left(\widetilde{\dot{\Delta}}_q\vv\right)\right\|_{L^p(\RR^3)}\left\|\vv_q\right\|_{L^p(\RR^3)}^{p-1}\\
\leq&\frac{C}{2\alpha} \left\|\widetilde{\dot{\Delta}}_q\vv \right\|_{L^p(\RR^3)}\left\|\vv_q\right\|_{L^p(\RR^3)}^{p-1}.
\end{split}
\end{equation}
Inserting estimates \eqref{eq-L-Lp-diff-lower}, \eqref{eq-L-Lp-self}, \eqref{eq-L-Lp-f} and \eqref{eq-L-Lp-comm-I} into the equality \eqref{eq-L-Lp} leads to
\begin{equation}\label{eq-L-Lp-sec}
\begin{split}
&c_p2^{2q\alpha} \left\|\vv_q\right\|_{L^p(\RR^3)}+\frac{3+p-2p\alpha}{2p\alpha}\left\|\vv_q\right\|_{L^p(\RR^3)}\\
\leq& \left\|\nabla P_q\right\|_{L^p(\RR^3)}+\left\|\ff_q\right\|_{L^p(\RR^3)}+\frac{C}{2\alpha} \left\|\widetilde{\dot{\Delta}}_q\vv \right\|_{L^p(\RR^3)}.
\end{split}
\end{equation}
Our task is now to bound the quantity concerning the pressure $P.$ Applying the divergence operator to the first equation of \eqref{E-q-L}, we readily obtain by a simple calculation that
\begin{equation}\label{eq-L-Lp-P}
-\Delta P_q=\operatorname{div}\ff_q\quad\text{in }\RR^3.
\end{equation}
 The standard $L^p$-estimate for elliptic equations for each $p\in[2,\infty),$
\[\left\|P_q\right\|_{\dot{W}^{2,p}(\RR^3)}\leq C\|\operatorname{div}\ff_q\|_{L^p(\RR^3)} \]
means that
\begin{equation}\label{eq-L-Lp-P-est}
2^{q}\left\|P_q\right\|_{L^p(\RR^3)}\leq C\|\ff_q\|_{L^p(\RR^3)}\quad\text{for each}\,\,p\in[2,\infty).
\end{equation}
Plugging estimate \eqref{eq-L-Lp-P-est} into \eqref{eq-L-Lp-sec}, we immediately have
\begin{equation}\label{eq-L-Lp-third}
\begin{split}
&c_p2^{2q\alpha} \left\|\vv_q\right\|_{L^p(\RR^3)}+\frac{3+p-2p\alpha}{2p\alpha}\left\|\vv_q\right\|_{L^p(\RR^3)}+2^q\left\| P_q\right\|_{L^p(\RR^3)}\\
\leq&C\left\|\ff_q\right\|_{L^p(\RR^3)}+\frac{C}{2\alpha} \left\|\widetilde{\dot{\Delta}}_q\vv \right\|_{L^p(\RR^3)}.
\end{split}
\end{equation}
Multiplying \eqref{eq-L-Lp-third} by $2^{qs}$ and then taking the $\ell^r$-norm of the resulting inequality with respect to $q$, we eventually obtain that for each $p\in[2,\infty),$
\[\left\|\vv\right\|_{\dot{B}^{s+2\alpha}_{p,r}(\RR^3)}+\left\|P\right\|_{\dot{B}^{s+1}_{p,r}(\RR^3)}\leq C_{p,\alpha}\left(\left\|\ff\right\|_{\dot{B}^{s }_{p,r}(\RR^3)}+\left\|\vv\right\|_{\dot{B}^{s}_{p,r}(\RR^3)}\right),\]
which implies the desired estimate in Theorem \ref{thm-V-H}.
\end{proof}
Next, we give the estimate of weak solution $\vv$ in the weighted Hilbert space. In forthcoming part of this subsection, we always assume that
\[\ff=\ff_1+\Div\ff_2.\]
\begin{theorem}\label{thm-V-H-L2-weight}
Assume  the divergence free vector field $\uu\in  L^\infty(\RR^3)\cap \dot{W}^{1,\frac3\alpha}(\RR^3)$.
Let the couple  $(\vv,P)\in H^{1}(\RR^3)\times H^1(\RR^3)$ be a weak solution to the system \eqref{E-L} with $\alpha\in(0,1]$. Then there exists  a constant  $C >0$ such that
\begin{itemize}
  \item[(i)]  if $\alpha\in(0,1)$, $\ff_1\in L_{\langle\xx\rangle^{2\beta}}^2(\RR^3)$  and $\ff_2\in H^1(\RR^3)$  satisfying $$\left\|\langle \cdot\rangle^{ \beta}\Div\ff_{2}\right\|_{H^{-\alpha}(\RR^3)}<\infty,$$ we have that  for each $\beta\in(0,\alpha),$
\begin{equation}\label{eq-V-H-weight}
  \left\|\langle \cdot\rangle^{\beta} \vv\right\| _{H^\alpha(\RR^3)}
  \leq   C\Big(  \left\|\ff_1\right\|_{L^2_{\langle\xx\rangle^{2\beta}}(\RR^3)}+\left\|\langle \cdot\rangle^{ \beta}\Div\ff_{2}\right\|_{H^{-\alpha}(\RR^3)}+ \|\ff_2\|_{H^1(\RR^3)}\Big);
\end{equation}
  \item  [(ii)]  if $\alpha=1,$ $\ff_1\in L_{\langle\xx\rangle^2}^2(\RR^3)$  and $\ff_2\in H^1(\RR^3)$  satisfying $$\left\|\langle \cdot\rangle^{ \beta}\Div\ff_{2}\right\|_{H^{-1}(\RR^3)}<\infty,$$ we have
\begin{equation}\label{eq-V-H-weight-1}
\left\|\langle \cdot\rangle \vv\right\| _{H^1(\RR^3)}
  \leq   C\Big(\left\|\ff_1\right\|_{L^2_{\langle\xx\rangle^2}(\RR^3)}+\left\|\langle \cdot\rangle\Div\ff_{2}\right\|_{\dot{H}^{-1}(\RR^3)}+ \|\ff_2\|_{H^1(\RR^3)}\Big).
\end{equation}
\end{itemize}
\end{theorem}
\begin{proof}
Since the proof of \eqref{eq-V-H-weight-1} was shown in \cite{LMZ21}, so we focus on the proof of the case where $\alpha\in(0,1).$
Denoting $g_R(\xx)=g(\xx/R)$ with $g(\xx)=\frac{1}{ (1+ {|\xx|^2} )^{\frac14+\frac{\beta}{2}}}$ and $R\gg1$,
\[\vv_R:=g_R\vv,\quad P_R := g_RP\quad\text{and}\quad  \ff_R := g_R\ff,\]
we easily find by taking  the test function as $\langle \xx\rangle^{2\beta} g_R(\xx)\vv_R$ in   \eqref{eq.weak}  that
\begin{equation}\label{eq-L2-weight-1}
\begin{split}
&\int_{\RR^3}\Lambda^{\alpha}\vv\cdot\Lambda^\alpha\left(g_R\langle \xx\rangle^{\beta} \vv_{\beta,R}\right)\,\mathrm{d}\xx-\frac{2-\alpha}{2\alpha}\left\|  \vv_{\beta,R}\right\|^2_{L^2(\RR^3)} \\&-\frac{1}{2\alpha}\int_{\RR^3}\xx\cdot\nabla\vv\left(g^2_R\langle \xx\rangle^{2\beta}\vv\right)\,\mathrm{d}\xx+\int_{\RR^3}\uu\cdot\nabla\vv\left(g^2_R\langle \xx\rangle^{2\beta}\vv\right)\,\mathrm{d}\xx\\
=& \int_{\RR^3} \ff_R\langle \xx\rangle^{2\beta} \vv_R\,\mathrm{d}\xx-\int_{\RR^3} g_R  \langle \xx\rangle^{\beta}\nabla P  \vv_{\beta,R}\,\mathrm{d}\xx,
\end{split}
\end{equation}
where and what in follows, we denote $f_{\beta,R}(\xx)=g_R\langle \xx\rangle^{ \beta}f(\xx).$

Firstly, we rewrite the first term in the left side of the above equality as
\begin{align*}
&\int_{\RR^3}\Lambda^\alpha\vv \cdot\Lambda^\alpha\left(g_R\langle \xx\rangle^{\beta} \vv_{\beta,R}\right)\,\mathrm{d}\xx\\
=&\int_{\RR^3}\left(g_R\langle \xx\rangle^{\beta}\right)\Lambda^\alpha\vv\cdot\Lambda^\alpha \vv_{\beta,R} \,\mathrm{d}\xx+\int_{\RR^3}\left[\Lambda^\alpha,g_R\langle \xx\rangle^{\beta}\right]\Lambda^\alpha\vv   \cdot\vv_{\beta,R}\,\mathrm{d}\xx\\
=& \left\|\Lambda^\alpha \vv_{\beta,R}\right\|_{L^2(\RR^3)}^2+\int_{\RR^3}\left[\Lambda^\alpha,g_R\langle \xx\rangle^{\beta}\right]\vv\cdot \Lambda^\alpha\vv_{\beta,R} \,\mathrm{d}\xx\\
 &+\int_{\RR^3}\left[\Lambda^\alpha,g_R\langle \xx\rangle^{\beta}\right]\Lambda^\alpha\vv \cdot  \vv_{\beta,R} \,\mathrm{d}\xx.
\end{align*}
By the Cauchy-Schwarz inequality and Lemma \ref{lem-Comm}, we find that for each $\beta<\alpha,$
\begin{align*}
 \int_{\RR^3}\left[\Lambda^\alpha,g_R\langle \xx\rangle^{\beta}\right]\vv\cdot \Lambda^\beta \vv_{\beta,R} \,\mathrm{d}\xx
\leq&\left\|\left[\Lambda^\alpha,g_R\langle \cdot\rangle^{\beta}\right]\vv\right\|_{L^2(\RR^3)}\left\|\Lambda^\alpha \left(\langle \cdot\rangle^{\beta}\vv_R\right)\right\|_{L^2(\RR^3)} \\
\leq &C\|\vv\|_{L^2(\RR^3)}\left\| \vv_{\beta,R} \right\|_{\dot{H}^{\alpha}(\RR^3)}.
\end{align*}
Also, we have that for each $\beta<\alpha,$
\begin{align*}
 \int_{\RR^3}\left[\Lambda^\alpha,g_R\langle \xx\rangle^{\beta}\right]\Lambda^\alpha\vv\cdot\vv_{\beta,R} \,\mathrm{d}\xx
\leq&\left\|\left[\Lambda^\alpha,g_R\langle \cdot\rangle^{\beta}\right]\Lambda^\alpha\vv \right\|_{L^2(\RR^3)}\left\| \vv_{\beta,R}\right\|_{L^2(\RR^3)}\\
\leq& C\left\|\vv\right\|_{\dot{H}^{\alpha}(\RR^3)}\left\|\vv_{\beta,R}\right\|_{L^2(\RR^3)}.
\end{align*}
Thus, we have
\begin{equation}\label{eq-L2-weight-2}
\int_{\RR^3}\Lambda^\alpha\vv \cdot\Lambda^\alpha\left(g_R\langle \xx\rangle^{2\beta} \vv_R\right)\,\mathrm{d}\xx\geq\left\|\Lambda^\alpha \vv_{\beta,R}\right\|_{L^2(\RR^3)}^2-C\|\vv\|_{H^\alpha(\RR^3)}\left\| \vv_{\beta,R} \right\|_{ {H}^{\alpha}(\RR^3)}.
\end{equation}
Noting that
\begin{align*}
&-\frac{2\alpha-1}{2\alpha}\left\|  \vv_{\beta,R}\right\|^2_{L^2(\RR^3)}-\frac{1}{2\alpha}\int_{\RR^3}\xx\cdot\nabla\vv g^2_R\langle \xx\rangle^{2\beta}\vv\,\mathrm{d}\xx\\
=&-\frac{2\alpha-1}{2\alpha}\left\|  \vv_{\beta,R}\right\|^2_{L^2(\RR^3)}-\frac{1}{4\alpha}\int_{\RR^3}\xx\cdot\nabla\left|\vv\right|^2\left(g^2_R\langle \xx\rangle^{2\beta}\right) \,\mathrm{d}\xx,
\end{align*}
one has by the integrating by parts that
\begin{align*}
&-\frac{2\alpha-1}{2\alpha}\left\|  \vv_{\beta,R}\right\|^2_{L^2(\RR^3)}-\frac{1}{2\alpha}\int_{\RR^3}\xx\cdot\nabla\vv g^2_R\langle \xx\rangle^{2\beta}\vv\,\mathrm{d}\xx\\
=&\frac{5-4\alpha}{4\alpha}\left\|\langle \cdot\rangle^{\beta} \vv_R\right\|^2_{L^2(\RR^3)}+
 \frac{1}{4\alpha}\int_{\RR^3}\xx\cdot\nabla  \langle \xx\rangle^{2\beta} |\vv_R|^2\,\mathrm{d}\xx+\frac{1}{4\alpha}\int_{\RR^3}\xx\cdot\nabla g^2_R \langle \xx\rangle^{2\beta} |\vv|^2 \,\mathrm{d}\xx.
\end{align*}
A simple calculation yields
\begin{align*}
\frac{1}{4\alpha}\int_{\RR^3}\xx\cdot\nabla  \langle \xx\rangle^{2\beta} \vv^2_R\,\mathrm{d}\xx=&\frac{\beta}{2\alpha}\int_{\RR^3} |\xx|^2 \langle \xx\rangle^{2(\beta-1)}\vv^2_R\,\mathrm{d}\xx\\
=&\frac{\beta}{2\alpha}\left\| \vv_{\beta,R}\right\|^2_{L^2(\RR^3)}-\frac{e\beta}{2\alpha}\int_{\RR^3} \langle \xx\rangle^{2(\beta-1)}\vv^2_R\,\mathrm{d}\xx,
\end{align*}
and
\begin{align*}
\frac{1}{4\alpha}\int_{\RR^3}\xx\cdot\nabla g^2_R \langle \xx\rangle^{2\beta} \vv^2 \,\mathrm{d}\xx=&-\frac{1+2\beta}{4\alpha}\int_{\RR^3} \frac{\frac{|\xx|^2}{R^2}}{1+\frac{|\xx|^2}{R^2}} \vv^2_{\beta,R} \,\mathrm{d}\xx
\geq-\frac{1+2\beta}{4\alpha}\left\|\vv_{\beta,R}\right\|^2_{L^2(\RR^3)}.
\end{align*}
Therefore, we have
\begin{equation}\label{eq-L2-weight-3}
\begin{split}
&-\frac{2\alpha-1}{2\alpha}\left\|  \vv_{\beta,R}\right\|^2_{L^2(\RR^3)}+\frac{1}{2\alpha}\int_{\RR^3}\xx\cdot\nabla\vv\varphi^2_R\langle \xx\rangle^{2\beta}\vv\,\mathrm{d}\xx\\
\geq&\frac{1- \alpha}{ \alpha}\left\|\langle \cdot\rangle^{\beta} \vv_R\right\|^2_{L^2(\RR^3)}-\frac{e\beta}{2\alpha}\int_{\RR^3} \langle \xx\rangle^{2(\beta-1)}\vv^2_R\,\mathrm{d}\xx.
\end{split}
\end{equation}
Integrating by parts, one obtains
\begin{align*}
-\int_{\RR^3}\uu\cdot\nabla\vv\left(g^2_R\langle \xx\rangle^{2\beta}\vv\right)\,\mathrm{d}\xx=&\int_{\RR^3}\uu\cdot\nabla g^2_R\left(\langle \xx\rangle^{2\beta}|\vv|^2\right)\,\mathrm{d}\xx\\
&+\int_{\RR^3}\uu\cdot\nabla \langle \xx\rangle^{2\beta}\left(g^2_R|\vv|^2\right)\,\mathrm{d}\xx.
\end{align*}
On one hand, we see that
\begin{align*}
\int_{\RR^3}\uu\cdot\nabla g^2_R\left(\langle \xx\rangle^{2\beta}|\vv|^2\right)\,\mathrm{d}\xx=&-(1+2\beta)\int_{\RR^3} \frac{1}{1+\frac{|\xx|^2}{R^2}}\frac{ \xx\cdot\uu }{R^2} \vv^2_{\beta,R} \,\mathrm{d}\xx\\
\leq&\frac{1+2\beta}{2R}\|\uu\|_{L^\infty(\RR^3)}\|\vv_{\beta,R}\|_{L^2(\RR^3)}^2.
\end{align*}
On the other hand, we get by using $\beta\in(0,1)$ that
\begin{align*}
\int_{\RR^3}\uu\cdot\nabla \langle \xx\rangle^{2\beta}\left(g^2_R|\vv|^2\right)\,\mathrm{d}\xx=&2\beta\int_{\RR^3}(\uu\cdot\xx) \langle \xx\rangle^{2(\beta-1)}\left(g^2_R|\vv|^2\right)\,\mathrm{d}\xx\\
\leq&2\beta\|\uu\|_{L^\infty(\RR^3)}\|\vv\|_{L^2(\RR^3)} \|\vv_{\beta,R}\|_{L^2(\RR^3)}.
\end{align*}
Thus we have
\begin{equation}\label{eq-convec}
-\int_{\RR^3}\uu\cdot\nabla\vv\left(g^2_R\langle \xx\rangle^{2\beta}\vv\right)\,\mathrm{d}\xx\leq \frac{C}{R}\|\vv_{\beta,R}\|_{L^2(\RR^3)}^2+\|\vv\|_{L^2(\RR^3)} \|\vv_{\beta,R}\|_{L^2(\RR^3)}.
\end{equation}
By the Cauchy-Schwarz inequality, one has
\begin{equation}\label{eq-L2-weight-4}
 \int_{\RR^3}\ff_{1R}\langle \xx\rangle^{2\beta} \vv_R\,\mathrm{d}\xx
\leq \left\|\langle \cdot\rangle^{ \beta}\ff_{1}\right\|_{L^2(\RR^3)}\left\|\langle \cdot\rangle^{ \beta}\vv_R\right\|_{L^2(\RR^3)}.
\end{equation}
While
\begin{align*}
 \int_{\RR^3}\Div\ff_{2R}\langle \xx\rangle^{2\beta} \vv_R\,\mathrm{d}\xx
\leq& \left\|\langle \cdot\rangle^{ \beta}\Div\ff_{2}\right\|_{H^{-\alpha}(\RR^3)}\left\|g_R\langle \cdot\rangle^{ \beta}\vv_R\right\|_{H^\alpha(\RR^3)}\\
\leq&C\left\|g_R\right\|_{W^{\alpha,\infty}(\RR^3)} \left\|\langle \cdot\rangle^{ \beta}\Div\ff_{2}\right\|_{H^{-\alpha}(\RR^3)}\left\|\langle \cdot\rangle^{ \beta}\vv_R\right\|_{ {H}^\alpha(\RR^3)},
\end{align*}
and this estimate  implies
\begin{equation*}
\int_{\RR^3}\Div\ff_{2R}\langle \xx\rangle^{2\beta} \vv_R\,\mathrm{d}\xx\leq C\left\|\langle \cdot\rangle^{ \beta}\Div\ff_{2}\right\|_{\dot{H}^{-\alpha}(\RR^3)}\left\|\langle \cdot\rangle^{ \beta}\vv_R\right\|_{ {H}^\alpha(\RR^3)}.
\end{equation*}
Recall that
\begin{equation}\label{eq-pres-2}
-\Delta P=-\operatorname{div}\left(\uu\cdot\nabla\vv-\ff_1-\Div\ff_2\right)=:-\Delta P_1 -\Delta P_2-\Delta P_3.
\end{equation}
Since $\beta\in(0,1)$, the weight $|\xx|^{2 \beta}$ belongs to the $A_2$-class, and then we obtain
\begin{equation}\label{eq-L2-weight-5}
\begin{split}
-\int_{\RR^3} g_R  \langle \xx\rangle^{\beta}\nabla P_2  \vv_{\beta,R}\,\mathrm{d}\xx\leq &\left\|\langle \cdot\rangle^{\beta}\nabla P_2\right\|_{L^2(\RR^3)}\left\|\vv_{\beta,R}\right\|_{L^2(\RR^3)}\\
\leq &C\left\|\langle \cdot\rangle^{\beta}\ff_1\right\|_{L^2(\RR^3)}\left\|\vv_{\beta,R}\right\|_{L^2(\RR^3)}.
\end{split}
\end{equation}
We have by the integration by parts that
\begin{align*}
-\int_{\RR^3} g_R  \langle \xx\rangle^{\beta}\nabla P_1  \vv_{\beta,R}\,\mathrm{d}\xx=& \int_{\RR^3}\left(\vv\cdot \nabla g^2_R \right) \langle \xx\rangle^{2\beta} P_1  \,\mathrm{d}\xx+\int_{\RR^3}\left(\vv\cdot \nabla \langle \xx\rangle^{2\beta}\right)g^2_R P_1  \,\mathrm{d}\xx.
\end{align*}
By the H\"older inequality and the fact the weight $|\xx|^{ 2\beta}$ belongs to $A_2$-class, we have
\begin{align*}
\int_{\RR^3}\left(\vv\cdot \nabla g^2_R \right) \langle \xx\rangle^{2\beta} P_1  \,\mathrm{d}\xx=&-(1+2\beta)\int_{\RR^3} \frac{1}{1+\frac{|\xx|^2}{R^2}}\frac{ \xx\cdot\vv }{R^2} \langle \xx\rangle^{2\beta}P_1 \,\mathrm{d}\xx\\
\leq&\frac{C}{R^{1-\beta}}\left\|\vv_{\beta,R}\right\|_{L^2(\RR^3)}\|P_1\|_{L^2(\RR^3)}\\
\leq& \frac{C}{R^{1-\beta}}\|\uu\|_{L^\infty(\RR^3)}\left\|\vv_{\beta,R}\right\|_{L^2(\RR^3)}\|\vv\|_{L^2(\RR^3)}
\end{align*}
and
\begin{align*}
\int_{\RR^3}\left(\vv\cdot \nabla \langle \xx\rangle^{2\beta}\right)g^2_R P_1  \,\mathrm{d}\xx=&2\beta\int_{\RR^3}(\vv\cdot\xx) \langle \xx\rangle^{2(\beta-1)}\left(g^2_RP_1\right)\,\mathrm{d}\xx\\
\leq& C \left\|\vv_{\beta,R}\right\|_{L^2(\RR^3)}\|P_1\|_{L^2(\RR^3)}\\
\leq&  C \|\uu\|_{L^\infty(\RR^3)}\left\|\vv_{\beta,R}\right\|_{L^2(\RR^3)}\|\vv\|_{L^2(\RR^3)}.
\end{align*}
Therefore
\begin{equation}\label{eq-L2-weight-55}
-\int_{\RR^3} g_R  \langle \xx\rangle^{\beta}\nabla P_1  \vv_{\beta,R}\,\mathrm{d}\xx\leq C \|\uu\|_{L^\infty(\RR^3)}\left\|\vv_{\beta,R}\right\|_{L^2(\RR^3)}\|\vv\|_{L^2(\RR^3)}.
\end{equation}
In the same way as in the proof of \eqref{eq-L2-weight-55}, the term involving $P_3$ can be bounded as follows:
\begin{equation}\label{eq-L2-weight-55-3}
\begin{split}
-\int_{\RR^3} g_R  \langle \xx\rangle^{\beta}\nabla P_3  \vv_{\beta,R}\,\mathrm{d}\xx\leq& C \left\|\vv_{\beta,R}\right\|_{L^2(\RR^3)}\|P_3\|_{L^2(\RR^3)}\\
\leq& C \left\|\vv_{\beta,R}\right\|_{L^2(\RR^3)}\|\ff_2\|_{L^2(\RR^3)}.
\end{split}
\end{equation}
 Inserting estimates \eqref{eq-L2-weight-2}-\eqref{eq-L2-weight-55-3} into the equality \eqref{eq-L2-weight-1}, we readily get that for large $R$ and $\forall\,\varepsilon>0,$
 \begin{equation}\label{eq-L2-weight-6}
 \begin{split}
& \left\|\Lambda^\alpha \vv_{\beta,R}\right\|_{L^2(\RR^3)}^2+\frac{1-\alpha}{\alpha}\left\|\langle \cdot\rangle^{\beta} \vv_R\right\|^2_{L^2(\RR^3)}\\
 \leq&  \left\|\langle \cdot\rangle^{ \beta}\ff_R\right\|_{L^2(\RR^3)}\left\|\langle \cdot\rangle^{ \beta}\vv_R\right\|_{L^2(\RR^3)}  +C\|\vv\|_{H^\alpha(\RR^3)}\left\| \vv_{\beta,R} \right\|_{ {H}^{\alpha}(\RR^3)}
 +\frac{e\beta}{2\alpha}\|\vv_R\|^2_{L^2(\RR^3)} \\
 \leq &C\left(\|\vv\|^2_{H^\alpha(\RR^3)}+ \left\|\langle \cdot\rangle^{\beta}\ff_1\right\|^2_{L^2(\RR^3)}+\left\|\langle \cdot\rangle^{ \beta}\Div\ff_{2}\right\|^2_{H^{-\alpha}(\RR^3)}+ \|\ff_2\|^2_{L^2(\RR^3)}\right)\\&+ \varepsilon\left\| \vv_{\beta,R} \right\|^2_{ {H}^{\alpha}(\RR^3)}+\frac{C}{R} \left\| \vv_{\beta,R} \right\|^2_{ {H}^{\alpha}(\RR^3)}.
 \end{split}
 \end{equation}
Since $\beta\in(0,1)$, we know that  the coefficient $\frac{1-\alpha}{ \alpha}>0$ and $g_R\langle\xx\rangle^{\beta}\vv $ meets the requirement of test function \eqref{eq-L2-weight-1}. Moreover, \eqref{eq-L2-weight-6} becomes
 \begin{align*}
  &\left\|\Lambda^\alpha  \vv_{\beta,R}\right\|_{L^2(\RR^3)}^2+\left\|\langle \cdot\rangle^{\beta} \vv_R\right\|^2_{L^2(\RR^3)}\\
  \leq  &C\left(\|\vv\|^2_{H^\alpha(\RR^3)}+ \left\|\langle \cdot\rangle^{\beta}\ff_1\right\|^2_{L^2(\RR^3)}+\left\|\langle \cdot\rangle^{ \beta}\Div\ff_{2}\right\|^2_{ H^{-\alpha}(\RR^3)}+ \|\ff_2\|^2_{L^2(\RR^3)}\right).
  \end{align*}
According to this estimate  and the following estimate from Theorem \ref{thm-V-H}
\begin{equation*}
\left\|\vv\right\|_{H^{ 2\alpha} (\RR^3)} \leq C \big(\left\|\ff\right\|_{L^2(\RR^3)} \big),
\end{equation*}
 we can show the desired estimate \eqref{eq-V-H-weight}  by performing the above process again and sending $R\to\infty.$
\end{proof}
\begin{theorem}\label{thm-V-H-L2-weight-22}
Let $s>0$, $\alpha\in(0,1]$,  $\ff_1\in H^s(\RR^3)$  and $\ff_2\in H^{s+1}(\RR^3)$ and $\uu\equiv 0$.
Assume that    $(\vv,P)\in H^{s+1}(\RR^3)\times H^{s+1}(\RR^3)$ is a weak solution to the system \eqref{E-L}.
\begin{itemize}
  \item[(i)]  When $\alpha\in(0,1 )$, if   $$\sup_{|\gamma|=[s]} \left( \left\|\langle \cdot\rangle^{\beta}(I_{d}+\Lambda^{\{s\}}D^\gamma)\ff_1\right\|_{L^2(\RR^3)}
 +   \left\|\langle \cdot\rangle^{\beta}\Lambda^{\{s\}}D^\gamma\Div\ff_2\right\|_{H^{-\alpha}(\RR^3)}\right)<\infty,\,\,\,\forall\beta\in(0,\alpha),$$  we have
 \begin{equation}\label{eq-V-H-weight-h-s}
 \begin{split}
 \sup_{|\gamma|=[s]} \left\|\langle \cdot\rangle^{\beta} \Lambda^{\{s\}} D^\gamma\vv\right\| _{H^\alpha(\RR^3)}
   \leq &C\left\|\ff_2\right\|_{H^{1+s}(\RR^3)}+C   \sup_{|\gamma|=[s]}  \left\|(I_{d}+\Lambda^{\{s\}}D^\gamma)\ff_1\right\|_{L^2_{\langle \cdot\rangle^{2\beta}}(\RR^3)}\\
&+C \sup_{|\gamma|=[s]} \left\|\langle \cdot\rangle^{\beta}\Lambda^{\{s\}}D^\gamma\Div\ff_2\right\|_{H^{-\alpha}(\RR^3)}.
\end{split}
\end{equation}
  \item  [(ii)]  When $\alpha=1$, if   $$\sup_{|\gamma|=[s]} \left( \left\|\langle \cdot\rangle (I_{d}+\Lambda^{\{s\}}D^\gamma)\ff_1\right\|_{L^2(\RR^3)}
 +   \left\|\langle \cdot\rangle\Lambda^{\{s\}}D^\gamma\Div\ff_2\right\|_{H^{-1}(\RR^3)}\right)<\infty,$$
 we have
 \begin{equation}\label{eq-V-H-weight-h-s-delta}
 \begin{split}
\sup_{|\gamma|=[s]} \left\|\langle \cdot\rangle  \Lambda^{\{s\}} D^\gamma\vv\right\| _{H^1(\RR^3)}
   \leq &C\left\|\ff_2\right\|_{H^{1+s}(\RR^3)}+C   \sup_{|\gamma|=[s]}  \left\|(I_{d}+\Lambda^{\{s\}}D^\gamma)\ff_1\right\|_{L^2_{\langle \cdot\rangle^2}(\RR^3)}\\
&+C \sup_{|\gamma|=[s]} \left\|\langle \cdot\rangle \Lambda^{\{s\}}D^\gamma\Div\ff_2\right\|_{H^{-1}(\RR^3)}.
\end{split}
\end{equation}
\end{itemize}
\end{theorem}
\begin{proof}
First of all, we consider the following   approximate system in the whole space
\begin{align}\label{E-L-Appr}
\left\{
\begin{aligned}
& (-\Delta)^{\alpha}\vv_N- \frac{2\alpha-1}{2\alpha}\vv_N-\frac{1}{2\alpha}\xx\cdot \nabla \vv_N
+\nabla P_N=S_N\ff, \\
& \textnormal{div}\,\vv_N=0.
\end{aligned}\ \right.
\end{align}
Since $\ff\in H^s(\RR^3)$, we know that $ S_N\ff\in H^\infty(\RR^3):= \cap_{s\geq0} H^s(\RR^3)$ belongs to the Schwarz class. Moreover, following the proof of Theorem \ref{thm-V-H},  we can get $\vv_N\in H^\infty(\RR^3)$ is a smooth solution to the system \eqref{E-L-Appr}. Performing
differential operator $D^\gamma$ with $|\gamma|=[s]$ to the approximate system yields
\begin{equation}\label{E-L-Appr-diff}
\left\{
\begin{aligned}
 &(-\Delta)^{\alpha}D^\gamma\vv_N- \frac{2\alpha-1}{2\alpha}D^\gamma\vv_N-\frac{1}{2\alpha}\xx\cdot \nabla D^\gamma\vv_N
+\nabla D^\gamma P_N\\
=&D^\gamma S_N\ff
+\sum_{|\gamma_1|=[s]-1}C_{\gamma}^{\gamma_1}(D\xx)\cdot\nabla \partial^{\gamma_1}\vv_N,\\
 &\operatorname{div}D^\gamma\vv_N=0.
\end{aligned}\right.
\end{equation}
Denoting $\{s\}=s-[s]$, one easily verifies that
\begin{align}\label{E-L-Appr-diff-Lam}
\left\{
\begin{aligned}
 &(-\Delta)^{\alpha}\Lambda^{\{s\}}D^\gamma\vv_N- \frac{2\alpha-1}{2\alpha}\Lambda^{\{s\}}D^\gamma\vv_N-\frac{1}{2\alpha}\Lambda^{\{s\}}\left(\xx\cdot \nabla D^\gamma \vv_N\right)    \\
=&-\nabla \Lambda^{\{s\}}D^\gamma P_N+\Lambda^{\{s\}}D^\gamma S_N\ff+\sum_{|\gamma_1|=[s]-1}C_{\gamma}^{\gamma_1}\Lambda^{\{s\}}\big((D\xx)\cdot\nabla \partial^{\gamma_1} \vv_N\big),\\
 &\operatorname{div}\Lambda^{\{s\}}D^\gamma\vv_N=0.
\end{aligned}\right.
\end{align}
In terms of Fourier transform, we easily find that
\begin{align*}
\mathscr{F}\left(\Lambda^{\{s\}}\left(\xx\cdot \nabla D^\gamma \vv_N\right)\right)=&|\xi|^{\{s\}}\operatorname{div}_\xi\left(\xi \mathscr{F}(D^\gamma \vv_N)(\xi)\right)\\
=&\operatorname{div}_\xi\left(\xi |\xi|^{\{s\}}\mathscr{F}(D^\gamma \vv_N)(\xi)\right)-\{s\}|\xi|^{\{s\}}\mathscr{F}(D^\gamma \vv_N)(\xi).
\end{align*}
Plugging this equality into \eqref{E-L-Appr-diff-Lam} leads to
\begin{align}\label{E-L-Appr-diff-Lam-2}
\left\{
\begin{aligned}
 &(-\Delta)^{\alpha}\Lambda^{\{s\}}D^\gamma\vv_N- \frac{2\alpha-1}{2\alpha}\Lambda^{\{s\}}D^\gamma\vv_N-\frac{1}{2\alpha}\xx\cdot \nabla\Lambda^{\{s\}} D^\gamma \vv_N
 +\nabla \Lambda^{\{s\}}D^\gamma P_N\\
=&\Lambda^{\{s\}}D^\gamma S_N\ff +\sum_{|\gamma_1|=[s]-1}C_{\gamma}^{\gamma_1}\Lambda^{\{s\}}\left((D\xx)\cdot\nabla \partial^{\gamma_1} \vv_N\right)+\frac{\{s\}}{2\alpha}\Lambda^{\{s\}}D^\gamma\vv_N,\\
 &\operatorname{div}\Lambda^{\{s\}}D^\gamma\vv_N=0.
\end{aligned}\right.
\end{align}
In particular, when $s\in(0,1),$ problem \eqref{E-L-Appr-diff-Lam-2}  can be reduced to
\begin{align}\label{E-L-Appr-diff-Lam-2-special}
\left\{
\begin{aligned}
 &(-\Delta)^{\alpha}\Lambda^{s}\vv_N- \frac{2\alpha-1}{2\alpha}\Lambda^{s}\vv_N-\frac{1}{2\alpha}\xx\cdot \nabla\Lambda^{s}  \vv_N
 +\nabla \Lambda^{s} P_N
\\=&\Lambda^{s} S_N\ff +\frac{s}{2\alpha}\Lambda^{s}\vv_N,\\
 &\operatorname{div}\Lambda^{s}\vv_N=0.
\end{aligned}\right.
\end{align}
Furthermore, we apply \eqref{eq-V-H-weight} to $\Lambda^{s}\vv_N$ in \eqref{E-L-Appr-diff-Lam-2-special} to get
\begin{equation}\label{eq-V-H-weight-2-small}
\begin{aligned}
\left\|\langle \cdot\rangle^{\beta} \Lambda^{s}\vv_N\right\| _{H^\alpha(\RR^3)}
  \leq&  C \|\Lambda^{s}\vv_N\|_{H^\alpha(\RR^3)}+ C\left\|\langle \cdot\rangle^{\beta}S_N\Lambda^{s}\ff_1\right\|_{L^2(\RR^3)}+C\|\Lambda^s\ff_2\|_{L^2(\RR^3)}\\
 &+ C\left\|\langle \cdot\rangle^{\beta}S_N\Div\Lambda^{s}\ff_2\right\|_{H^{-\alpha}(\RR^3)}+C\left\|\langle \cdot\rangle^{\beta}\Lambda^{s}\vv_N\right\|_{L^2(\RR^3)}\\
& + C\left\|\langle \cdot\rangle^{\beta}\left[\Lambda^{ s },\uu\cdot\nabla\right] \vv_N\right\|_{L^2(\RR^3)}.
  \end{aligned}
\end{equation}
Since
\begin{align*}
\big|| \xx |^{\beta}S_N\Lambda^{s}\ff_1\big|=&\left|| \xx |^{\beta}\int_{\RR^3}{h}_N(\xx-\yy)\left( \Lambda^{s}\ff_1\right)(\yy)\,\mathrm{d}\yy\right|\\
\leq& \int_{\RR^3}| \xx-\yy |^{\beta}\left| {h}_N\right|(\xx-\yy)\left| \Lambda^{s}\ff_1\right|(\yy)\,\mathrm{d}\yy\\
& +\int_{\RR^3}\left|h_N\right|(\xx-\yy)| \yy |^{\beta}\left|\Lambda^{s}\ff_1\right|(\yy)\,\mathrm{d}\yy,
\end{align*}
we have
\begin{equation}\label{eq-w-22}
\left\|\langle \cdot\rangle^{\beta}S_N\Lambda^{s}\ff_1\right\|_{L^2(\RR^3)}\leq C2^{-\beta N}\left\| \Lambda^{s}\ff_1\right\|_{L^2(\RR^3)}+C\left\|| \cdot|^{\beta}\Lambda^{s}\ff_1\right\|_{L^2(\RR^3)}.
\end{equation}
Since $$S_N\Div\Lambda^{s}\ff_2= \Div\Lambda^{s}\ff_2+\sum_{k\geq N}\dot{\Delta}_k\Div\Lambda^{s}\ff_2,$$ we have by the triangle inequality that
\begin{align*}
\left\|\langle \cdot\rangle^{\beta}S_N\Div\Lambda^{s}\ff_2\right\|_{H^{-\alpha}(\RR^3)}\leq&\left\|\langle \cdot\rangle^{\beta}\Div\Lambda^{s}\ff_2\right\|_{ H^{-\alpha}(\RR^3)}+\sum_{k\geq N}\left\|\langle \cdot\rangle^{\beta}\dot{\Delta}_k\Div\Lambda^{s}\ff_2\right\|_{H^{-\alpha}(\RR^3)}.
\end{align*}
Furthermore, we calculate
\begin{align*}
&\sum_{k\geq N}\left\|\langle \cdot\rangle^{\beta}\dot{\Delta}_k\Div\Lambda^{s}\ff_2\right\|_{ H^{-\alpha}(\RR^3)} \\
\leq&\sum_{k\geq N}\left\|\dot{\Delta}_k\langle \cdot\rangle^{\beta}\Div\Lambda^{s}\ff_2\right\|_{ H^{-\alpha}(\RR^3)}
+\sum_{k\geq N}\left\|[\dot{\Delta}_k,\langle \cdot\rangle^{\beta}]\Div\Lambda^{s}\ff_2\right\|_{H^{-\alpha}(\RR^3)}\\
\leq& C\left\|\langle \cdot\rangle^{\beta}\Div\Lambda^{s}\ff_2\right\|_{ H^{-\alpha}(\RR^3)} +\sup_{k\geq N}\left\|[\dot{\Delta}_k,\langle \cdot\rangle^{\beta}]\Div\Lambda^{s}\ff_2\right\|_{L^{2}(\RR^3)}\\
\leq& C\left\|\langle \cdot\rangle^{\beta}\Div\Lambda^{s}\ff_2\right\|_{ H^{-\alpha}(\RR^3)} +C\sup_{k\geq N}2^{-k}\left\| \Div\Lambda^{s}\ff_2\right\|_{L^{2}(\RR^3)}.
\end{align*}
Thus, it follows that
\begin{equation}\label{eq-w-22-3}
\left\|\langle \cdot\rangle^{\beta}S_N\Div\Lambda^{s}\ff_2\right\|_{H^{-\alpha}(\RR^3)}\leq C\left\|\langle \cdot\rangle^{\beta}\Div\Lambda^{s}\ff_2\right\|_{ H^{-\alpha}(\RR^3)} +C \left\| \Div\Lambda^{s}\ff_2\right\|_{L^{2}(\RR^3)}.
\end{equation}
Inserting \eqref{eq-w-22} and \eqref{eq-w-22-3} into \eqref{eq-V-H-weight-2-small} yields
\begin{equation}\label{eq-V-H-weight-2-small-22}
\begin{aligned}
   \left\|\langle \cdot\rangle^{\beta} \Lambda^{s}\vv_N\right\| _{H^\alpha(\RR^3)}
  \leq&   C \|\Lambda^{s}\vv_N\|_{H^\alpha(\RR^3)}+ C\left\|\langle \cdot\rangle^{\beta} \Lambda^{s}\ff_1\right\|_{L^2(\RR^3)}+C\|\Lambda^s\ff_2\|_{H^1(\RR^3)}\\
 &+ C\left\|\langle \cdot\rangle^{\beta} \Div\Lambda^{s}\ff_2\right\|_{H^{-\alpha}(\RR^3)}+C\left\|\langle \cdot\rangle^{\beta}\Lambda^{s}\vv_N\right\|_{L^2(\RR^3)}.
  \end{aligned}
\end{equation}
Since $\ff_1\in H^s(\RR^3)$ and $\ff_2\in H^{s+1}(\RR^3)$, we get from Theorem \ref{thm-V-H} and Theorem \ref{thm-V-H-L2-weight} that
\begin{equation}\label{eq-V-H-weight-22}
\left\|\vv_N\right\| _{H^{s+2\alpha}(\RR^3)}+\left\| \nabla P_N\right\| _{H^s(\RR^3)}
  \leq  C\|\ff_1\|_{H^s(\RR^3)}+ C\|\ff_2\|_{H^{s+1}(\RR^3)}.
\end{equation}
and
\begin{equation}\label{eq-V-H-weight-2233}
\begin{split}
 & \left\|\langle \cdot\rangle^{\beta} \vv_N\right\| _{H^\alpha(\RR^3)}\\
  \leq&  C\left(\|\vv_N\|_{H^\alpha(\RR^3)}+ \left\|\langle \cdot\rangle^{\beta}\ff_1\right\|_{L^2(\RR^3)}+\left\|\langle \cdot\rangle^{ \beta}\Div\ff_{2}\right\|_{H^{-\alpha}(\RR^3)}+ \|\ff_2\|_{L^2(\RR^3)}\right).
  \end{split}
\end{equation}
For each $s\in(0,1),$ we have by Lemma \ref{lem-com-two} that
\begin{equation}\label{eq-w-23}
\left\|\langle \cdot\rangle^{\beta}\Lambda^{s}\vv_N\right\|_{L^2(\RR^3)}\leq C\left\|\Lambda^{s}\left(\langle \cdot\rangle^{\beta}\vv_N\right)\right\|_{L^2(\RR^3)}+C\left\|\langle \cdot\rangle^{\beta}\vv_N\right\|_{L^2(\RR^3)}.
\end{equation}
Inserting \eqref{eq-V-H-weight-22}, \eqref{eq-V-H-weight-2233} and \eqref{eq-w-23} into \eqref{eq-V-H-weight-2-small-22}, we obtain that for each $s\in[0,\alpha],$
\begin{equation}\label{eq-V-H-weight-2-small-22-33}
\begin{split}
 \left\|\langle \cdot\rangle^{\beta} \Lambda^{s}\vv_N\right\| _{H^\alpha(\RR^3)}
  \leq&  C \left\|\langle \cdot\rangle^{\beta}(I_{d}+\Lambda^{s})\ff_1\right\|_{L^2(\RR^3)}+C\|\Lambda^{s}\ff_2\|_{H^1(\RR^3)}\\
   &+ C \left\|\langle \cdot\rangle^{\beta}(I_{d}+\Lambda^{s})\Div\ff_2\right\|_{H^{-\alpha}(\RR^3)}.
  \end{split}
\end{equation}
For each $s\in(\alpha,\min\{1,2\alpha\}),$ we have by Lemma \ref{lem-com-two} again that
\begin{equation}\label{eq-w-23-2}
\left\|\langle \cdot\rangle^{\beta}\Lambda^{s}\vv_N\right\|_{L^2(\RR^3)}\leq C\left\|\Lambda^{s-\alpha}\left(\langle \cdot\rangle^{\beta}\Lambda^{\alpha}\vv_N\right)\right\|_{L^2(\RR^3)}+C\left\|\langle \cdot\rangle^{\beta}\vv_N\right\|_{L^2(\RR^3)}.
\end{equation}
Plugging \eqref{eq-V-H-weight-22}, \eqref{eq-V-H-weight-2-small-22-33} and \eqref{eq-w-23-2} into \eqref{eq-V-H-weight-2-small-22} leads to that
for each $s\in(\alpha,\min\{1,2\alpha\}),$
\begin{equation*}
\begin{split}
 \left\|\langle \cdot\rangle^{\beta} \Lambda^{s}\vv_N\right\| _{H^\alpha(\RR^3)}
 \leq&  C \left\|\langle \cdot\rangle^{\beta}(I_{d}+\Lambda^{s})\ff_1\right\|_{L^2(\RR^3)}+C\|\Lambda^{s}\ff_2\|_{H^1(\RR^3)}\\
   &+ C \left\|\langle \cdot\rangle^{\beta}(I_{d}+\Lambda^{s})\Div\ff_2\right\|_{H^{-\alpha}(\RR^3)}.
  \end{split}
\end{equation*}
Proceeding step by step, we can show that for each $s\in(0,1),$
\begin{equation}\label{eq-V-H-weight-2-small-22-33-stepf}
\begin{split}
 \left\|\langle \cdot\rangle^{\beta} \Lambda^{s}\vv_N\right\| _{H^\alpha(\RR^3)}
  \leq & C \left\|\langle \cdot\rangle^{\beta}(I_{d}+\Lambda^{s})\ff_1\right\|_{L^2(\RR^3)}+C\|\Lambda^{s}\ff_2\|_{H^1(\RR^3)}\\
   &+ C \left\|\langle \cdot\rangle^{\beta}(I_{d}+\Lambda^{s})\Div\ff_2\right\|_{H^{-\alpha}(\RR^3)}.
   \end{split}
\end{equation}
Next, we consider the case where  $s\geq1$. Suppose there exists an integer $k\in\NN^+$ such that for $0\leq s<k,$
\begin{equation}\label{eq-2222}
\begin{split}
 \sup_{|\gamma|=[s]} \left\|\langle \cdot\rangle^{\beta} \Lambda^{\{s\}}D^{\gamma}\vv_N\right\| _{H^\alpha(\RR^3)}
  \leq  &  C \left\|\langle \cdot\rangle^{\beta}(I_{d}+\Lambda^{\{s\}}D^{\gamma})\ff_1\right\|_{L^2(\RR^3)}+C\|\Lambda^{s}\ff_2\|_{H^1(\RR^3)}\\
   & + C \left\|\langle \cdot\rangle^{\beta}(I_{d}+\Lambda^{\{s\}}D^{\gamma})\Div\ff_2\right\|_{H^{-\alpha}(\RR^3)}.
  \end{split}
\end{equation}
Employing \eqref{eq-V-H-weight}  in the system \eqref{E-L-Appr-diff}, we readily have
\begin{equation}\label{eq-V-H-weight-2-1}
\begin{aligned}
& \sup_{|\gamma|=k} \left\|\langle \cdot\rangle^{\beta}  D^\gamma\vv_N\right\| _{H^\alpha(\RR^3)} \\
  \leq&  C \sup_{|\gamma|=k}\left(\| D^\gamma\vv_N\|_{H^\alpha(\RR^3)}+ \left\|\langle \cdot\rangle^{\beta}S_ND^\gamma\ff_1\right\|_{L^2(\RR^3)}+\left\|\langle \cdot\rangle^{\beta}D^\gamma\vv_N\right\|_{L^2(\RR^3)}\right)\\
& + C \sup_{|\gamma|=k}\left( \left\|\langle \cdot\rangle^{\beta}S_ND^\gamma\Div\ff_2\right\|_{H^{-\alpha}(\RR^3)}+\left\|D^\gamma \ff_2\right\|_{H^1(\RR^3)}\right).
  \end{aligned}
\end{equation}
From \eqref{eq-w-22}, it is easy to infer that
\begin{equation}\label{eq-w-22-2}
\left\|\langle \cdot\rangle^{\beta}S_ND^{\gamma}\ff_1\right\|_{L^2(\RR^3)}\leq C2^{-\beta N}\left\| D^{\gamma}\ff_1\right\|_{L^2(\RR^3)}+\left\|| \cdot|^{\beta}D^{\gamma}\ff_1\right\|_{L^2(\RR^3)}.
\end{equation}
The same way as deriving \eqref{eq-w-22-3} enables us to conclude that
\begin{equation}\label{eq-w-22-3-R}
 \left\|\langle \cdot\rangle^{\beta}S_ND^\gamma\Div\ff_2\right\|_{H^{-\alpha}(\RR^3)}\leq C \left\|\langle \cdot\rangle^{\beta}D^\gamma\Div\ff_2\right\|_{H^{-\alpha}(\RR^3)} + C\left\| D^\gamma\Div \ff_2\right\|_{L^{2}(\RR^3)}.
\end{equation}
Since $\alpha\in(0,1)$ and $k\in\NN^+$, we have by the Leibniz rule that for each $M\in\NN^+,$
\begin{align*}
\left\|\langle \cdot\rangle^{\beta}D^\gamma\vv_N\right\|_{L^2(\RR^3)}\leq & \sum_{k\geq M}\left\|\dot{\Delta}_k\left(\langle \cdot\rangle^{\beta}D^\gamma\vv_N\right)\right\|_{L^2(\RR^3)}+\left\|S_M\left(\langle \cdot\rangle^{\beta}D^\gamma\vv_N\right)\right\|_{L^2(\RR^3)}\\
\leq & C2^{-M\alpha}\left\|\langle \cdot\rangle^{\beta}D^\gamma\vv_N\right\|_{\dot{H}^\alpha(\RR^3)}+\sup_{|\gamma|=k-1}\left\|S_M D\left(\langle \cdot\rangle^{\beta}D^\gamma\vv_N\right)\right\|_{L^2(\RR^3)}\\&+\sup_{|\gamma|=k-1}\left\|S_M\left( D\langle \cdot\rangle^{\beta}D^\gamma\vv_N\right)\right\|_{L^2(\RR^3)},
\end{align*}
which implies
\begin{equation}\label{eq-w-22-23}
\begin{split}
&\sup_{|\gamma|=k}\left\|\langle \cdot\rangle^{\beta}D^\gamma\vv_N\right\|_{L^2(\RR^3)}\\
\leq &C\sup_{|\gamma|=k-1}\left\|\langle \cdot\rangle^{\beta}D^\gamma\vv_N\right\|_{L^2(\RR^3)}+C2^{-M\alpha}\sup_{|\gamma|=k}\left\|\langle \cdot\rangle^{\beta}D^\gamma\vv_N\right\|_{\dot{H}^\alpha(\RR^3)}.
\end{split}
\end{equation}
Inserting \eqref{eq-w-22-2}-\eqref{eq-w-22-23} into \eqref{eq-V-H-weight-2-1} and then taking the suitable $M$, we get
\begin{equation}\label{eq-w-22-24}
\begin{aligned}
 \sup_{|\gamma|=k}\left\|\langle \cdot\rangle^{\beta}  D^\gamma\vv_N\right\| _{H^\alpha(\RR^3)}
  \leq&  C \sup_{|\gamma|=k}\left(\| D^\gamma\vv_N\|_{H^\alpha(\RR^3)}+ \left\|\langle \cdot\rangle^{\beta}D^\gamma\ff_1\right\|_{L^2(\RR^3)}\right)\\
  & + C \sup_{|\gamma|=k}\left( \left\|\langle \cdot\rangle^{\beta}D^\gamma\Div\ff_2\right\|_{H^{-\alpha}(\RR^3)}+\left\|D^\gamma \ff_2\right\|_{H^1(\RR^3)}\right)\\
  &+C \sup_{|\gamma|=k-1}\left\|\langle \cdot\rangle^{\beta}D^\gamma\vv_N\right\|_{L^2(\RR^3)}.
  \end{aligned}
\end{equation}
With estimates \eqref{eq-V-H-weight-22} and \eqref{eq-w-22-23} in hand, the estimate \eqref{eq-w-22-24} becomes
\begin{equation}\label{eq-w-22-25}
\begin{aligned}
 \sup_{|\gamma|=k} \left\|\langle \cdot\rangle^{\beta}  D^\gamma\vv_N\right\| _{H^\alpha(\RR^3)}    \leq& C  \sup_{|\gamma|=k}\left( \left\|\langle \cdot\rangle^{\beta}(I_{d}+D^{\gamma})\ff_1\right\|_{L^2(\RR^3)}+\left\|D^\gamma \ff_2\right\|_{H^1(\RR^3)}\right)\\&+ C \sup_{|\gamma|=k} \left\|\langle \cdot\rangle^{\beta}D^\gamma\Div\ff_2\right\|_{H^{-\alpha}(\RR^3)}.
\end{aligned}
\end{equation}
Employing \eqref{eq-V-H-weight} in the system \eqref{E-L-Appr-diff-Lam-2}, we readily have that for each $s\in(k,k+1),$
\begin{equation}\label{eq-V-H-weight-217}
\begin{aligned}
 &\sup_{|\gamma|=k} \left\|\langle \cdot\rangle^{\beta} \Lambda^{\{s\}}D^\gamma\vv_N\right\| _{H^\alpha(\RR^3)}  \\
  \leq&  C\sup_{|\gamma|=k} \left(\|\Lambda^{\{s\}}D^\gamma\vv_N\|_{H^\alpha(\RR^3)}+ \left\|\langle \cdot\rangle^{\beta}S_N\Lambda^{\{s\}}D^\gamma\ff_1\right\|_{L^2(\RR^3)}\right)\\
& + C \sup_{|\gamma|=k}\left( \left\|\langle \cdot\rangle^{\beta}S_N\Lambda^{\{s\}}D^\gamma\Div\ff_2\right\|_{H^{-\alpha}(\RR^3)}+\left\|\Lambda^{\{s\}}D^\gamma \ff_2\right\|_{H^1(\RR^3)}\right)\\
&+ C \sup_{|\gamma|=k}\left\|\langle \cdot\rangle^{\beta}\Lambda^{\{s\}}D^\gamma\vv_N\right\|_{L^2(\RR^3)}.
  \end{aligned}
\end{equation}
By mimicking \eqref{eq-w-22}, we can show
\begin{equation}\label{eq-w-22-2-17}
\left\|\langle \cdot\rangle^{\beta}S_N\Lambda^{\{s\}}D^{\gamma}\ff_1\right\|_{L^2(\RR^3)}\leq C2^{-\beta N}\left\| \Lambda^{\{s\}}D^{\gamma}\ff_1\right\|_{L^2(\RR^3)}+C\left\|\langle \cdot\rangle^{\beta}\Lambda^{\{s\}}D^{\gamma}\ff_1\right\|_{L^2(\RR^3)}.
\end{equation}
From \eqref{eq-w-22-3},  we get
\begin{equation}\label{eq-w-22-3-RR}
\begin{split}
& \left\|\langle \cdot\rangle^{\beta}S_N\Lambda^{\{s\}}D^\gamma\Div\ff_2\right\|_{H^{-\alpha}(\RR^3)}\\
 \leq &C \left\|\langle \cdot\rangle^{\beta}\Lambda^{\{s\}}D^\gamma\Div\ff_2\right\|_{H^{-\alpha}(\RR^3)} + C\left\| \Lambda^{\{s\}}D^\gamma\Div \ff_2\right\|_{L^{2}(\RR^3)}.
 \end{split}
\end{equation}
In the similar method as in proof of \eqref{eq-w-22-23}, we infer that for each $k\in\NN^+,$
\begin{equation}\label{eq-w-22-23-17}
\begin{split}
&\sup_{|\gamma|=k}\left\|\langle \cdot\rangle^{\beta}\Lambda^{\{s\}}D^\gamma\vv_N\right\|_{L^2(\RR^3)}\\
\leq& C\sup_{|\gamma|=k-1}\left\|\langle \cdot\rangle^{\beta}\Lambda^{\{s\}}D^\gamma\vv_N\right\|_{L^2(\RR^3)}+C2^{-M\alpha}\sup_{|\gamma|=k}\left\|\langle \cdot\rangle^{\beta}\Lambda^{\{s\}}D^\gamma\vv_N\right\|_{\dot{H}^\alpha(\RR^3)}.
\end{split}
\end{equation}
Inserting \eqref{eq-w-22-2-17}-\eqref{eq-w-22-23-17} into \eqref{eq-V-H-weight-2-1} and then taking the suitable integer $M$, we immediately get that for each $s\in(k,k+1),$
\begin{align*}
&\sup_{|\gamma|=k} \left\|\langle \cdot\rangle^{\beta} \Lambda^{\{s\}} D^\gamma\vv_N\right\| _{H^\alpha(\RR^3)}  \\
  \leq&  C \sup_{|\gamma|=k}\left(\|\Lambda^{\{s\}} D^\gamma\vv_N\|_{H^\alpha(\RR^3)}+ \left\|\langle \cdot\rangle^{\beta}\Lambda^{\{s\}}D^\gamma\ff_1\right\|_{L^2(\RR^3)}\right)\\
 & + C \sup_{|\gamma|=k}\left( \left\|\langle \cdot\rangle^{\beta}\Lambda^{\{s\}}D^\gamma\Div\ff_2\right\|_{H^{-\alpha}(\RR^3)}+\left\|\Lambda^{\{s\}}D^\gamma \ff_2\right\|_{H^1(\RR^3)}\right)\\
  &+C \sup_{|\gamma|=k-1}\left\|\langle \cdot\rangle^{\beta}\Lambda^{\{s\}}D^\gamma\vv_N\right\|_{L^2(\RR^3)}.
\end{align*}
Plugging \eqref{eq-V-H-weight-22} and \eqref{eq-w-22-23} into the above inequality leads to that for each $s\in(k,k+1),$
\begin{equation}\label{eq-w-22-25-17}
\begin{split}
 &\sup_{|\gamma|=k} \left\|\langle \cdot\rangle^{\beta} \Lambda^{\{s\}} D^\gamma\vv_N\right\| _{H^\alpha(\RR^3)}  \\
\leq &C \sup_{|\gamma|=k}  \left\|\langle \cdot\rangle^{\beta}(I_{d}+\Lambda^{\{s\}}D^\gamma)\ff_1\right\|_{L^2(\RR^3)}+ C \sup_{|\gamma|=k} \left\|\Lambda^{\{s\}}D^\gamma \ff_2\right\|_{H^1(\RR^3)} \\
&+C \sup_{|\gamma|=k}  \left\|\langle \cdot\rangle^{\beta}\Lambda^{\{s\}}D^\gamma\Div\ff_2\right\|_{H^{-\alpha}(\RR^3)} +C \sup_{|\gamma|=k-1}\left\|\langle \cdot\rangle^{\beta}\Lambda^{\{s\}}D^\gamma\vv_N\right\|_{L^2(\RR^3)}.
\end{split}
\end{equation}
Using the estimate  \eqref{eq-w-22-25}, we get from  \eqref{eq-w-22-25-17} that for each $s\in[k,k+1),$
\begin{align*}
&\sup_{|\gamma|=k} \left\|\langle \cdot\rangle^{\beta} \Lambda^{\{s\}} D^\gamma\vv_N\right\| _{H^\alpha(\RR^3)}  \\
\leq &C   \sup_{|\gamma|=k}  \left\|\langle \cdot\rangle^{\beta}(I_{d}+\Lambda^{\{s\}}D^\gamma)\ff_1\right\|_{L^2(\RR^3)}+ C \sup_{|\gamma|=k} \left\|\Lambda^{\{s\}}D^\gamma \ff_2\right\|_{H^1(\RR^3)} \\
&+C \sup_{|\gamma|=k}  \left\|\langle \cdot\rangle^{\beta}\Lambda^{\{s\}}D^\gamma\Div\ff_2\right\|_{H^{-\alpha}(\RR^3)}.
\end{align*}
By induction, we eventually get that for each $s\geq0,$
\begin{align*}
&\sup_{|\gamma|=[s]} \left\|\langle \cdot\rangle^{\beta} \Lambda^{\{s\}} D^\gamma\vv_N\right\| _{H^\alpha(\RR^3)}  \\
 \leq&C    \sup_{|\gamma|=[s]} \left\|\langle \cdot\rangle^{\beta}(I_{d}+\Lambda^{\{s\}}D^\gamma)\ff_1\right\|_{L^2(\RR^3)}+ C \sup_{|\gamma|=[s]} \left\|\Lambda^{\{s\}}D^\gamma \ff_2\right\|_{H^1(\RR^3)} \\
&+C \sup_{|\gamma|=[s]}  \left\|\langle \cdot\rangle^{\beta}\Lambda^{\{s\}}D^\gamma\Div\ff_2\right\|_{H^{-\alpha}(\RR^3)}.
\end{align*}
Taking $N\to\infty$ implies the desired estimate \eqref{eq-V-H-weight-h-s}.

In the same way as in the proof of \eqref{eq-V-H-weight-h-s}, we can show the  desired estimate~\eqref{eq-V-H-weight-h-s-delta} of Theorem \ref{thm-V-H-L2-weight-22}.
\end{proof}

%%%%%%%%%%%%%%%%%%%%%%%%%%%%%%%%%%%%%%%%%%%%%%%%%%%%%%%%%%%%%%%%%%%%%%%%%%%%%%%%%%%%%%%%%%%%%%%%%%%
\section{High Regularity for weak solutions to Leray problem }\label{HighR}
%%%%%%%%%%%%%%%%%%%%%%%%%%%%%%%%%%%%%%%%%%%%%%%%%%%%%%%%%%%%%%%%%%%%%%%%%%%%%%%%%%%%%%%%%%%%%%
\setcounter{section}{3}\setcounter{equation}{0}
In this section, we are going to show the high regularity for weak solutions to Leray problem~\eqref{E-forwad-leray} and the perturbed Leray problem \eqref{E}. For simplicity, we just need to consider the following hybrid problem in the whole space
\begin{align}\label{E-2}
\left\{
\begin{aligned}
&(-\Delta)^{\alpha}\vv- \frac{2\alpha-1}{2\alpha}\vv-\frac{1}{2\alpha}\xx\cdot \nabla\vv+\nabla P=-\vv\cdot\nabla \vv  +L_{\uu_0}(\vv)+\operatorname{div}\ff,\\
&\textnormal{div}\,\vv=0,
\end{aligned}  \right.
\end{align}
where
\[L_{\uu_0}\vv=- \uu_{0} \cdot\nabla \vv-\vv\cdot\nabla \uu_{0}.\]
Let $\ff_1=-\vv\cdot\nabla \uu_{0}+\operatorname{div}\ff.$  Then the system \eqref{E-2} can be rewritten as
\begin{align}\label{E-G-Leray}
\left\{
\begin{aligned}
& (-\Delta)^{\alpha}\vv- \frac{2\alpha-1}{2\alpha}\vv-\frac{1}{2\alpha}\xx\cdot \nabla\vv+(\uu_0+\vv)\cdot\nabla \vv+\nabla P=\ff_1, \\
& \textnormal{div}\,\vv=0.
\end{aligned}  \right.
\end{align}

%%%%%%%%%%%%%%%%%%%%%%%%%%%%%%%%%%%%%%%%%%%%%%%%%%%%%%%%%%%%%%%%%%%%%%%%%%%%%%%%%%%%%%%%%%%%%%%%%%%
\subsection{High regularity for weak solutions to the perturbed Leray problem}
%%%%%%%%%%%%%%%%%%%%%%%%%%%%%%%%%%%%%%%%%%%%%%%%%%%%%%%%%%%%%%%%%%%%%%%%%%%%%%%%%%%%%%%%%%%%%%

This subsection is devoted to proving $L^{p}$-type elliptic regularity for the weak solutions to the problem~\eqref{E-2}.
\begin{theorem}\label{thm-Lp-U-G}
Assume $\alpha\in[5/6,1]$, $p\in[2,+\infty)$, $\operatorname{div}\ff \in L^2(\RR^3)\cap \dot{B}^0_{p,p}(\RR^3)$ and $\sigma\in C^{1,0}(\mathbb{S}^2)$. Let $(\vv,P)\in H^\alpha(\RR^3)\times L^{\frac{3}{3-2\alpha}}(\RR^3)$ be a weak solution to the system \eqref{E-2}. Then there exists a constant $C= C_{p,\sigma}>0$ such that
\begin{equation}\label{eq-Lp-vvvv}
\left\|\uu\right\|_{\dot{B}^{\frac{5}{3}}_{p,p}(\RR^3)}+\left\|P\right\|_{\dot{B}^{1}_{p,p}(\RR^3)}
\leq  C_{p,\sigma}\left(1+\left\|\operatorname{div}\ff \right\|_{\dot{B}^{0}_{p,p}(\RR^3)}+ \|\operatorname{div}\ff \|_{L^2(\RR^3)}\right).
\end{equation}
\end{theorem}
\begin{proof}
 Let $\vv_k=J_k\vv$ be a sequence of smooth functions approximating $\vv$, and  $\ww_k\in H^\alpha_{\sigma}(\RR^3)$ be a weak solution to the corresponding linearized equations  in the whole space
 \begin{align}\label{E-Approx}
\left\{
\begin{aligned}
 &(-\Delta)^{\alpha}\ww_k- \frac{2\alpha-1}{2\alpha}\ww_k-\frac{1}{2\alpha}\xx\cdot \nabla \ww_k
+J_k\big((\uu_0+\vv_k)\cdot\nabla J_k\ww_k \big) +\nabla P_k=\ff_1, \\
 &\textnormal{div}\,\ww_k=0.
\end{aligned}  \right.
\end{align}
Since $\uu_0\in L^2(\RR^3),\,\vv\in H^\alpha(\RR^3)$ and $\nabla \uu_0\in L^\infty(\RR^3)$, we readily have that $\ff_1\in L^2(\RR^3),$ moreover, we get the following estimate by the standard $L^2$-estimate
\begin{equation}\label{eq-V-H-L-five}
\|\ww_k\| _{ {H}^{\alpha}(\RR^3)} \leq C_\alpha \|\ff_1\|_{L^2(\RR^3)}.
\end{equation}
Taking $s=0$ in Theorem \ref{thm-V-H}, we find that weak solution $(\ww_k,P_k)$ fulfils
\begin{equation}\label{eq-V-H-L}
\begin{split}
&\left\|\ww_k\right\|_{\dot{B}^{2\alpha}_{p,p}(\RR^3)}+\left\|P_k\right\|_{\dot{B}^{1}_{p,p}(\RR^3)}\\
\leq& C_{p}\left(\left\|\operatorname{div}\left((\uu_0+\vv_k)\otimes J_k\ww_k \right)\right\|_{\dot{B}^{0}_{p,p}(\RR^3)}
+\left\|\ff_1\right\|_{\dot{B}^{0}_{p,p}(\RR^3)}+\left\|\ww_k\right\|_{\dot{B}^{0}_{p,p}(\RR^3)}\right).
\end{split}
\end{equation}
Let us first consider the case where $p\in[2,{9}/{2})$. Since $\vv\in H^\alpha(\RR^3)$ with $\alpha\in[5/6,1],$  we  obtain by Lemma \ref{lem-SYB} that for each $p\in[2,{9}/{2})$,
\begin{equation}\label{eq-KY-Fractional-BS-R}
\|\operatorname{div}\left(\vv_k\otimes J_k\ww_k\right)\|_{\dot{B}^{0}_{p,p}(\RR^3)}\leq \varepsilon\|\ww_k\|_{\dot{B}^{\frac{5}{3}}_{p,p}(\RR^3)}+C_\varepsilon\|\ww_k\|_{\dot{B}^0_{p,p}(\RR^3)}.
\end{equation}
Since $\alpha\geq\frac{5}{6}$, one has by using the interpolation inequality and the Young inequality that
\begin{align*}
\|\ww_k\|_{\dot{B}^{\frac{5}{3}}_{p,p}(\RR^3)}\leq &\|\ww_k\|^{\frac{6\alpha-5}{6\alpha}}_{\dot{B}^{0}_{p,p}(\RR^3)}\|\ww_k\|^{\frac{5}{6\alpha}}_{\dot{B}^{2\alpha}_{p,p}(\RR^3)}\\
\leq &\|\ww_k\|_{\dot{B}^{2\alpha}_{p,p}(\RR^3)}+C_\alpha\|\ww_k\|_{\dot{B}^{0}_{p,p}(\RR^3)}.
\end{align*}
 This estimate together with \eqref{eq-KY-Fractional-BS-R} leads to
\begin{equation}\label{eq-KY-Fractional-BS-RR}
\|\operatorname{div}\left(\vv_k\otimes J_k\ww_k\right)\|_{\dot{B}^{0}_{p,p}(\RR^3)}\leq \varepsilon\|\ww_k\|_{\dot{B}^{2\alpha}_{p,p}(\RR^3)}+C_\varepsilon\|\ww_k\|_{\dot{B}^0_{p,p}(\RR^3)}.
\end{equation}
Using the Leibniz estimate, we see that
\begin{equation}\label{eq-p-1}
\begin{split}
\|\operatorname{div}\left(\uu_0\otimes J_k\ww_k\right)\|_{\dot{B}^{0}_{p,p}(\RR^3)}\leq & C\|\uu_0\|_{L^\infty(\RR^3)}\|\ww_k\|_{\dot{B}^1_{p,p}(\RR^3)}\\
&+C\|\nabla\uu_0\|_{L^\infty(\RR^3)}\|\ww_k\|_{\dot{B}^0_{p,p}(\RR^3)}.
\end{split}
\end{equation}
Since $\alpha\in[5/6,1],$  we get by the interpolation inequality and the Young inequality that
\begin{align*}
\|\ww_k\|_{\dot{B}^1_{p,p}(\RR^3)}\leq & C\|\ww_k\|^{\frac{2\alpha-1}{2\alpha}}_{\dot{B}^0_{p,p}(\RR^3)}\|\ww_k\|^{\frac{1}{2\alpha}}_{\dot{B}^{2\alpha}_{p,p}(\RR^3)}\\
\leq & C_\varepsilon\|\ww_k\|_{\dot{B}^0_{p,p}(\RR^3)}+\varepsilon\|\ww_k\|_{\dot{B}^{2\alpha}_{p,p}(\RR^3)}.
\end{align*}
Plugging this estimate into \eqref{eq-p-1} leads to
\begin{equation}\label{eq-p-2}
\|\operatorname{div}\left(\uu_0\otimes J_k\ww_k\right)\|_{\dot{B}^{0}_{p,p}(\RR^3)} \leq \varepsilon\|\ww_k\|_{\dot{B}^{2\alpha}_{p,p}(\RR^3)}+C_\varepsilon\|\ww_k\|_{\dot{B}^0_{p,p}(\RR^3)}.
\end{equation}
Inserting estimates \eqref{eq-KY-Fractional-BS-RR} and \eqref{eq-p-2} into \eqref{eq-V-H-L}, we have
\begin{equation}\label{eq-V-H-L-sec}
\begin{split}
&\left\|\ww_k\right\|_{\dot{B}^{2\alpha}_{p,p}(\RR^3)}+\left\|P_k\right\|_{\dot{B}^{1}_{p,p}(\RR^3)}\\
\leq& C_{p}\left(\left\|\ff_1\right\|_{\dot{B}^{0}_{p,p}(\RR^3)}+C_\varepsilon\|\ww_k\|_{\dot{B}^0_{p,p}(\RR^3)}\right)+\varepsilon C_{p,\uu_0}\|\ww_k\|_{\dot{B}^{2\alpha}_{p,p}(\RR^3)}.
\end{split}
\end{equation}
Since $p\geq2,$ we obtain the following sharp  inequality
\begin{equation}\label{eq-sharp-1}
\begin{split}
\| \ww_k \|_{ {B}^{0}_{p,1}(\RR^3)}\leq&\sum_{q\leq N}\big\|\dot{\Delta}_q\ww_k\big\|_{L^p(\RR^3)}+\sum_{q> N}\big\|\dot{\Delta}_q\ww_k\big\|_{L^p(\RR^3)}\\
\leq & C2^{\frac{p-2}{2p}N}\|\ww_k\|_{L^2(\RR^3)}+2^{-2\alpha N}\left\|\ww_k\right\|_{\dot{B}^{2\alpha}_{p,\infty}(\RR^3)},
\end{split}
\end{equation}
where $N$ is an integer to be chosen later.

Plugging this sharp inequality into estimate \eqref{eq-V-H-L-sec} allows us to conclude
\begin{equation}\label{eq-V-H-L-third}
\begin{split}
&\left\|\ww_k\right\|_{\dot{B}^{2\alpha}_{p,p}(\RR^3)}+\left\|P_k\right\|_{\dot{B}^{1}_{p,p}(\RR^3)}\\
\leq& C_{p }\left(\left\|\ff_1\right\|_{\dot{B}^{0}_{p,p}(\RR^3)}+2^{\frac{p-2}{2p}N}\|\ww_k\|_{L^2(\RR^3)}\right) +C_{p }\left(\varepsilon+C_\varepsilon2^{-\frac53N}\right)\|\ww_k\|_{\dot{B}^{\frac{5}{3}}_{p,p}(\RR^3)}.
\end{split}
\end{equation}
Choosing a suitable number $\varepsilon$ and a sufficiently large  integer $N$ in \eqref{eq-V-H-L-third}, one has
\begin{equation} \label{eq-V-H-L-four}
 \left\|\ww_k\right\|_{\dot{B}^{2\alpha}_{p,p}(\RR^3)}+\left\|P_k\right\|_{\dot{B}^{1}_{p,p}(\RR^3)}
\leq  C_{p }\left(\left\|\ff_1\right\|_{\dot{B}^{0}_{p,p}(\RR^3)}+\|\ww_k\|_{L^2(\RR^3)}\right).
\end{equation}
 Plugging \eqref{eq-V-H-L-five} into \eqref{eq-V-H-L-four}, we immediately obtain
\begin{equation} \label{eq-V-H-L-six}
 \left\|\ww_k\right\|_{\dot{B}^{2\alpha}_{p,p}(\RR^3)}+\left\|P_k\right\|_{\dot{B}^{1}_{p,p}(\RR^3)}
\leq  C_{p }\left(\left\|\ff_1\right\|_{\dot{B}^{0}_{p,p}(\RR^3)}+ \|\ff_1\|_{L^2(\RR^3)}\right)
\end{equation}
According to $\ff_1=-\vv\cdot\nabla \uu_{0}+\operatorname{div}\ff,$ we have by the H\"older inequality that
\begin{equation}\label{eq-p-p-1}
\|\ff_1\|_{L^2(\RR^3)}\leq\|\nabla \uu_0\|_{L^\infty(\RR^3)}\|\vv\|_{L^2(\RR^3)}+\|\operatorname{div}\ff\|_{L^2(\RR^3)},
\end{equation}
and by the  Leibniz estimate  that for each $p\in[2,9/2),$
\begin{equation}\label{eq-p-p-2}
\begin{split}
\|\ff_1\|_{\dot{B}^0_{p,p}(\RR^3)}\leq  C\|\ff_1\|_{\dot{F}^0_{p,2}(\RR^3)}
\leq&C\|\nabla \uu_0\|_{L^\infty(\RR^3)}\|\vv\|_{L^p(\RR^3)}+\|\operatorname{div}\ff\|_{L^p(\RR^3)}\\
\leq&C\|\nabla \uu_0\|_{L^\infty(\RR^3)}\|\vv\|_{H^\alpha(\RR^3)}+\|\operatorname{div}\ff\|_{L^p(\RR^3)}.
\end{split}
\end{equation}
Plugging \eqref{eq-p-p-1} and \eqref{eq-p-p-2} into \eqref{eq-V-H-L-six}, we immediately have that for each $p\in[2,9/2),$
\begin{equation} \label{eq-V-H-L-seven}
 \left\|\ww_k\right\|_{\dot{B}^{2\alpha}_{p,p}(\RR^3)}+\left\|P_k\right\|_{\dot{B}^{1}_{p,p}(\RR^3)}
\leq  C_{p,\sigma }\left(1+\left\|\ff_1\right\|_{\dot{B}^{0}_{p,p}(\RR^3)}+ \|\operatorname{div}\ff\|_{L^2(\RR^3)}\right).
\end{equation}
The above uniform estimate \eqref{eq-V-H-L-seven} uniformly in $k $ enables us to infer that subsequences of $\ww_k$ and $P_k$
converge locally weakly in  $H^{2\alpha}(\RR^3)$ and in  $H^{1}(\RR^3)$ to a solution $(\ww,P)$ of Cauchy problem
 \begin{align} \label{eq-Limit}
\left\{
\begin{aligned}
& (-\Delta)^{\alpha}\ww - \frac{2\alpha-1}{2\alpha}\ww -\frac{1}{2\alpha}\xx\cdot \nabla \ww +\big(\uu_0+\vv\big)\cdot\nabla\ww +\nabla P =\ff_1 \\
 &\textnormal{div}\,\ww =0.
\end{aligned}\ \right.
\end{align}
Since $\vv$ and $\ww$ are both weak solutions to the equations \eqref{eq-Limit}, we get by uniqueness that  $\vv = \ww,$ and we have
\begin{equation} \label{eq-V-H-L-seven-vv}
 \left\|\vv\right\|_{\dot{B}^{\frac{5}{3}}_{p,p}(\RR^3)}+\left\|P\right\|_{\dot{B}^{1}_{p,p}(\RR^3)}
\leq  C_{p,\sigma }\left(1+\left\|\operatorname{div}\ff\right\|_{\dot{B}^{0}_{p,p}(\RR^3)}+ \|\operatorname{div}\ff\|_{L^2(\RR^3)}\right).
\end{equation}
Next we consider the case where $\operatorname{div}\ff\in L^2(\RR^3)\cap \dot{B}^0_{p,p}(\RR^3)$ with $\frac92\leq p<\infty$. By the interpolation theorem, we have  that for each $r\in(2,p),$
\[\|\operatorname{div}\ff\|_{\dot{B}^0_{r,r}(\RR^3)}\leq \|\operatorname{div}\ff\|^{\frac{2(p-r)}{r(p-2)}}_{L^2(\RR^3)}\|\operatorname{div}\ff\|^{\frac{(r-2)p}{r(p-2)}}_{\dot{B}^0_{p,p}(\RR^3)},\]
which implies that
$\operatorname{div}\ff\in L^2(\RR^3)\cap \dot{B}^0_{p,p}(\RR^3)$ for each $p\in[2,9/2)$. We know from \eqref{eq-V-H-L-seven-vv} that for each $p\in[2,9/2)$,
\[\vv\in H^{2\alpha}(\RR^3)\cap\dot{B}^{2\alpha}_{p,p}(\RR^3)  \quad \text { with } \alpha \in[5/6,1].\]
Hence,  by using the Sobolev embedding theorem, we have $\vv\in L^\infty(\RR^3)$  and
$$
\nabla  \vv\in L^p(\RR^3) \quad \text { for each } p \in[2, \infty).
$$
These facts imply the  Leibniz estimate
\begin{equation}\label{eq-2023}
\begin{split}
&\left\|(\vv+\uu_0)\cdot\nabla \vv\right\|_{\dot{B}^0_{p,p}(\RR^3)} +\left\|\vv\cdot\nabla \uu_0\right\|_{\dot{B}^0_{p,p}(\RR^3)}+\left\|\operatorname{div}\ff\right\|_{\dot{B}^0_{p,p}(\RR^3)}\\
\leq &C\left(\|\vv\|_{L^\infty(\RR^3)}+\|\uu_0\|_{L^\infty(\RR^3)}\right)\left\|  \nabla \vv\right\|_{L^p(\RR^3)}+C\left\|\vv\right\|_{L^\infty(\RR^3)}\left\|\nabla \uu_0\right\|_{L^p(\RR^3)}\\&+\left\|\operatorname{div}\ff\right\|_{\dot{B}^0_{p,p}(\RR^3)}<\infty.
\end{split}
\end{equation}
Since $\vv$ is  weak solution to the system \eqref{E-2} with $\alpha\in[5/6,1]$, we have by Theorem \ref{thm-V-H} that
\begin{equation}\label{eq-V-H-L-HHHH}
\begin{split}
&\left\|\vv\right\|_{\dot{B}^{2\alpha}_{p,p}(\RR^3)}+\left\|P \right\|_{\dot{B}^{1}_{p,p}(\RR^3)}\\
\leq& C_{p}\left(\left\|\operatorname{div}\left((\uu_0+\vv )\otimes  \vv \right)\right\|_{\dot{B}^{0}_{p,p}(\RR^3)}
+\left\|\ff_1\right\|_{\dot{B}^{0}_{p,p}(\RR^3)}+\left\|\ww_k\right\|_{\dot{B}^{0}_{p,p}(\RR^3)}\right).
\end{split}
\end{equation}
Inserting \eqref{eq-sharp-1} and \eqref{eq-2023} into \eqref{eq-V-H-L-HHHH} and then choosing the suitable $N$, we finally obtain
Theorem \ref{thm-Lp-U-G} for the case where $p>\frac92$.
\end{proof}
\begin{theorem}\label{thm-Lp-V-H}
Assume $\alpha\in[5/6,1]$, $s\geq0$, $p\in[2,+\infty)$, $\operatorname{div}\ff \in L^2(\RR^3)\cap \dot{B}^s_{p,p}(\RR^3)$ and $\sigma\in C^{1,0}(\mathbb{S}^2)$.
Let $\vv$ be a weak solution to the system \eqref{E-2}. Then there exists a constant $C$ such that
\begin{equation}\label{eq-Lp-vvvv}
\left\|\vv\right\|_{\dot{B}^{s+2\alpha}_{p,p}(\RR^3)}+\left\|P\right\|_{\dot{B}^{s+1}_{p,p}(\RR^3)}
\leq  C\left(p,\sigma,\|\operatorname{div}\ff\|_{L^2(\RR^3)\cap\dot{B}^s_{p,p}(\RR^3)}\right).
\end{equation}
\end{theorem}
\begin{proof}
 Applying Theorem \ref{thm-V-H} to the equations \eqref{E-2}, we get that for all $s>0,$
 \begin{equation}\label{eq-V-H-L-s}
\begin{split}
\left\|\vv\right\|_{\dot{B}^{s+2\alpha}_{p,p}(\RR^3)}+\left\|P\right\|_{\dot{B}^{s+1}_{p,p}(\RR^3)}
\leq C_{p}\left(
\left\|\overline{\ff}\right\|_{\dot{B}^{s}_{p,p}(\RR^3)}+\left\|\vv\right\|_{\dot{B}^{s}_{p,p}(\RR^3)}\right),
\end{split}
\end{equation}
where
\[\overline{\ff}=-\operatorname{div}\left((\vv+\uu_0)\otimes\vv\right)-\operatorname{div}\left(\vv\otimes\uu_0\right)+\Div \ff.\]
Since $\Div\ff\in L^2(\RR^3)\cap\dot{B}^s_{p,p}(\RR^3)$ with  $s\geq0$, we have by the interpolation theorem that there exists $\theta=\frac{2ps}{2ps+3(p-2)} $ such that
\begin{equation*}
\|\Div\ff\|_{\dot{B}^0_{p,p}(\RR^3)}\leq\|\Div\ff\|^\theta_{L^2(\RR^3)}\|\Div\ff\|^{1-\theta}_{\dot{B}^s_{p,p}(\RR^3)}<\infty.
\end{equation*}
 Moreover, it follows from Theorem \ref{thm-Lp-U-G} that
\begin{equation}\label{eq.eq-r-1}
\vv\in H^{2\alpha}(\RR^3)\cap \dot{B}^{2\alpha}_{p,p}(\RR^3).
\end{equation}
Thanks to the Bony paraproduct decomposition, we have by the Leibniz estimates that
 \begin{align*}
  &\left\|\operatorname{div}\left((\vv+\uu_0)\otimes\vv\right)\right\|_{\dot{B}^{s}_{p,p}(\RR^3)}\\
 \leq &C\|\vv\|_{L^\infty(\RR^3)} \|\vv\|_{\dot{B}^{s+1}_{p,p}(\RR^3)}+C\|\vv\|_{L^p(\RR^3)}
 \|\uu_0\|_{\dot{W}^{s+1,\infty}(\RR^3)}  + C\|\uu_0\|_{L^\infty(\RR^3)} \left\| \vv \right\|_{\dot{B}^{s+1}_{p,p}(\RR^3)},
\end{align*}
and
\begin{align*}
\left\|\operatorname{div}\left(\vv\otimes\uu_0\right)\right\|_{\dot{B}^{s}_{p,p}(\RR^3)}
 \leq &C\|\nabla \uu_0\|_{L^\infty(\RR^3)} \left\| \vv \right\|_{\dot{B}^{s}_{p,p}(\RR^3)}+C\|  \vv\|_{L^\infty(\RR^3)} \left\| \nabla \uu_0 \right\|_{\dot{B}^{s}_{p,p}(\RR^3)}.
\end{align*}
Both estimates above imply
\begin{align*}
\|\overline{\ff}\|_{\dot{B}^s_{p,p}(\RR^3)} \leq &\|\Div \ff \|_{\dot{B}^s_{p,p}(\RR^3)} +C\|\vv\|_{L^\infty(\RR^3)} \|\vv\|_{\dot{B}^{s+1}_{p,p}(\RR^3)}+C\|\vv\|_{\dot{B}^s_{p,p}}
 \|\uu_0\|_{\dot{W}^{s+1,\infty}(\RR^3)} \\
&+C\|\nabla \uu_0\|_{L^\infty(\RR^3)} \left\| \vv \right\|_{\dot{B}^{s}_{p,p}(\RR^3)}+C\|  \vv\|_{L^\infty(\RR^3)} \left\| \nabla \uu_0 \right\|_{\dot{B}^{s}_{p,p}(\RR^3)}.
\end{align*}
Inserting this estimate into \eqref{eq-V-H-L-s} yields
 \begin{equation}\label{eq-V-H-L-s-2}
\begin{split}
&\left\|\vv\right\|_{\dot{B}^{s+2\alpha}_{p,p}(\RR^3)}+\left\|P\right\|_{\dot{B}^{s+1}_{p,p}(\RR^3)}\\
\leq&C_p\|\Div \ff\|_{\dot{B}^s_{p,p}(\RR^3)} +C\|\vv\|_{L^\infty(\RR^3)} \|\vv\|_{\dot{B}^{s+1}_{p,p}(\RR^3)}+C\|\vv\|_{L^p(\RR^3)}
 \|\uu_0\|_{\dot{W}^{s+1,\infty}(\RR^3)} \\
&+C\|\nabla \uu_0\|_{L^\infty(\RR^3)} \left\| \vv \right\|_{\dot{B}^{s}_{p,p}(\RR^3)}+C\|  \vv\|_{L^\infty(\RR^3)} \left\| \nabla \uu_0 \right\|_{\dot{B}^{s}_{p,p}(\RR^3)}+C_p\left\|\vv\right\|_{\dot{B}^{s}_{p,p}(\RR^3)}.
\end{split}
\end{equation}
Since $ \sigma\in C^{1,0}(\mathbb{S}^2),$ we have by Lemma \ref{lem-exe-decay} that  for each $s\geq0$,
\begin{equation}\label{eq.est-u0}
\begin{split}
\|\uu_{0}\|_{\dot{B}^s_{p,1}(\RR^3)}\leq & C\|\Delta_{-1}\uu_{0}\|_{L^p(\RR^3)}+\sum_{q\geq0}e^{-2^{qs}}\|\Delta_{q}\uu_{0}\|_{L^p(\RR^3)}\\
\leq&C\|\sigma\|_{L^\infty(\mathbb{S}^2)}\left\|1/|\cdot|^{2\alpha-1}\right\|_{L^{\frac{3}{2\alpha-1},\infty}(\RR^3)}
\end{split}
\end{equation}
for each $p>\frac{3}{2\alpha-1}$ and
\begin{equation}\label{eq.est-nablau0}
\|\nabla\uu_{0}\|_{\dot{B}^s_{q,1}(\RR^3)}\leq C\|\sigma\|_{W^{1,\infty}(\mathbb{S}^2)}\left\|1/|\cdot|^{2\alpha}\right\|_{L^{\frac{3}{2\alpha},\infty}(\RR^3)}\quad\text{for each}\,\,\,q>\frac{3}{2\alpha}.
\end{equation}
Thanks to \eqref{eq.eq-r-1}, \eqref{eq.est-u0} and \eqref{eq.est-nablau0}, we easily find by condition $\overline{\ff}\in L^2(\RR^3)\cap\dot{B}^{s}_{p,p}(\RR^3)$ that all terms in the right side of \eqref{eq-V-H-L-s-2} are bound as long as $s\leq 1-\alpha$, that is, for each $0\leq s\leq1-\alpha,$
 \begin{equation}\label{eq-V-H-L-s-3}
 \left\|\uu\right\|_{\dot{B}^{s+2\alpha}_{p,p}(\RR^3)}+\left\|P \right\|_{\dot{B}^{s+1}_{p,p}(\RR^3)}\\
\leq C\left(p,\sigma,\|\Div\ff\|_{L^2(\RR^3)\cap\dot{B}^s_{p,p}(\RR^3)}\right).
\end{equation}
With estimate \eqref{eq-V-H-L-s-3} in hand, we can show by repeating the above step  that
 for each $0\leq s\leq2(1-\alpha),$
 \begin{equation*}
 \left\|\uu\right\|_{\dot{B}^{s+2\alpha}_{p,p}(\RR^3)}+\left\|P \right\|_{\dot{B}^{s+1}_{p,p}(\RR^3)}\\
\leq C\left(p,\sigma,\|\Div\ff\|_{L^2(\RR^3)\cap\dot{B}^s_{p,p}(\RR^3)}\right).
\end{equation*}
Repeat the above process until the termination condition of Theorem \ref{thm-Lp-V-H} is reached.
\end{proof}
Taking $\sigma=0$ in Theorem \ref{thm-Lp-U-G} and Theorem \ref{thm-Lp-V-H}, they become Theorem \ref{thm-leray}, (ii)-(1). Now we come back to show the high regularity of weak solutions in Theorem \ref{thm-leray-per}. Since $\sigma\in C^{1,0}(\mathbb{S}^2),$
we have by the Leibniz estimate, \eqref{eq.est-u0} and \eqref{eq.est-nablau0} that for each $s\geq0$ and $p\geq2,$
\begin{align*}
&\left\|\Div( \uu_0\otimes\uu_0)\right\|_{\dot{B}^{s}_{p,p}(\RR^3)}\\\leq&C\|\uu_0\|_{L^{3p}(\RR^3)}\left\| \nabla \uu_0 \right\|_{\dot{B}^{s}_{3p/2,1}(\RR^3)}+C\|\nabla\uu_0\|_{L^{3p/2}(\RR^3)}\left\| \uu_0 \right\|_{\dot{B}^{s}_{3p ,1}(\RR^3)}<\infty.
\end{align*}
Taking $\Div\ff=\Div( \uu_0\otimes\uu_0)$ in Theorem \ref{thm-Lp-U-G} and Theorem \ref{thm-Lp-V-H}, we immediately obtain the high regularity of weak solutions in Theorem \ref{thm-leray-per}.
%%%%%%%%%%%%%%%%%%%%%%%%%%%%%%%%%%%%%%%%%%%%%%%%%%%%%%%%%%%%%%%%%%%%%%%%%%%%%%%%%%%%%%%%%%%%%%%%%%%
\subsection{ Regularity in the framework of the weighted Hilbert space}
%%%%%%%%%%%%%%%%%%%%%%%%%%%%%%%%%%%%%%%%%%%%%%%%%%%%%%%%%%%%%%%%%%%%%%%%%%%%%%%%%%%%%%%%%%%%%%
In this subsection, we will study the regularity in the framework of the weighted Hilbert space, which helps us to analysis the behaviour of weak solution for large $|\xx|.$
\begin{theorem}\label{thm-Lp-V-H-W}
Let $s\geq0,$  $\sigma\in C^{1,0}(\mathbb{S}^2)$ and $\ff\in H^{s+1}(\RR^3)$. Assume that $\vv$ is a weak solution to the system \eqref{E-2}. Then there exists a constant $C$ such that  \begin{itemize}
                       \item[(i)]   if $\alpha\in[\frac56,1],$ $\Div\ff\in H^{s}_{\langle\xx\rangle^{2\beta}}$, we have that for each $\beta\in(0,\alpha),$
                       \begin{equation}\label{eq-Lp-vvvv-L2}
  \left\|\vv\right\| _{H^{s+\alpha}_{\langle\xx\rangle^{2\beta}}(\RR^3)} +\left\|  P\right\| _{H^{s+\alpha}_{\langle\xx\rangle^{2\beta}}(\RR^3)}
  \leq C\Big(\sigma,\left\|\Div\ff\right\|_{H^{s}_{\langle\xx\rangle^{2\beta}}(\RR^3)}  \Big);
\end{equation}
                       \item[(ii)]   if $\alpha=1,$ $\Div\ff\in H^{s}_{\langle\xx\rangle^2}$,  we have
                       \begin{equation}\label{eq-Lp-vvvv-L2-2}
  \left\|\vv\right\| _{H^{s+1}_{\langle\xx\rangle^{2}}(\RR^3)} +\left\|  P\right\| _{H^{s+1}_{\langle\xx\rangle^{2}}(\RR^3)}
  \leq C\Big(\sigma,\left\|\Div\ff\right\|_{H^{s}_{\langle\xx\rangle^{2}}(\RR^3)}  \Big).
\end{equation}
                     \end{itemize}
\end{theorem}
\begin{proof}
First of all, we have from  Theorem \ref{thm-Lp-U-G} and Theorem \ref{thm-Lp-V-H} that for each $s\geq0,$
\begin{equation}\label{eq.s2alpha}
\|\vv\|_{H^{s+2\alpha}(\RR^3)}+\|\nabla P\|_{H^{s}(\RR^3)}\leq C\|\Div\ff\|_{H^{s}(\RR^3)}.
\end{equation}
Since $\alpha\in [5/6,1]$, we have that
\[\vv+\uu_0\in L^\infty(\RR^3)\cap \dot{W}^{1,\frac3\alpha}(\RR^3).\]
By Theorem \ref{thm-V-H-L2-weight}, we readily have
\begin{align*}
 \left\|\langle \cdot\rangle^{\beta} \vv\right\| _{H^{\alpha}(\RR^3)}
 \leq& C\left(\left\|\langle \cdot\rangle^{\beta}\vv\cdot\nabla\uu_0\right\|_{L^2(\RR^3)}+\left\|\langle \cdot\rangle^{\beta}\Div\ff\right\|_{L^2(\RR^3)}  \right)\\
  \leq& C\left(\left\|\langle \cdot\rangle^{\beta}\nabla\uu_0\right\|_{L^\infty(\RR^3)}\|\vv\|_{L^2(\RR^3)}+\left\|\langle \cdot\rangle^{\beta}\Div\ff\right\|_{L^2(\RR^3)}  \right).
 \end{align*}
This estimate together with the fact $\uu_0=\frac{\sigma}{|\xx|^{2\alpha-1}}$ and $\sigma\in C^{1,0}(\mathbb{S}^2)$ implies
\[\left\|\langle \cdot\rangle^{\beta}\nabla\uu_0\right\|_{L^\infty(\RR^3)}\leq C,\]
which together with \eqref{eq.s2alpha} gives
\begin{equation}\label{eq.s2alpha-x}
\left\|\langle \cdot\rangle^{\beta} \vv\right\| _{H^{\alpha}(\RR^3)}\leq C_\sigma\left\|\langle \cdot\rangle^{\beta}\Div\ff\right\|_{L^2(\RR^3)}.
\end{equation}
Thanks to Lemma \ref{lem-com-two} and \eqref{eq.s2alpha}, it follows from \eqref{eq.s2alpha-x} that
\begin{equation}\label{eq.s2alpha-xx}
\left\|\langle \cdot\rangle^{\beta}(I_d+\Lambda^\alpha) \vv\right\| _{L^{2}(\RR^3)}\leq C_\sigma\left\|\langle \cdot\rangle^{\beta}\Div\ff\right\|_{L^2(\RR^3)}.
\end{equation}
Next, we consider the case where $s>0.$ By Theorem \ref{thm-V-H-L2-weight-22}, we get for each $s>0,$
\begin{equation}\label{eq-II-1}
 \begin{split}
&\sup_{|\gamma|=[s]} \left\|\langle \cdot\rangle^{\beta} \Lambda^{\{s\}} D^\gamma\vv\right\| _{H^\alpha(\RR^3)}  \\
 \leq&C   \sup_{|\gamma|=[s]} \left( \left\|\langle \cdot\rangle^{\beta}(I_{d}+\Lambda^{\{s\}}D^\gamma)\Div\ff \right\|_{L^2(\RR^3)}+\left\|\langle \cdot\rangle^{\beta}(I_{d}+\Lambda^{\{s\}}D^\gamma)\widetilde\ff_1\right\|_{L^2(\RR^3)}\right)   \\
&+C \sup_{|\gamma|=[s]}   \left\|\langle \cdot\rangle^{\beta}\Lambda^{\{s\}}D^\gamma\Div\ff_2\right\|_{H^{-\alpha}(\RR^3)}+   C\left\| \ff_2\right\|_{H^{1+s}(\RR^3)}  .
\end{split}
\end{equation}
where
\[\widetilde\ff_1=-\operatorname{div}\left(\vv\otimes\uu_0\right)\quad\text{and}\quad\ff_2=-(\vv+\uu_0)\otimes\vv.\]
Since $\alpha\in[5/6,1],$  we have by the Leibniz estimate and \eqref{eq.s2alpha} that
\begin{equation}\label{eq-II}
\begin{split}
  \left\|  \ff_2\right\|_{H^{1+s}(\RR^3)}\leq& C\left\|(\vv+\uu_0)\otimes\vv\right\|_{H^{1+s}(\RR^3)}\\
\leq &C\left(\|\vv\|_{L^\infty(\RR^3)}+\|\uu_0\|_{W^{1+s,\infty}(\RR^3)}\right)\|\vv\|_{H^{1+s}(\RR^3)}\\
\leq&C\|\Div\ff\|^2_{H^s(\RR^3)}+C_\sigma\|\Div\ff\|_{H^s(\RR^3)}.
\end{split}
\end{equation}
Now we deal with the term involving $\widetilde\ff_1$. For the case $\{s\}=0$, it is obvious that
\begin{align*}
\sup_{|\gamma|=s}\left\|\langle \cdot\rangle^{\beta}(I_{d}+\Lambda^{\{s\}}D^\gamma)\widetilde\ff_1\right\|_{L^2(\RR^3)}=&\sup_{|\gamma|=s}\left\|\langle \cdot\rangle^{\beta}(I_{d}+ D^\gamma)\widetilde\ff_1\right\|_{L^2(\RR^3)}\\
\leq& \left\|\langle \cdot\rangle^{\beta} \widetilde\ff_1\right\|_{L^2 (\RR^3)}+ \sup_{|\gamma|=s}\left\|\langle \cdot\rangle^{\beta}   D^\gamma \widetilde\ff_1\right\|_{L^2(\RR^3)}.
\end{align*}
The Leibniz formula enables us to conclude that
\begin{align*}
&\sup_{|\gamma|=s}\left\|\langle \cdot\rangle^{\beta}  D^\gamma \widetilde\ff_1\right\|_{L^2(\RR^3)}\\
\leq&\sup_{|\gamma|=s}\left\| D^\gamma \left(\langle \cdot\rangle^{\beta} \widetilde\ff_1\right)\right\|_{L^2(\RR^3)}+C\sum_{|\gamma_1|=1}^s\sup_{{|\gamma_1|+|\gamma_2|=s}}\left\|  \partial^{\gamma_1}\langle \cdot\rangle^{\beta} \partial^{\gamma_2}\widetilde\ff_1 \right\|_{L^2(\RR^3)}\\
\leq& C\left\| \langle \cdot\rangle^{\beta} \widetilde\ff_1 \right\|_{\dot{H}^s(\RR^3)}+C\|\widetilde\ff_1\|_{H^{s-1}(\RR^3)}\leq   C\left\| \langle \cdot\rangle^{\beta} \widetilde\ff_1 \right\|_{H^s(\RR^3)}.
\end{align*}
As for the case $\{s\}\in(0,1)$, we know by the triangle inequality that
\begin{align*}
 \sup_{|\gamma|=[s]}\left\|\langle \cdot\rangle^{\beta}(I_{d}+\Lambda^{\{s\}}D^\gamma)\widetilde\ff_1\right\|_{L^2(\RR^3)}
\leq&  \left\|\langle \cdot\rangle^{\beta} \widetilde\ff_1\right\|_{L^2(\RR^3)}+\sup_{|\gamma|=[s]}\left\|\langle \cdot\rangle^{\beta }  \Lambda^{\{s\}}D^\gamma\widetilde \ff_1\right\|_{L^2(\RR^3)}.
\end{align*}
By Lemma  \ref{lem-com-two}, we get
\begin{align*}
 &\sup_{|\gamma|=[s]}\left\|\langle \cdot\rangle^{\beta} \Lambda^{\{s\}}D^\gamma \widetilde\ff_1\right\|_{L^2(\RR^3)}\\
 \leq& \sup_{|\gamma|=[s]}\left\| \Lambda^{\{s\}}\left(\langle \cdot\rangle^{\beta}D^\gamma\widetilde \ff_1\right)\right\|_{L^2(\RR^3)}+\sup_{|\gamma|=[s]}\left\|[\Lambda^{\{s\}},\langle \cdot\rangle^{\beta }]  D^\gamma \widetilde \ff_1\right\|_{L^2(\RR^3)}\\
\leq&  \sup_{|\gamma|=[s]}\left\| \Lambda^{\{s\}}\left(\langle \cdot\rangle^{\beta}D^\gamma \widetilde\ff_1\right)\right\|_{L^2(\RR^3)}+ C\sup_{|\gamma|=[s]}\left\| \langle \cdot\rangle^{\beta}D^\gamma \widetilde\ff_1 \right\|_{L^2(\RR^3)}.
\end{align*}
According to the Leibniz formula, we get by Leibniz estimates that
\begin{align*}
&\left\| \Lambda^{\{s\}}\left(\langle \cdot\rangle^{\beta}D^\gamma \widetilde\ff_1\right)\right\|_{L^2(\RR^3)}\\
\leq&\left\| \Lambda^{\{s\}}D^\gamma\left(\langle \cdot\rangle^{\beta}\widetilde \ff_1\right)\right\|_{L^2(\RR^3)}+C\sum_{|\gamma_1|=1}^{[s]}\sup_{{|\gamma_1|+|\gamma_2|=[s]}}\left\| \Lambda^{\{s\}}\left(\partial^{\gamma_1}\langle \cdot\rangle^{\beta}\partial^{\gamma_2} \widetilde\ff_1\right)\right\|_{L^2(\RR^3)}\\
\leq&\left\| \Lambda^{\{s\}}D^\gamma\left(\langle \cdot\rangle^{\beta}\widetilde \ff_1\right)\right\|_{L^2(\RR^3)}+C\|\widetilde\ff_1\|_{H^{s }(\RR^3)},
\end{align*}
and
\begin{align*}
 \sup_{|\gamma|=[s]}\left\|\langle \cdot\rangle^{\beta}  D^\gamma \widetilde\ff_1\right\|_{L^2(\RR^3)}
 \leq& C\left\|\langle \cdot\rangle^{\beta} \widetilde\ff_1\right\|_{\dot{H}^s(\RR^3)}+\sup_{|\gamma|=[s]}\left\|\langle \cdot\rangle^{\beta }  D^\gamma\widetilde \ff_1\right\|_{L^2(\RR^3)}.
\end{align*}
Hence, we have that for each $s>0$,
\begin{equation}\label{eq-add-s}
\sup_{|\gamma|=s}\left\|\langle \cdot\rangle^{\beta}(I_{d}+\Lambda^{\{s\}}D^\gamma)\widetilde\ff_1\right\|_{L^2(\RR^3)}\leq C\left\|\langle \cdot\rangle^{\beta}\widetilde \ff_1\right\|_{H^s(\RR^3)}.
\end{equation}
Inserting \eqref{eq-II} and \eqref{eq-add-s} into \eqref{eq-II-1} sand using $\Div\ff\in H^s(\RR^3)$, we get
\begin{equation}\label{eq-II-s}
 \begin{split}
&\sup_{|\gamma|=[s]} \left\|\langle \cdot\rangle^{\beta} \Lambda^{\{s\}} D^\gamma\vv \right\| _{H^\alpha(\RR^3)}  \\
 \leq&C   \sup_{|\gamma|=[s]}  \left\|\langle \cdot\rangle^{\beta}(I_{d}+\Lambda^{\{s\}}D^\gamma)\Div\ff\right\|_{L^2(\RR^3)}+C   \left\|\langle \cdot\rangle^{\beta} \widetilde\ff_1\right\|_{H^s(\RR^3)}+C
 \\&+C \sup_{|\gamma|=[s]}  \left\|\langle \cdot\rangle^{\beta}\Lambda^{\{s\}}D^\gamma\Div\ff_2\right\|_{H^{-\alpha}(\RR^3)}.
\end{split}
\end{equation}
Since
\[
\langle\xx\rangle^\beta\widetilde\ff_1=
-\langle\xx\rangle^\beta \vv\cdot\nabla\uu_0,
\]
 we get by the Leibniz estimate that for each $s\geq0,$
\begin{align*}
&\left\|\langle\cdot\rangle^\beta \vv\cdot\nabla\uu_0\right\|_{{H}^s(\RR^3)}\\
\leq&C\left\|\langle\cdot\rangle^\beta \vv\right\|_{{H}^s(\RR^3)}\left\|\nabla\uu_0\right\|_{L^\infty(\RR^3)}+C\left\|\langle\cdot\rangle^\beta \vv\right\|_{L^2(\RR^3)}\left\|\nabla\uu_0\right\|_{W^{s,\infty}(\RR^3)}\\
 \leq&C\|\sigma\|_{W^{s+1,\infty}(\mathbb{S}^2)}\left\|\langle\cdot\rangle^\beta \vv\right\|_{{H}^s(\RR^3)}.
\end{align*}
Hence, we have
\begin{equation}\label{eq.taiw}
 \sup_{|\gamma|=s}\left\|\langle \cdot\rangle^{\beta}(I_{d}+\Lambda^{\{s\}}D^\gamma)\widetilde\ff_1\right\|_{L^2(\RR^3)}\leq C\left\|\langle\cdot\rangle^\beta \vv\right\|_{{H}^s(\RR^3)}.
\end{equation}
We turn to deal with the term involving $\ff_2.$ As long as $\{s\}=0$, we see that
\begin{align*}
&\sup_{|\gamma|=[s]}  \left\|\langle \cdot\rangle^{\beta}\Lambda^{\{s\}}D^\gamma\Div\ff_2\right\|_{H^{-\alpha}(\RR^3)}\\
=&\sup_{|\gamma|= s }  \left\|\langle \cdot\rangle^{\beta} D^\gamma\Div\ff_2\right\|_{H^{-\alpha}(\RR^3)}\\
\leq&\sup_{|\gamma|= s }  \left\|D^\gamma\left(\langle \cdot\rangle^{\beta}\Div\ff_2\right)\right\|_{H^{-\alpha}(\RR^3)}+C\sum_{|\gamma_1|=1}^{s}\sup_{ |\gamma_1|+|\gamma_2|= s  }  \left\|\partial^{\gamma_1}\langle \cdot\rangle^{\beta}\partial^{\gamma_2}\Div \ff_2\right\|_{H^{-\alpha}(\RR^3)}.
\end{align*}
By the H\"older inequality, one has
\begin{align*}
 &\sum_{|\gamma_1|=1}^{s}\sup_{|\gamma_1|+|\gamma_2|= s } \left\|\partial^{\gamma_1}\langle \cdot\rangle^{\beta}\partial^{\gamma_2} \Div\ff_2\right\|_{H^{-\alpha}(\RR^3)}\\
 \leq& C\sum_{|\gamma_1|=1}^{s} \sup_{ |\gamma_1|+|\gamma_2|= s } \left\|\partial^{\gamma_1}\langle \cdot\rangle^{\beta}\partial^{\gamma_2} \Div\ff_2\right\|_{L^{2}(\RR^3)}\leq C\left\|\ff_2\right\|_{H^{s}(\RR^3)}.
\end{align*}
Therefore, we have that  for the case $\{s\}=0$,
\begin{equation}\label{eq-s-0}
\sup_{|\gamma|=[s]}  \left\|\langle \cdot\rangle^{\beta}\Lambda^{\{s\}}D^\gamma\Div\ff_2\right\|_{H^{-\alpha}(\RR^3)}\leq C \left\| \langle \cdot\rangle^{\beta}\Div\ff_2 \right\|_{H^{s-\alpha}(\RR^3)}+ C\left\|\ff_2\right\|_{H^{s}(\RR^3)}.
\end{equation}
For the case $\{s\}\in(0,1),$ we observe that
\begin{equation}\label{eq.-case-00}
\begin{split}
&\sup_{|\gamma|=[s]}  \left\|\langle \cdot\rangle^{\beta}\Lambda^{\{s\}}D^\gamma\Div\ff_2\right\|_{H^{-\alpha}(\RR^3)}\\
\leq&\sup_{|\gamma|=[s]}  \left\|\Div\left(\langle \cdot\rangle^{\beta}\Lambda^{\{s\}}D^\gamma\ff_2)\right)\right\|_{H^{-\alpha}(\RR^3)}+\sup_{|\gamma|=[s]}  \left\| \nabla\langle \cdot\rangle^{\beta} \cdot\Lambda^{\{s\}} D^\gamma \ff_2\right\|_{H^{-\alpha}(\RR^3)}\\
\leq&C\sup_{|\gamma|=[s]}  \left\|\langle\cdot\rangle^{\beta}\Lambda^{\{s\}}D^\gamma \ff_2\right\|_{H^{1-\alpha}(\RR^3)}+ C\sup_{|\gamma|=[s]}  \left\| \nabla\langle \cdot\rangle^{\beta} \cdot\Lambda^{\{s\}} D^\gamma \ff_2\right\|_{L^2 (\RR^3)}.
\end{split}
\end{equation}
By the H\"older inequality, we get
\begin{equation}\label{eq-add-2022-1}
\left\| \nabla\langle \cdot\rangle^{\beta} \cdot\Lambda^{\{s\}} D^\gamma \ff_2\right\|_{L^2 (\RR^3)}\leq \left\| \nabla\langle \cdot\rangle^{\beta}\right\|_{L^\infty(\RR^3)} \left\|\Lambda^{\{s\}} D^\gamma \ff_2\right\|_{L^2 (\RR^3)}\leq C\|\ff_2\|_{H^s(\RR^3)}.
\end{equation}
By Lemma \ref{lem-com-two}, we obtain
\begin{equation}\label{eq.-case-0}
\begin{split}
&\left\|\langle\cdot\rangle^{\beta}\Lambda^{\{s\}}D^\gamma \ff_2\right\|_{H^{1-\alpha}(\RR^3)}\\
\leq&\left\|\langle\cdot\rangle^{\beta}\Lambda^{\{s\}}D^\gamma \ff_2\right\|_{L^{2}(\RR^3)}+\left\|\Lambda^{1-\alpha}\left(\langle\cdot\rangle^{\beta}\Lambda^{\{s\}}D^\gamma \ff_2\right)\right\|_{L^{2}(\RR^3)}\\
\leq&\left\|\Lambda^{\{s\}}\left(\langle\cdot\rangle^{\beta}D^\gamma \ff_2\right)\right\|_{L^{2}(\RR^3)}+\left\|[\langle\cdot\rangle^{\beta},\Lambda^{\{s\}}]D^\gamma \ff_2\right\|_{L^{2}(\RR^3)}\\
&+\left\|\langle\cdot\rangle^{\beta}\Lambda^{1+\{s\}-\alpha} D^\gamma \ff_2\right\|_{L^{2}(\RR^3)}+\left\|[\Lambda^{1-\alpha},\langle\cdot\rangle^{\beta}]\left(\Lambda^{\{s\}}D^\gamma \ff_2\right)\right\|_{L^{2}(\RR^3)}\\
\leq&\left\|\Lambda^{\{s\}}\left(\langle\cdot\rangle^{\beta}D^\gamma \ff_2\right)\right\|_{L^{2}(\RR^3)}+C\left\| \langle\cdot\rangle^{\beta} D^\gamma \ff_2\right\|_{L^{2}(\RR^3)}\\
&+\left\|\langle\cdot\rangle^{\beta}\Lambda^{1+\{s\}-\alpha} D^\gamma \ff_2\right\|_{L^{2}(\RR^3)}+C\left\| \langle\cdot\rangle^{\beta} \left(\Lambda^{\{s\}}D^\gamma \ff_2\right)\right\|_{L^{2}(\RR^3)}.
\end{split}
\end{equation}
Our task is now to bound the first term in the last line of \eqref{eq.-case-0}. By Lemma \ref{lem-com-two} again,  we have that  for each $\{s\}\in(0,\alpha),$
\begin{equation}\label{eq.-case-1}
\begin{split}
&\left\|\langle\cdot\rangle^{\beta}\Lambda^{1+\{s\}-\alpha} D^\gamma \ff_2\right\|_{L^{2}(\RR^3)}
 \\
\leq&\left\|\Lambda^{1+\{s\}-\alpha}\left(\langle\cdot\rangle^{\beta} D^\gamma \ff_2\right)\right\|_{L^{2}(\RR^3)} +\left\|[\langle\cdot\rangle^{\beta},\Lambda^{1+\{s\}-\alpha}] D^\gamma \ff_2\right\|_{L^{2}(\RR^3)} \\
\leq&\left\|\Lambda^{1+\{s\}-\alpha}\left(\langle\cdot\rangle^{\beta}D^\gamma \ff_2\right)\right\|_{L^{2}(\RR^3)}+C\left\| \langle\cdot\rangle^{\beta} D^\gamma \ff_2\right\|_{L^{2}(\RR^3)}.
\end{split}
\end{equation}
Similarly, we can show that
\begin{equation}\label{eq.-case-2}
\begin{split}
\left\| \langle\cdot\rangle^{\beta} \left(\Lambda^{\{s\}}D^\gamma \ff_2\right)\right\|_{L^{2}(\RR^3)}
\leq&\left\| \Lambda^{\{s\}}\left( \langle\cdot\rangle^{\beta}D^\gamma \ff_2\right)\right\|_{L^{2}(\RR^3)}+\left\| [\langle\cdot\rangle^{\beta}, \Lambda^{\{s\}}]D^\gamma \ff_2 \right\|_{L^{2}(\RR^3)}\\
\leq&\left\|\Lambda^{\{s\}}\left(\langle\cdot\rangle^{\beta}D^\gamma \ff_2\right)\right\|_{L^{2}(\RR^3)}+C\left\| \langle\cdot\rangle^{\beta} D^\gamma \ff_2\right\|_{L^{2}(\RR^3)}.
\end{split}
\end{equation}
Inserting \eqref{eq.-case-1} and \eqref{eq.-case-2} into \eqref{eq.-case-0} leads to that for each $\{s\}\in(0,\alpha),$
\begin{equation}\label{eq.aa}
\begin{split}
&\sup_{|\gamma|=[s]}  \left\|\langle \cdot\rangle^{\beta}\Lambda^{\{s\}}D^\gamma\Div\ff_2\right\|_{H^{-\alpha}(\RR^3)}\\
\leq&\sup_{|\gamma|=[s]} \left\|\Lambda^{1+\{s\}-\alpha}\left(\langle\cdot\rangle^{\beta}D^\gamma \ff_2\right)\right\|_{L^{2}(\RR^3)}+C\sup_{|\gamma|=[s]} \left\| \langle\cdot\rangle^{\beta}D^\gamma \ff_2 \right\|_{L^{2}(\RR^3)}.
\end{split}
\end{equation}
In terms of the Leibniz rule,  one has
\begin{equation}\label{eq.bb}
\begin{split}
\left\| \langle\cdot\rangle^{\beta}D^\gamma \ff_2 \right\|_{L^{2}(\RR^3)}\leq &\left\| D^\gamma\left(\langle\cdot\rangle^{\beta} \ff_2 \right)\right\|_{L^{2}(\RR^3)}+C\sum_{|\gamma_1|=1}^{[s]}\sup_{|\gamma_1|+|\gamma_2|=[s]}\left\| \partial^{\gamma_1}\langle\cdot\rangle^{\beta}\partial^{\gamma_2} \ff_2 \right\|_{L^{2}(\RR^3)}\\
\leq &\left\| D^\gamma\left(\langle\cdot\rangle^{\beta} \ff_2 \right)\right\|_{L^{2}(\RR^3)}+C\|\ff_2\|_{H^s(\RR^3)}.
\end{split}
\end{equation}
Also, we have
\begin{equation}\label{eq.cc}
\begin{split}
&\left\| \langle\cdot\rangle^{\beta}D^\gamma \ff_2 \right\|_{\dot{H}^{1+\{s\}-\alpha}(\RR^3)}\\
\leq &\left\| D^\gamma\left(\langle\cdot\rangle^{\beta} \ff_2 \right)\right\|_{\dot{H}^{1+\{s\}-\alpha}(\RR^3)}+C\sum_{|\gamma_1|=1}^{[s]}\sup_{|\gamma_1|+|\gamma_2|=[s]}\left\| \partial^{\gamma_1}\langle\cdot\rangle^{\beta}\partial^{\gamma_2} \ff_2 \right\|_{\dot{H}^{1+\{s\}-\alpha}(\RR^3)}\\
\leq &C\left\| \langle\cdot\rangle^{\beta} \ff_2  \right\|_{\dot{H}^{1+ s -\alpha}(\RR^3)}+C\|\ff_2\|_{H^s(\RR^3)}.
\end{split}
\end{equation}
Inserting \eqref{eq.bb} and  \eqref{eq.cc} into \eqref{eq.aa} yields  that for each $\{s\}\in(0,\alpha),$
\begin{equation}\label{eq.aa-a}
 \sup_{|\gamma|=[s]}  \left\|\langle \cdot\rangle^{\beta}\Lambda^{\{s\}}D^\gamma\Div\ff_2\right\|_{H^{-\alpha}(\RR^3)}
\leq C\left\| \langle\cdot\rangle^{\beta} \ff_2  \right\|_{\dot{H}^{1+ s -\alpha}(\RR^3)}+C\|\ff_2\|_{H^{s}(\RR^3)}.
\end{equation}
Now we consider the case where $\{s\}\in[ \alpha,1),$ which means $1+\{s\}-\alpha\geq1$. In terms of Lemma \ref{weighted-h-equiv}, we immediately have
 \begin{equation}\label{eq-22-crr}
 \left\|\langle\cdot\rangle^{\beta}\Lambda^{1+\{s\}-\alpha} D^\gamma \ff_2\right\|_{L^{2}(\RR^3)}\
 \leq C\sum_{i=1}^3\left\|\langle\cdot\rangle^{\beta} \partial_{x_i}\Lambda^{ \{s\}-\alpha} D^\gamma \ff_2\right\|_{L^{2}(\RR^3)}.
 \end{equation}
 Moreover, we have by Lemma \ref{lem-com-two} that
 \begin{align*}
 &\left\|\langle\cdot\rangle^{\beta} \partial_{x_i}\Lambda^{ \{s\}-\alpha} D^\gamma \ff_2\right\|_{L^{2}(\RR^3)}\\
 \leq &\left\|\Lambda^{ \{s\}-\alpha}\left(\langle\cdot\rangle^{\beta} \partial_{x_i} D^\gamma \ff_2\right)\right\|_{L^{2}(\RR^3)}+\left\|[\Lambda^{ \{s\}-\alpha},\langle\cdot\rangle^{\beta}] \partial_{x_i} D^\gamma \ff_2\right\|_{L^{2}(\RR^3)}\\
 \leq &\left\|\Lambda^{ \{s\}-\alpha}\left(\langle\cdot\rangle^{\beta} \partial_{x_i} D^\gamma \ff_2\right)\right\|_{L^{2}(\RR^3)}+C\left\| \langle\cdot\rangle^{\beta}  \partial_{x_i} D^\gamma \ff_2\right\|_{L^{2}(\RR^3)}.
 \end{align*}
 By the Leibniz rule,  one has
 \begin{equation}\label{eq-22-crr-1}
 \begin{split}
 &\left\|\Lambda^{ \{s\}-\alpha}\left(\langle\cdot\rangle^{\beta} \partial_{x_i} D^\gamma \ff_2\right)\right\|_{L^{2}(\RR^3)}\\
 \leq&\left\|\Lambda^{ \{s\}-\alpha}\partial_{x_i} D^\gamma\left(\langle\cdot\rangle^{\beta}  \ff_2\right)\right\|_{L^{2}(\RR^3)}+C\sum_{|\gamma_1|=1}^{[s]+1} \left\|\Lambda^{ \{s\}-\alpha}\left(\partial^{\gamma_1} \langle\cdot\rangle^{\beta} \partial^{\gamma_2} \ff_2\right)\right\|_{L^{2}(\RR^3)}\\
  \leq&C\left\|  \langle\cdot\rangle^{\beta}  \ff_2 \right\|_{H^{s+1-\alpha}(\RR^3)}+C \left\| \ff_2 \right\|_{H^{s-\alpha}(\RR^3)},
 \end{split}
 \end{equation}
 and
 \begin{equation}\label{eq-22-crr-2}
 \left\| \langle\cdot\rangle^{\beta} \partial_{x_i} D^\gamma \ff_2 \right\|_{L^{2}(\RR^3)}
  \leq C\left\|  \langle\cdot\rangle^{\beta}  \ff_2 \right\|_{H^{[s]+1}(\RR^3)}+C \left\| \ff_2 \right\|_{H^{s-\alpha}(\RR^3)}.
 \end{equation}
 Inserting \eqref{eq-22-crr-1} and \eqref{eq-22-crr-2} into  \eqref{eq-22-crr}, we have that for each $\{s\}\in[ \alpha,1),$
 \begin{equation}\label{eq-add-2022-2}
  \left\|\langle\cdot\rangle^{\beta}\Lambda^{1+\{s\}-\alpha} D^\gamma \ff_2\right\|_{L^{2}(\RR^3)}\leq C\left\|  \langle\cdot\rangle^{\beta}  \ff_2 \right\|_{H^{s+1-\alpha}(\RR^3)}+C \left\| \ff_2 \right\|_{H^{s}(\RR^3)}.
 \end{equation}
Plugging \eqref{eq-add-2022-1}, \eqref{eq.aa-a} and \eqref{eq-add-2022-2} into  \eqref{eq.-case-00} yields that for each $\{s\}\in(0,1),$
\begin{equation}\label{eq.aa-aa}
 \sup_{|\gamma|=[s]}  \left\|\langle \cdot\rangle^{\beta}\Lambda^{\{s\}}D^\gamma\Div\ff_2\right\|_{H^{-\alpha}(\RR^3)}
\leq C\left\| \langle\cdot\rangle^{\beta} \ff_2  \right\|_{\dot{H}^{1+ s -\alpha}(\RR^3)}+C\|\ff_2\|_{H^{s}(\RR^3)}.
\end{equation}
Collecting both estimates \eqref{eq-s-0} and \eqref{eq.aa-aa}, we obtain that for each $s>0,$
\begin{equation}\label{eq-tai-0}
\sup_{|\gamma|=[s]}  \left\|\langle \cdot\rangle^{\beta}\Lambda^{\{s\}}D^\gamma\Div\ff_2\right\|_{H^{-\alpha}(\RR^3)}\leq C\left\|\langle \cdot\rangle^{\beta}   \ff_2\right\|_{H^{s+1-\alpha}(\RR^3)}+ C\left\|\ff_2\right\|_{H^{s}(\RR^3)}.
\end{equation}
Since $\ff_2=-(\vv+\uu_0)\otimes\vv,$ we have
\begin{equation}\label{eq-tai}
\left\|\langle \cdot\rangle^{\beta}   \ff_2\right\|_{H^{s+1-\alpha}(\RR^3)}\leq \left\|\langle \cdot\rangle^{\beta}   \vv\otimes\vv\right\|_{H^{s+1-\alpha}(\RR^3)}+\left\|\langle \cdot\rangle^{\beta}   \uu_0\otimes\vv\right\|_{H^{s+1-\alpha}(\RR^3)},
\end{equation}
and
\begin{equation}\label{eq-tai-0-1}
\left\|\ff_2\right\|_{H^{s}(\RR^3)}\leq C \|\vv\|_{L^\infty(\RR^3)}\|\vv\|_{H^s(\RR^3)}+C\|\uu_0\|_{W^{s,\infty}(\RR^3)}\|\vv\|_{H^s(\RR^3)}.
\end{equation}
Thanks to the Leibniz estimate, we get by using the fact $\alpha\in[5/6,1]$ that
\begin{equation}\label{eq-tai-1}
\begin{split}
 &\left\|\langle\cdot\rangle^\beta\vv\otimes\vv\right\|_{{H}^{s+1-\alpha}(\RR^3)}\\
\leq&C\left\|\langle\cdot\rangle^\beta\vv\right\|_{H^{1-\alpha}(\RR^3)}\| \vv\|_{H^{\frac32+s}(\RR^3)}+C\left\|\langle\cdot\rangle^\beta\vv\right\|_{H^{s+1-\alpha}(\RR^3)}\| \vv\|_{L^{\infty}(\RR^3)}\\
\leq&C\|  \vv\|_{H^{s+2\alpha}(\RR^3)}\left\|\langle\cdot\rangle^\beta\vv\right\|_{H^{s+1-\alpha}(\RR^3)}.
\end{split}
\end{equation}
Also, we have
\begin{equation}\label{eq-tai-2}
 \left\|\langle\cdot\rangle^\beta\uu_0\otimes\vv\right\|_{{H}^s(\RR^3)}
 \leq C\|  \uu_0\|_{H^{s+2\alpha}(\RR^3)}\left\|\langle\cdot\rangle^\beta\vv\right\|_{H^{s+1-\alpha}(\RR^3)}.
\end{equation}
Plugging  \eqref{eq-tai-1} and  \eqref{eq-tai-2} into  \eqref{eq-tai} leads to
\begin{equation*}
\left\|\langle \cdot\rangle^{\beta}   \ff_2\right\|_{H^{s+1-\alpha}(\RR^3)}\leq C\left\|\langle\cdot\rangle^\beta\vv\right\|_{H^{s+1-\alpha}(\RR^3)}.
\end{equation*}
Inserting this estimate and \eqref{eq-tai-0-1} into \eqref{eq-tai-0} and then using \eqref{eq.s2alpha},  we get
 \begin{equation}\label{eq-tai-00}
\sup_{|\gamma|=[s]}  \left\|\langle \cdot\rangle^{\beta}\Lambda^{\{s\}}D^\gamma\Div\ff_2\right\|_{H^{-\alpha}(\RR^3)}\leq C\left\|\langle \cdot\rangle^{\beta}   \vv \right\|_{H^{s+1-\alpha}(\RR^3)}+ C.
\end{equation}
Then we put \eqref{eq.taiw} and \eqref{eq-tai-00} into \eqref{eq-II-s} to yield
\begin{equation}\label{eq-II-ss}
 \begin{split}
&\sup_{|\gamma|=[s]} \left\|\langle \cdot\rangle^{\beta} \Lambda^{\{s\}} D^\gamma\vv\right\| _{H^\alpha(\RR^3)}  \\
 \leq&C\left\|\langle \cdot\rangle^{\beta}   \vv\right\|_{H^{s+1-\alpha}(\RR^3)}+ C   \sup_{|\gamma|=[s]}  \left\|\langle \cdot\rangle^{\beta}(I_{d}+\Lambda^{\{s\}}D^\gamma)\Div\ff\right\|_{L^2(\RR^3)}+C.
\end{split}
\end{equation}
By the same argument as used in the proof of \eqref{eq.aa-a}, we can show that
\begin{equation}\label{eq-II-w}
\sup_{|\gamma|=[s]}  \left\|\langle \cdot\rangle^{\beta}(I_{d}+\Lambda^{\{s\}}D^\gamma)\ff\right\|_{L^2(\RR^3)}\leq C\left\|\langle \cdot\rangle^{\beta} \ff\right\| _{H^{s}(\RR^3)}+C\left\|  \ff\right\| _{H^{s}(\RR^3)}.
\end{equation}
Thanks to the definition of  the inhomogeneous Hilbert space, one has
\begin{equation}\label{eq.inhongmo-def}
\left\|\langle \cdot\rangle^{\beta} \vv\right\| _{H^{s+\alpha}(\RR^3)}\leq C\left\|\langle \cdot\rangle^{\beta} \vv\right\| _{L^2(\RR^3)}+C\sup_{|\gamma|=[s+\alpha]}\left\|\Lambda^{\{s+\alpha\}}D^{\gamma}\left(\langle \cdot\rangle^{\beta} \vv\right)\right\| _{L^2 (\RR^3)}.
\end{equation}
In terms of the Leibniz rule, we find that
\begin{align*}
&\sup_{|\gamma|=[s+\alpha]}\left\|\Lambda^{\{s+\alpha\}}D^{\gamma}\left(\langle \cdot\rangle^{\beta} \vv\right)\right\| _{L^2 (\RR^3)}\\
\leq&\sup_{|\gamma|=[s+\alpha]}\left\|\Lambda^{\{s+\alpha\}}\left(\langle \cdot\rangle^{\beta} D^{\gamma}\vv\right)\right\| _{L^2 (\RR^3)}+C\sum_{|\gamma_1|=1}^{[s+\alpha]}\sup_{|\gamma_1|+|\gamma_2|=[s+\alpha]}\left\|  \partial^{\gamma_1}\langle \cdot\rangle^{\beta} \partial^{\gamma_2}\vv \right\| _{\dot{H}^{\{s+\alpha\}} (\RR^3)}\\
\leq&\sup_{|\gamma|=[s+\alpha]}\left\|\Lambda^{\{s+\alpha\}}\left(\langle \cdot\rangle^{\beta} D^{\gamma}\vv\right)\right\| _{L^2 (\RR^3)}+C\|\vv\|_{H^{s+\alpha}(\RR^3)}.
\end{align*}
For $\{s\}\in[0,1-\alpha),$ we have $[s+\alpha]=[s]$ and
\[\alpha\leq\{\alpha+ s \}= \alpha +\{s\}<1.\]
By Lemma \ref{lem-com-two}, we further get
\begin{align*}
&\left\|\Lambda^{\{s+\alpha\}}\left(\langle \cdot\rangle^{\beta} D^{\gamma}\vv\right)\right\| _{L^2 (\RR^3)}\\
\leq& \left\|\langle \cdot\rangle^{\beta}\Lambda^\alpha \Lambda^{\{s\}} D^\gamma\vv\right\| _{L^2(\RR^3)}+\left\|[\Lambda^{\alpha+\{s\}},\langle \cdot\rangle^{\beta}] D^\gamma\vv\right\| _{L^2(\RR^3)} \\
\leq&\left\|\Lambda^\alpha\left(\langle \cdot\rangle^{\beta} \Lambda^{\{s\}} D^\gamma\vv\right)\right\| _{L^2(\RR^3)}
+  \left\|[\Lambda^\alpha,\langle \cdot\rangle^{\beta} ]\Lambda^{\{s\}} D^\gamma\vv\right\| _{L^2(\RR^3)}+C\|\vv\|_{H^s(\RR^3)}\\
\leq&C\left\|\langle \cdot\rangle^{\beta} \Lambda^{\{s\}} D^\gamma\vv\right\| _{\dot{H}^\alpha(\RR^3)}
+C\|\vv\|_{H^s(\RR^3)}.
\end{align*}
Hence, we have
\begin{align*}
\sup_{|\gamma|=[s+\alpha]}\left\|\Lambda^{\{s+\alpha\}}D^{\gamma}\left(\langle \cdot\rangle^{\beta} \vv\right)\right\| _{L^2 (\RR^3)}\leq C\left\|\langle \cdot\rangle^{\beta} \Lambda^{\{s\}} D^\gamma\vv\right\| _{\dot{H}^\alpha(\RR^3)}+C\|\vv\|_{H^{s+\alpha}(\RR^3)}.
\end{align*}
Plugging this estimate into \eqref{eq.inhongmo-def} yields that for each $\{s\}\in[0,1-\alpha),$
\begin{equation}\label{eq.202208}
\begin{split}
&\left\| \langle \cdot\rangle^{\beta} \vv \right\| _{H^{s+\alpha} (\RR^3)}\\
\leq& C\left\|\langle \cdot\rangle^{\beta} \vv\right\| _{L^2(\RR^3)}+C\sup_{|\gamma|=[s]}\left\|\langle \cdot\rangle^{\beta} \Lambda^{\{s\}} D^\gamma\vv\right\| _{\dot{H}^\alpha(\RR^3)}
+C\|\vv\|_{H^s(\RR^3)}.
\end{split}
\end{equation}
For the case where $\{s\}\in[ \alpha,1),$ we have
\[[s+\alpha]=[s]+1\quad\text{and}\quad\{s+\alpha\}=\{s\}+\alpha-1. \]
In view of Lemma \ref{lem-com-two} and Lemma \ref{weighted-h-equiv}, we have
\begin{align*}
&\sup_{|\gamma|=[s+\alpha]}\left\|\Lambda^{\{s+\alpha\}}\left(\langle \cdot\rangle^{\beta} D^{\gamma}\vv\right)\right\| _{L^2 (\RR^3)}\\
\leq&\sup_{|\gamma|=[s]+1} \left\|\langle \cdot\rangle^{\beta} \Lambda^{\alpha +\{s\}-1} D^\gamma\vv\right\| _{L^2(\RR^3)}+ \sup_{|\gamma|=[s]+1} \left\|[\Lambda^{\alpha+\{s\}-1},\langle \cdot\rangle^{\beta}] D^\gamma\vv\right\| _{L^2(\RR^3)}  \\
\leq&\sup_{|\gamma|=[s]} \left\|\langle \cdot\rangle^{\beta} \Lambda^{\alpha}\Lambda^{\{s\}} D^\gamma\vv\right\| _{L^2(\RR^3)}+ \sup_{|\gamma|=[s]+1} \left\| \langle \cdot\rangle^{\beta}  D^\gamma\vv\right\| _{L^2(\RR^3)}.
\end{align*}
By Lemma  \ref{lem-com-two} again, one has
\begin{align*}
&\sup_{|\gamma|=[s]} \left\|\langle \cdot\rangle^{\beta} \Lambda^{\alpha}\Lambda^{\{s\}} D^\gamma\vv\right\| _{L^2(\RR^3)}\\
\leq& \sup_{|\gamma|=[s]} \left\| \Lambda^{\alpha}\left(\langle \cdot\rangle^{\beta}\Lambda^{\{s\}} D^\gamma\vv\right)\right\| _{L^2(\RR^3)}+\sup_{|\gamma|=[s]} \left\|[\langle \cdot\rangle^{\beta}, \Lambda^{\alpha}]\Lambda^{\{s\}} D^\gamma\vv\right\| _{L^2(\RR^3)}\\
\leq& C\sup_{|\gamma|=[s]} \left\| \langle \cdot\rangle^{\beta}\Lambda^{\{s\}} D^\gamma\vv \right\| _{\dot{H}^\alpha(\RR^3)}+\sup_{|\gamma|=[s]} \left\| \Lambda^{\{s\}} D^\gamma\vv\right\| _{L^2(\RR^3)}.
\end{align*}
By Lemma \ref{lem-com-two} and Lemma \ref{weighted-h-equiv},  we obtain
\begin{align*}
\sup_{|\gamma|=[s]+1} \left\| \langle \cdot\rangle^{\beta}  D^\gamma\vv\right\| _{L^2(\RR^3)}\leq&C\sup_{|\gamma|=[s] } \left\| \langle \cdot\rangle^{\beta} \Lambda D^\gamma\vv\right\| _{L^2(\RR^3)}\\
\leq& C\sup_{|\gamma|=[s] } \left\|\Lambda^{1-\{s\}}\left( \langle \cdot\rangle^{\beta} \Lambda^{\{s\}} D^\gamma\vv\right)\right\| _{L^2(\RR^3)}\\&+C\sup_{|\gamma|=[s] } \left\|[\Lambda^{1-\{s\}}, \langle \cdot\rangle^{\beta} ] \Lambda^{\{s\}} D^\gamma\vv\right\| _{L^2(\RR^3)}.
\end{align*}
From this estimate, it follows that
\begin{align*}
\sup_{|\gamma|=[s]+1} \left\| \langle \cdot\rangle^{\beta}  D^\gamma\vv\right\| _{L^2(\RR^3)}\leq& C\sup_{|\gamma|=[s]} \left\| \langle \cdot\rangle^{\beta}\Lambda^{\{s\}} D^\gamma\vv \right\| _{  H^{1-\{s\}} (\RR^3)}.
\end{align*}
Furthermore, we get by the fact $\alpha\in[5/6,1]$ that for each $\{s\}\in[ \alpha,1),$
\begin{align*}
\sup_{|\gamma|=[s+\alpha]}\left\|\Lambda^{\{s+\alpha\}}\left(\langle \cdot\rangle^{\beta} D^{\gamma}\vv\right)\right\| _{L^2 (\RR^3)}\leq C\sup_{|\gamma|=[s]} \left\| \langle \cdot\rangle^{\beta}\Lambda^{\{s\}} D^\gamma\vv \right\| _{  H^{1-\alpha} (\RR^3)}+C\|\vv\|_{H^s(\RR^3)}.
\end{align*}
Plugging this estimate into \eqref{eq.inhongmo-def} gives that for each $\{s\}\in[ \alpha,1),$
\begin{equation}\label{eq.202208-2}
\begin{split}
&\left\| \langle \cdot\rangle^{\beta} \vv \right\| _{H^{s+\alpha} (\RR^3)}\\
\leq& C\left\|\langle \cdot\rangle^{\beta} \vv\right\| _{L^2(\RR^3)}+C\sup_{|\gamma|=[s]}\left\|\langle \cdot\rangle^{\beta} \Lambda^{\{s\}} D^\gamma\vv\right\| _{\dot{H}^\alpha(\RR^3)}
+C\|\vv\|_{H^s(\RR^3)}.
\end{split}
\end{equation}
Inserting \eqref{eq-II-w}, \eqref{eq.202208} and \eqref{eq.202208-2} into \eqref{eq-II-ss}, we readily have for each $s>0,$
\begin{equation}\label{eq-II-sss}
 \left\|\langle \cdot\rangle^{\beta} \vv\right\| _{H^{s+\alpha}(\RR^3)}
 \leq C\left\|\langle \cdot\rangle^{\beta}   \vv\right\|_{H^{s+1-\alpha}(\RR^3)}+  C\left\|\langle \cdot\rangle^{\beta} \Div\ff\right\| _{H^{s}(\RR^3)}+C.
\end{equation}
Using estimate \eqref{eq.s2alpha-x}, we get from \eqref{eq-II-sss} that for each $s\in[0,2\alpha-1],$
\begin{equation}\label{eq-II-ssss}
 \left\|\langle \cdot\rangle^{\beta} \vv\right\| _{H^{s+\alpha}(\RR^3)}
 \leq   C_\sigma\left\|\langle \cdot\rangle^{\beta}\Div\ff\right\|_{L^2(\RR^3)}+  C\left\|\langle \cdot\rangle^{\beta}\Div \ff\right\| _{H^{s}(\RR^3)}+C,
\end{equation}
which means that for each $s\in[0,2\alpha-1],$
\begin{equation*}
 \left\|\langle \cdot\rangle^{\beta} \vv\right\| _{H^{s+\alpha}(\RR^3)}
 \leq   C\left(\sigma,\left\|\langle \cdot\rangle^{\beta}\Div \ff\right\| _{H^{s}(\RR^3)}\right).
\end{equation*}
Plugging \eqref{eq-II-ssss} into \eqref{eq-II-sss} yields that for each $s\in[2\alpha-1,2(2\alpha-1)],$
\begin{equation*}
 \left\|\langle \cdot\rangle^{\beta} \vv\right\| _{H^{s+\alpha}(\RR^3)}
 \leq    C\left(\sigma,\left\|\langle \cdot\rangle^{\beta}\Div \ff\right\| _{H^{s}(\RR^3)}\right).
\end{equation*}
By repeating this process, we can conclude that for each $s\geq0,$
\begin{equation}\label{eq.123333}
 \left\|\langle \cdot\rangle^{\beta} \vv\right\| _{H^{s+\alpha}(\RR^3)}
 \leq   C\left(\sigma,\left\|\langle \cdot\rangle^{\beta}\Div \ff\right\| _{H^{s}(\RR^3)}\right).
\end{equation}
Lastly, we turn to estimate pressure. Since
\[-\Delta P=\Div\Div(\vv\otimes\vv+\uu_0\otimes\vv+\vv\otimes\uu_0-\ff) \]
 and the fact $\langle\xx\rangle^{2\beta}$ belongs to $A_2$ class, we get by \eqref{eq.est-u0} and  \eqref{eq.s2alpha} that
\begin{align*}\left\|\langle\cdot\rangle^\beta P\right\|_{H^{s+\alpha}(\RR^3)}\leq&C \left\|\left(\langle\cdot\rangle^\beta\vv\otimes\vv+\uu_0\otimes\vv+\vv\otimes\uu_0-\ff\right)\right\|_{H^s(\RR^3)}\\
\leq &C\|\langle\cdot\rangle^\beta\vv\|_{H^{s+\alpha}(\RR^3)}+C\|\langle\cdot\rangle^\beta\ff\|_{H^{s+\alpha}(\RR^3)}<\infty.
\end{align*}
This estimate  together with \eqref{eq.123333} implies the  estimate \eqref{eq-Lp-vvvv-L2}.

The proof of the  estimate \eqref{eq-Lp-vvvv-L2-2} is similar to the proof of \eqref{eq-Lp-vvvv-L2}, so we here omit it.
\end{proof}
Now we come back to show the second estimate $\rm(ii)$-(2) in Theorem \ref{thm-leray}. Taking $\uu_0\equiv0$ in Theorem \ref{thm-Lp-V-H-W}, we immediately have
$$(\uu,P)\in H^{s+\alpha}_{\langle \xx\rangle^{2\beta}}(\RR^3)\times H_{\langle \xx\rangle^{2\beta}}^{s+\alpha}(\RR^3).$$
Lastly, we show  the weighted estimate in Theorem \ref{thm-leray-per}.  According to Theorem \ref{thm-Lp-V-H-W}, it suffice to show that for each $\alpha\in[5/6,1),$
\begin{equation}\label{eq.claim-1}
\left\|\langle\cdot\rangle^\beta\Div\ff_0\right\|_{H^{s}(\RR^3)}<\infty,
\end{equation}
and for $\alpha=1,$
\begin{equation}\label{eq.claim-1-2}
\left\|\langle\cdot\rangle\Div\ff_0\right\|_{H^{1}(\RR^3)}<\infty.
\end{equation}
 Since $\Div\ff_0=-\uu_{0}\cdot\nabla \uu_{0}$ and $\uu_0(\xx)=e^{-(-\Delta)^\alpha}\uuu_0(\xx)$ with $\uuu_{0}(\xx)=\frac{\sigma(\xx)}{|\xx|^{2\alpha-1}}$,  we have
\begin{align*}
 &\left\|\langle\cdot\rangle^\beta \Div\ff_0\right\|_{H^{s}(\RR^3)} \\
\leq&\left\|\Delta_{-1}\big(\langle\cdot\rangle^\beta \Div\ff_0\big)\right\|_{L^2(\RR^3)}+ \bigg(\sum_{k\geq0}2^{2ks}\big\|\dot{\Delta}_k\big(\langle\cdot\rangle^\beta \Div\ff_0\big)\big\|^2_{L^2(\RR^3)}\bigg)^{\frac12}.
\end{align*}
Since $\beta\in(0,\alpha)$ and $\alpha\in[5/6,1]$, we have
\[4\alpha-(\beta+1)>\frac32.\]
Moreover, one has by the H\"older inequality and Proposition 2.8 in \cite{LMZ19} that
\begin{equation}\label{eq.L-2-estimate}
\begin{split}
& \left\|\langle\cdot\rangle^\beta\Div\ff_0\right\|_{L^2(\RR^3)} \\
 \leq &\left\|\langle\cdot\rangle^{2\alpha-1}\uu_0\right\|_{L^\infty(\RR^3)} \left\|\langle\cdot\rangle^{\beta-2\alpha+1}\nabla\uu_0\right\|_{L^2(\RR^3)}\\
 \leq&\left\|\langle\cdot\rangle^{2\alpha-1}\uu_0\right\|_{L^\infty(\RR^3)} \left\|\langle\cdot\rangle^{2\alpha}\nabla\uu_0\right\|_{L^\infty(\RR^3)}\left\|\langle\cdot\rangle^{\beta- 4\alpha+1}\right\|_{L^2(\RR^3)}<\infty.
 \end{split}
\end{equation}
On the other hand, we see by the Bernstein inequality in  Lemma  \ref{lem-bern} that
\begin{align*}
& \bigg(\sum_{k\geq0}2^{2ks}\big\|\dot{\Delta}_k\big(\langle\cdot\rangle^\beta \Div\ff_0\big)\big\|^2_{L^2(\RR^3)}\bigg)^{\frac12} \\
\leq& \bigg(\sum_{k\geq0}2^{2k(s-[s]-1)} \bigg)^{\frac12}\sup_{|\gamma|=[s]+1}\big\|D^{\gamma}\big(\langle\cdot\rangle^\beta \Div\ff_0\big)\big\|_{L^2(\RR^3)}\\
\leq & C\sup_{|\gamma|=[s]+1}\big\|D^{\gamma}\big(\langle\cdot\rangle^\beta \Div\ff_0\big)\big\|_{L^2(\RR^3)}.
\end{align*}
Thanks to the Leibniz formula,  we obtain
\begin{align*}
&\sup_{|\gamma|=[s]+1}\big\|D^{\gamma}\big(\langle\cdot\rangle^\beta \Div\ff_0\big)\big\|_{L^2(\RR^3)}\\
\leq&\sup_{|\gamma|=[s]+1}\big\| \langle\cdot\rangle^\beta D^{\gamma}\Div\ff_0\big\|_{L^2(\RR^3)}+C\sum_{|\gamma_1|=1}^{[s]+1}\sup_{|\gamma_1|+|\gamma_2|=[s]+1}\big\|\partial^{\gamma_1} \langle\cdot\rangle^\beta \partial^{\gamma_2}\Div\ff_0 \big\|_{L^2(\RR^3)}.
\end{align*}
Since $\Div\ff_0=-\uu_{0}\cdot\nabla \uu_{0}$, we obtain by Proposition 2.8 in \cite{LMZ19}  that
\begin{align*}
&\sup_{|\gamma|=[s]+1}\big\| \langle\cdot\rangle^\beta D^{\gamma}\Div\ff_0\big\|_{L^2(\RR^3)}\\
\leq&C\sum_{|\gamma_1|=1}^{[s]+1}\sup_{|\gamma_1|+|\gamma_2|=[s]+1}\big\|\langle\cdot\rangle^\beta\partial^{\gamma_1} \uu_0\partial^{\gamma_2}\nabla\uu_0 \big\|_{L^2(\RR^3)}\\
 \leq&C\sum_{|\gamma_1|=1}^{[s]+1}\sup_{|\gamma_1|+|\gamma_2|=[s]+1}\left\|\langle\cdot\rangle^{2\alpha-1+|\gamma_1|}\partial^{\gamma_1} \uu_0\right\|_{L^\infty(\RR^3)} \left\|\langle\cdot\rangle^{2\alpha+|\gamma_2|}\partial^{\gamma_2}\nabla\uu_0\right\|_{L^\infty(\RR^3)}\times\\
 &\times\left\|\langle\cdot\rangle^{\beta- 4\alpha-[s]}\right\|_{L^2(\RR^3)}<\infty.
\end{align*}
By the H\"older inequality, we get
\begin{align*}
&\sum_{|\gamma_1|=1}^{[s]+1}\sup_{|\gamma_1|+|\gamma_2|=[s]+1}\big\|\partial^{\gamma_1} \langle\cdot\rangle^\beta \partial^{\gamma_2}\Div\ff_0 \big\|_{L^2(\RR^3)}\\
\leq&\sum_{|\gamma_1|=1}^{[s]+1}\sup_{|\gamma_1|+|\gamma_2|=[s]+1}\big\|\partial^{\gamma_1} \langle\cdot\rangle^\beta \big\|_{L^\infty(\RR^3)}\big\|\partial^{\gamma_2}\Div\ff_0 \big\|_{L^2(\RR^3)}\leq C\big\|\Div\ff_0 \big\|_{H^s(\RR^3)}.
\end{align*}
Since $\alpha\in[5/6,1],$ we have by the Leibniz estimate, \eqref{eq.est-u0} and \eqref{eq.est-nablau0} that
\begin{align*}
\big\|\Div\ff_0 \big\|_{H^s(\RR^3)}\leq& C\|\uu_0\|_{L^\infty(\RR^3)}\|\nabla\uu_0\|_{H^s(\RR^3)}\\
&+C\|\nabla\uu_0\|_{L^2\RR^3)}\|\uu_0\|_{W^{s,\infty}(\RR^3)}<\infty.
\end{align*}
So, we have
\begin{equation*}
\bigg(\sum_{k\geq0}2^{2ks}\big\|\dot{\Delta}_k\big(\langle\cdot\rangle^\beta \Div\ff_0\big)\big\|^2_{L^2(\RR^3)}\bigg)^{\frac12}<\infty.
\end{equation*}
This estimate together with \eqref{eq.L-2-estimate} yields the claim \eqref{eq.claim-1}. Next, by the same argument used in the proof of \eqref{eq.claim-1}, we can show the claim \eqref{eq.claim-1-2}. So, we complete the proof of Theorem \ref{thm-leray-per}.

%%%%%%%%%%%%%%%%%%%%%%%%%%%%%%%%%%%%%%%%%%%%%%%%%%%%%%%%%%%%%%%%%%%%%%%%%%%%%%%%%%%%%%%%%%%%%%%%%%%
\section{The optimal decay estimate for the forward self-similar solutions to the generalized Navier-Stokes equations  }\label{DECAY}
%%%%%%%%%%%%%%%%%%%%%%%%%%%%%%%%%%%%%%%%%%%%%%%%%%%%%%%%%%%%%%%%%%%%%%%%%%%%%%%%%%%%%%%%%%%%%%
\setcounter{section}{4}\setcounter{equation}{0}

In this section, we are going to show the optimal decay estimate for the forward self-similar solutions to the generalized Navier-Stokes equations. Firstly, the existence of  the large forward self-similar solutions
\begin{equation}\label{eq-d-1}
\uuu=t^{\frac{2\alpha-1}{2\alpha}} \uu_0(\xx/t^{\frac{1}{2\alpha}})+t^{\frac{2\alpha-1}{2\alpha}} \vv(\xx/t^{\frac{1}{2\alpha}})
 \end{equation}
 to the three-dimensional generalized Navier-Stokes system with $\alpha\in(5/8,1]$ was shown in~\cite{LMZ19}, where the homogeneous part $\uu_{0}=\frac{\sigma(\xx)}{|\xx|^{2\alpha-1}} $ with $\sigma(\xx)=\sigma(\xx/|\xx|)\in C^{1,0}(
 \mathbb{S}^{2})$, and the inhomogeneous part $\vv$ solves
\begin{equation}\label{eq.LMZ19}
\left\{
\begin{aligned}
&(-\Delta)^{\alpha}\vv- \frac{2\alpha-1}{2\alpha}\vv-\frac{1}{2\alpha}\xx\cdot \nabla \vv+\nabla P=-\vv\cdot\nabla \vv +L_{\uu_0}(\vv)-\uu_{0}\cdot\nabla \uu_{0}, \\
&\textnormal{div}\,\vv=0,
\end{aligned}\ \right.
\end{equation}
with
\[L_{\uu_0}\vv=- \uu_{0} \cdot\nabla \vv-\vv\cdot\nabla \uu_{0}.\]
Since $\sigma\in C^{1,0}(\mathbb{S}^2)$, we have by Theorem \ref{thm-Lp-U-G} that for each $2\leq p<\infty,$
\begin{equation}\label{eq.pp-est}
\|\vv\|_{B^{2\alpha}_{p,p}(\RR^3)}\leq C\left(\|\sigma\|_{W^{1,\infty}(\mathbb{S}^2)}\right).
\end{equation}
Thus, our task is devoted to showing the optimal  decay estimate for the solutions $\vv$ to the problem \eqref{eq.LMZ19}. To do this, we firstly construct the approximate solutions $(\vv_k,\,P_k)$ to the corresponding linearized equations
 \begin{align}\label{E-Approx-k}
\left\{
\begin{aligned}
 &(-\Delta)^{\alpha}\vv_k- \frac{2\alpha-1}{2\alpha}\vv_k-\frac{1}{2\alpha}\xx\cdot \nabla \vv_k+\nabla P_k= \mathbf{f}_k, \\
& \textnormal{div}\,\vv_k=0,
\end{aligned}\ \right.
\end{align}
where  \[\mathbf{f}_k(\xx)=-\phi_k(\xx)\Big(\vv\otimes \vv_k+\mathbf{U}_{0}\otimes \vv_k+\vv_k\otimes \mathbf{U}_{0}+\mathbf{U}_{0}\otimes \mathbf{U}_{0}\Big)(\xx)\]
and $\phi_k= \phi(\xx/k)$ with the smooth cut-off function $\phi$ satisfying
\begin{equation*}
\phi(\xx)=\begin{cases}
1\quad&\xx\in B_{\frac{R}{2}}(0),\\
0\quad&\xx\in \RR^3\backslash B_{R }(0).
\end{cases}
\end{equation*}
By the standard degree theory as used in \cite{LMZ19}, it is easy to prove that Cauchy problem~\eqref{E-Approx-k} admits one weak solutions  $\vv_k\in H^\alpha_{\sigma}(\RR^3)$  such that
\begin{equation*}
\|\vv_k\|_{H^\alpha(\RR^3)}\leq C(\|\sigma_0\|_{W^{1,\infty}(\mathbb{S}^2)} ).
\end{equation*}
Since $\operatorname{supp}\mathbf{f}_k\subset B_{k}(0)$, we can write the weak solutions  $\vv_k$  in the integral form in terms of Proposition~4.1 established in \cite{LMZ21} that
\begin{equation}\label{eq.fundamental-solution-k}
\vv_k(\xx)=\int_{0}^{1}e^{-(-\Delta)^{\alpha}(1-s)}\mathbb{P}{\rm div}_{\xx}\left(s^{\frac{1}{\alpha}-2}\mathbf{f}_k\big(\cdot/s^{\frac{1}{2\alpha}}\big)\right)\,{\rm d}s.
\end{equation}
Next, we are going to show the optimal decay estimate for the $\vv_k$  in two steps, which implies the optimal decay estimate for  $\vv$ by taking $k\to\infty$. Let us begin with the lower  decay estimate for $\vv.$
\subsection[The  lower  decay estimate]{The  lower  decay estimate} In this subsection, we proceed  with the lower decay estimate for the $\vv_k$ by using boundedness of the operator $\nabla\mathbb{P}e^{-(-\Delta)^{\alpha}s}$ in the weighted space. Specifically,
\begin{proposition}\label{Prop-optimal-decay-low}
Let
\begin{equation}\label{eq.fundamental-solution}
\ww(\xx)=\int_{0}^{1}e^{-(-\Delta)^{\alpha}(1-s)}\mathbb{P}{\rm div}_{\xx}\left(s^{\frac{1}{\alpha}-2}\mathbf{f}\big(\cdot/s^{\frac{1}{2\alpha}}\big)\right)\,{\rm d}s.
\end{equation}
\begin{enumerate}
  \item[\rm (i)]  If $\alpha\in[5/6,1]$ and $\mathbf{f}(\xx)\in L^6(\RR^3)$ satisfies
\begin{equation*}
\sup_{\xx\in\RR^3}\left|| \xx|^{4\alpha-2}\mathbf{f}(\xx)\right| <\infty,
\end{equation*}
 we have
\begin{equation}\label{eq-est-decay-2}
\sup_{\xx\in\RR^3}\langle\xx\rangle^{4\alpha-2}|\ww|(\xx)\leq C\|\mathbf{f}\|_{L^6(\RR^3)}+C\sup_{\xx\in\RR^3}\left(|\xx|^{4\alpha-2}|\mathbf{f}|(\xx)\right).
\end{equation}
  \item  [\rm (ii)]  If $\alpha\in[5/6,1)$ and $ \mathbf{f}(\xx)\in  \dot{H}^1(\RR^3)\cap \dot{W}^{1,6}(\RR^3)$   satisfies
\begin{equation*}
\sup_{\xx\in\RR^3}\left|| \xx|^{4\alpha-1}\operatorname{div}\mathbf{f}(\xx)\right| <\infty,
\end{equation*}
 we have
\begin{equation}\label{eq-est-decay-2-02}
\sup_{\xx\in\RR^3}\langle\xx\rangle^{4\alpha-1}|\nabla\otimes\ww|(\xx)\leq C\|\operatorname{div}\mathbf{f}\|_{L^2(\RR^3)\cap L^6(\RR^3)}+C\sup_{\xx\in\RR^3}\left(|\xx|^{4\alpha-1}|\operatorname{div}\mathbf{f}|(\xx)\right).
\end{equation}
 \item  [\rm (iii)]  If $\alpha=1$ and $\mathbf{f}(\xx)\in \dot{W}^{1,\frac54}(\RR^3)\cap \dot{W}^{1,6}(\RR^3)$ satisfies
\begin{equation*}
\sup_{\xx\in\RR^3}\left|| \xx|^{3}\operatorname{div}\mathbf{f}(\xx)\right| <\infty,
\end{equation*}
 we have
\begin{equation}\label{eq-est-decay-2-03}
\sup_{\xx\in\RR^3}\langle\xx\rangle^{3}|\nabla\otimes\ww|(\xx)\leq C\|\operatorname{div}\mathbf{f}\|_{L^{\frac54}(\RR^3)\cap L^6(\RR^3)}+C\sup_{\xx\in\RR^3}\left(\langle\xx\rangle^{3}|\operatorname{div}\mathbf{f}|(\xx)\right).
\end{equation}
\end{enumerate}

\end{proposition}
\begin{proof}
From \eqref{eq.fundamental-solution}, we know that
\begin{equation*}
\ww(\xx) =\int_{0}^{1}\int_{\RR^3}\nabla\big( O_{\alpha}(1-s,\xx-\yy) \big) \left(s^{\frac{1}{\alpha}-2}\mathbf{f}\big(\yy/s^{\frac{1}{2\alpha}}\big)\right)\,\mathrm{d}\yy{\rm d}s,
\end{equation*}
where $O_\alpha$ is the generalized Oseen kernel of operator $\mathbb{P}G_\alpha$ with $G_\alpha=e^{-(-\Delta)^{\alpha}s}$.

Since $\alpha\in[5/6,1]$ and $\mathbf{f}\in L^6(\RR^3),$ we have by the Young inequality that
\begin{equation}\label{eq-est-infty}
\begin{split}
\sup_{\xx\in\RR^3}|\ww|(\xx)\leq &\int_0^1s^{-\frac{1}{2\alpha}}\|(\nabla O_\alpha)(1-s,\cdot)\|_{L^{\frac65}(\RR^3)}\left\| \mathbf{f}\big(\cdot/s^{\frac{1}{2\alpha}}\big)\right\|_{L^6(\RR^3)}s^{\frac{1}{\alpha}-2}\,\mathrm{d}s\\
\leq&C\left\| \mathbf{f} \right\|_{L^6(\RR^3)}
\int_0^1 (1-s)^{-\frac{3}{4\alpha}}s^{\frac{5}{4\alpha}-2}\,\mathrm{d}s\leq C\left\| \mathbf{f} \right\|_{L^6(\RR^3)}.
\end{split}
\end{equation}
Noting that $O_{q,\alpha}=\dot{\Delta}_q O_{\alpha}$,  we rewrite $\ww $  as
\begin{align*}
\ww(\xx)=&\int_0^1\int_{\RR^3}\nabla O_\alpha(1-s,\xx-\yy)\cdot\mathbf{f}\big(\cdot/s^{\frac{1}{2\alpha}}\big)\,\mathrm{d}\yy \,s^{\frac{1}{\alpha}-2}\,\mathrm{d}s\\
=&\left(\sum_{q<0}+\sum_{q\geq0}\right)\int_0^1\int_{\RR^3}\nabla O_{q,\alpha}(1-s,\xx-\yy)\cdot\mathbf{f}\big(\yy/s^{\frac{1}{2\alpha}}\big)\,\mathrm{d}\yy \,s^{\frac{1}{\alpha}-2}\,\mathrm{d}s\\
=:&\ww^\natural+\ww^\sharp.
\end{align*}
We first proceed  to bound $ \ww^\sharp$, which can be decomposed   into two parts
\begin{align*}
 \ww^\sharp=&\sum_{q\geq0} \int_0^1\bigg(\int_{B_{\frac{|\xx|}{2}}(0)}+\int_{\RR^3\backslash B_{\frac{|\xx|}{2}}(0)}\bigg)\nabla O_{q,\alpha}(1-s,\xx-\yy)\cdot\mathbf{f}\big(\yy/s^{\frac{1}{2\alpha}}\big)\,\mathrm{d}\yy \,s^{\frac{1}{\alpha}-2}\,\mathrm{d}s\\
 =:& \ww^\sharp_1+ \ww^\sharp_2.
\end{align*}
Since $\alpha\in[5/6,1]$, we have
\begin{equation}\label{eq-est-sharp-2}
\begin{split}
\big|\ww^\sharp_2\big|\leq &\frac{C}{|\xx|^{4\alpha-2}}\sum_{q\geq0} \int_0^1 \int_{\RR^3 } |\nabla O_{q,\alpha}(1-s,\xx-\yy)|\,\left||\yy|^{4\alpha-2}\mathbf{f}\right|\big(\yy/s^{\frac{1}{2\alpha}}\big)\,\mathrm{d}\yy \,s^{\frac{1}{\alpha}-2}\,\mathrm{d}s\\
\leq&C\sup_{\xx\in\RR^3}\left(|\xx|^{4\alpha-2}|\mathbf{f}|(\xx)\right)\,\frac{1}{|\xx|^{4\alpha-2}}\sum_{q\geq0} \int_0^1 2^qe^{-c2^{2q\alpha}(1-s)}\,\mathrm{d}s\\
\leq& C\sup_{\xx\in\RR^3}\left(|\xx|^{4\alpha-2}|\mathbf{f}|(\xx)\right)\,\frac{1}{|\xx|^{4\alpha-2}}.
\end{split}
\end{equation}
For the term $\ww^\sharp_1$, we see that
\begin{align*}
\big|\ww^\sharp_1\big|\leq \frac{C}{|\xx|^{3}}\sum_{q\geq0} \int_0^1 \int_{B_{\frac{|\xx|}{2}}(0)}|\xx-\yy|^{3} |\nabla O_{q,\alpha}|(1-s,\xx-\yy)\cdot\mathbf{f}\big(\yy/s^{\frac{1}{2\alpha}}\big)\,\mathrm{d}\yy \,s^{\frac{1}{\alpha}-2}\,\mathrm{d}s.
\end{align*}
By the Young inequality and Lemma \ref{lem-exe-decay}, we have
\begin{equation}\label{eq-est-sharp-1}
\begin{split}
\big|\ww^\sharp_1\big|\leq& \frac{C}{|\xx|^{3}}\sup_{\xx\in\RR^3}\left(|\xx|^{4\alpha-2}|\mathbf{f}|(\xx)\right)\sum_{q\geq0} \int_0^1 2^{q}e^{-c2^{2q\alpha}(1-s)}\,\mathrm{d}s\int_{B_{\frac{|\xx|}{2}}(0)}\frac{1}{|\yy|^{4\alpha-2}}\,\mathrm{d}\yy\\
\leq& C\sup_{\xx\in\RR^3}\left(|\xx|^{4\alpha-2}|\mathbf{f}|(\xx)\right)\,\frac{1}{|\xx|^{4\alpha-2}}.
\end{split}
\end{equation}
Now we turn to bound $ \vv^\natural$. Firstly, one writes
\begin{align*}
 \ww^\natural=&\sum_{q<0} \int_0^1\bigg(\int_{B_{\frac{|\xx|}{2}}(0)}+\int_{\RR^3\backslash B_{\frac{|\xx|}{2}}(0)}\bigg)\nabla O_{q,\alpha}(1-s,\xx-\yy)\cdot\mathbf{f}\big(\yy/s^{\frac{1}{2\alpha}}\big)\,\mathrm{d}\yy \,s^{\frac{1}{\alpha}-2}\,\mathrm{d}s\\
 =:& \ww^\natural_1+ \ww^\natural_2.
\end{align*}
Similar to  the proof of \eqref{eq-est-sharp-2}, we can show that
\begin{equation}\label{eq-est-natural-2}
\begin{split}
\big|\ww^\natural_2\big|\leq &\frac{C}{|\xx|^{4\alpha-2}}\sum_{q<0} \int_0^1 \int_{\RR^3\backslash B_{\frac{|\xx|}{2}}(0) } |\nabla O_{q,\alpha}(1-s,\xx-\yy)|\times\\
&\qquad\qquad\qquad \qquad\qquad\quad \times\left||\yy|^{4\alpha-2}\mathbf{f}\right|\big(\yy/s^{\frac{1}{2\alpha}}\big)\,\mathrm{d}\yy \,s^{\frac{1}{\alpha}-2}\,\mathrm{d}s\\
\leq&C\sup_{\xx\in\RR^3}\left(|\xx|^{4\alpha-2}|\mathbf{f}|(\xx)\right)\,\frac{1}{|\xx|^{4\alpha-2}}\sum_{q<0}2^q \int_0^1 \,\mathrm{d}s\\
\leq& C\sup_{\xx\in\RR^3}\left(|\xx|^{4\alpha-2}|\mathbf{f}|(\xx)\right)\,\frac{1}{|\xx|^{4\alpha-2}}.
\end{split}
\end{equation}
Lastly, we obtain
\begin{equation}\label{eq-est-natural-1}
\begin{split}
\big|\ww^\natural_1\big|\leq& \frac{C}{|\xx|^{3}}\sup_{\xx\in\RR^3}\left(|\xx|^{4\alpha-2}|\mathbf{f}|(\xx)\right)\sum_{q<0}2^{q} \int_0^1 \,\mathrm{d}s\int_{B_{\frac{|\xx|}{2}}(0)}\frac{1}{|\yy|^{4\alpha-2}}\,\mathrm{d}\yy\\
\leq& C\sup_{\xx\in\RR^3}\left(|\xx|^{4\alpha-2}|\mathbf{f}|(\xx)\right)\,\frac{1}{|\xx|^{4\alpha-2}}.
\end{split}
\end{equation}
Collecting estimates \eqref{eq-est-sharp-2}, \eqref{eq-est-sharp-1}, \eqref{eq-est-natural-2} and \eqref{eq-est-natural-1}, we readily have
\begin{equation*}
\sup_{|\xx|\geq1}|\xx|^{4\alpha-2}|\ww|(\xx)\leq C\sup_{\xx\in\RR^3}\left(|\xx|^{4\alpha-2}|\mathbf{f}|(\xx)\right).
\end{equation*}
This inequality together with \eqref{eq-est-infty} leads to
\begin{equation*}
\sup_{\xx\in\RR^3}\langle\xx\rangle^{4\alpha-2}|\ww|(\xx)\leq C\| \mathbf{f}\|_{L^6(\RR^3)}+C\sup_{\xx\in\RR^3}\langle\xx\rangle^{4\alpha-2}|\mathbf{f}|(\xx),
\end{equation*}
which implies the first desired estimate in Proposition \ref{Prop-optimal-decay-low}.

Next, we turn to show decay estimate for $\nabla\otimes\ww.$ Taking derivative of $\ww$ with respect to $x_k$, it follows from \eqref{eq.fundamental-solution} that  $\forall\,k\in\{1,2,3\},$
\begin{equation}\label{eq.fundamental-solution-der}
\partial_{x_k}\ww(\xx)=\int_{0}^{1}\int_{\RR^3}\partial_{x_k}\big( O_{\alpha}(1-s,\xx-\yy) \big) {\rm div} \left(s^{\frac{1}{\alpha}-2}\mathbf{f}\big(\yy/s^{\frac{1}{2\alpha}}\big)\right)\,\mathrm{d}\yy{\rm d}s.
\end{equation}
By the Young inequality, one has
\begin{equation}\label{eq-der-1}
\begin{split}
\sup_{\xx\in\RR^3}\left|\partial_{x_k}\ww\right|(\xx)\leq &\int_0^{\frac12}\left\|\partial_{x_k}\big( O_{\alpha}(1-s,\cdot) \big)\right\|_{L^2(\RR^3)}\left\|\operatorname{div} \mathbf{f}\big(\cdot/s^{\frac{1}{2\alpha}}\big)\right\|_{L^2(\RR^3)}s^{\frac{1}{2\alpha}-2}\,\mathrm{d}s\\
 &+\int_{\frac12}^1\left\|\partial_{x_k}\big( O_{\alpha}(1-s,\cdot) \big)\right\|_{L^{\frac65}(\RR^3)}\left\|\operatorname{div} \mathbf{f}\big(\cdot/s^{\frac{1}{2\alpha}}\big)\right\|_{L^6(\RR^3)}s^{\frac{1}{2\alpha}-2}\,\mathrm{d}s\\
 \leq &C\int_0^{\frac12} \left\|\operatorname{div} \mathbf{f} \right\|_{L^2(\RR^3)}s^{\frac{5}{4\alpha}-2}\,\mathrm{d}s
 +C\int_{\frac12}^1 \left\|\operatorname{div} \mathbf{f} \right\|_{L^6(\RR^3)}(1-s)^{-\frac{3}{4\alpha} }\,\mathrm{d}s\\
\leq & C\left(\left\|\operatorname{div} \mathbf{f} \right\|_{L^2(\RR^3)}+\left\|\operatorname{div} \mathbf{f} \right\|_{L^6(\RR^3)}\right).
\end{split}
\end{equation}
According to \eqref{eq.fundamental-solution-der}, we decompose $\ww$  into the following two parts
\begin{align*}
\partial_{x_k}\ww =&\int_{0}^{1}\bigg(\int_{B_{\frac{|\xx|}{2}}(0)}+\int_{\RR^3\backslash B_{\frac{|\xx|}{2}}(0)}\bigg)\partial_{x_k}\big( O_{\alpha}(1-s,\xx-\yy) \big) \operatorname{div}\mathbf{f}\big(\yy/s^{\frac{1}{2\alpha}}\big) s^{\frac{1}{2\alpha}-2}\,\mathrm{d}\yy{\rm d}s\\
=:&\ww_{kS}+\ww_{kL}.
\end{align*}
Noting that
\begin{align*}
\left|\ww_{kL}\right|\leq &\frac{ C}{|\xx|^{4\alpha-1}}\int_{0}^{1}  \int_{\RR^3\backslash B_{\frac{|\xx|}{2}}(0)}\big| \partial_{x_k}\big( O_{\alpha}(1-s,\xx-\yy) \big) \big| |\yy|^{4\alpha-1}\operatorname{div}\mathbf{f}\big(\yy/s^{\frac{1}{2\alpha}}\big) s^{\frac{1}{2\alpha}-2}\,\mathrm{d}\yy{\rm d}s,
\end{align*}
we get by the Young inequality that
\begin{equation}\label{eq-der-2}
\begin{split}
\left|\ww_{kL}\right|(\xx)\leq &\frac{ C}{|\xx|^{4\alpha-1}}\|\partial_{x_k}O_\alpha(1,\cdot)\|_{L^1(\RR^3)}\sup_{\xx\in\RR^3}|\xx|^{4\alpha-1}|\operatorname{div}\mathbf{f}|(\xx)\int_0^1(1-s)^{-\frac{1}{2\alpha}}\,\mathrm{d}s\\
\leq &\frac{ C}{|\xx|^{4\alpha-1}} \sup_{\xx\in\RR^3}|\xx|^{4\alpha-1}|\operatorname{div}\mathbf{f}|(\xx),
\end{split}
\end{equation}
in the last line of \eqref{eq-der-2}, we have used the fact that $\partial_{x_k}O_\alpha(1,\xx)\in L^1(\RR^3)$. Indeed, we have by the Bernstein inequality in  Lemma  \ref{lem-bern} that
\begin{align*}
\|\partial_{x_k}O_\alpha(1,\cdot)\|_{L^1(\RR^3)}\leq& C\sum_{q\in\ZZ}2^{q}\|\dot{\Delta}_q\mathbb{P}G_\alpha(1,\cdot)\|_{L^1(\RR^3)}\\
\leq& C\sum_{q\leq 0}2^{q}\|\dot{\Delta}_qG_\alpha(1,\cdot)\|_{L^1(\RR^3)}+ C\sum_{q>0}2^{-q}\|D^2\dot{\Delta}_qG_\alpha(1,\cdot)\|_{L^1(\RR^3)}\\
\leq&C\|G_\alpha(1,\cdot)\|_{W^{2,1}(\RR^3)}<\infty.
\end{align*}
As for  the term $\ww_{kS},$ we write in terms of the high-low decomposition
\begin{align*}
\ww_{kS}=&\left(\sum_{q\leq0}+\sum_{q>0}\right)\int_{0}^{1}\int_{B_{\frac{|\xx|}{2}}(0)}\partial_{x_k}\big( O_{q,\alpha}(1-s,\xx-\yy) \big) \operatorname{div}\mathbf{f}\big(\yy/s^{\frac{1}{2\alpha}}\big) s^{\frac{1}{2\alpha}-2}\,\mathrm{d}\yy{\rm d}s\\
=:&\ww_{kS}^\natural+\ww_{kS}^\sharp.
\end{align*}
For the high frequency regime, we easily find that
\begin{align*}
\left|\ww_{kS}^\sharp\right|\leq\frac{ C}{|\xx|^{3}}\sum_{q\geq0}\int_{0}^{1}  \int_{ \RR^3} &|\xx-\yy|^{3}\big| \partial_{x_k}\big( O_{q,\alpha}(1-s,\xx-\yy) \big)\big|\times\\
 &\times\big| \operatorname{div}\mathbf{f}\big(\yy/s^{\frac{1}{2\alpha}}\big) \big| s^{\frac{1}{2\alpha}-2}\,\mathrm{d}\yy{\rm d}s.
\end{align*}
Moreover, we have by Lemma \ref{lem-exe-decay} and $\alpha\in[5/6,1)$ that
\begin{equation}\label{eq-der-3}
\begin{split}
\big|\ww_{kS}^\sharp\big|\leq& \frac{C}{|\xx|^{3}}\sup_{\xx\in\RR^3}\left(|\xx|^{4\alpha-1}|\operatorname{div}\mathbf{f}|(\xx)\right)\sum_{q\geq0}2^{q} \int_0^1 e^{-c2^{2q\alpha}(1-s)} \,\mathrm{d}s\int_{B_{\frac{|\xx|}{2}}(0)}\frac{1}{|\yy|^{4\alpha-1}}\,\mathrm{d}\yy\\
\leq& C\sup_{\xx\in\RR^3}\left(|\xx|^{4\alpha-2}|\operatorname{div}\mathbf{f}|(\xx)\right)\,\frac{1}{|\xx|^{4\alpha-1}}.
\end{split}
\end{equation}
For the low frequency regime, we see that
 \begin{equation}\label{eq-der-4}
 \begin{split}
\big|\ww_{kS}^\natural\big|\leq \frac{C}{|\xx|^{3}}\sum_{q\leq0} \int_0^1 \int_{B_{\frac{|\xx|}{2}}(0)}&|\xx-\yy|^{3} |\partial_{x_k} \left( O_{q,\alpha}(1-s,\xx-\yy)\right)|\times\\
&\times\left|\operatorname{div}\mathbf{f}\big(\yy/s^{\frac{1}{2\alpha}}\big)\right|\,\mathrm{d}\yy \,s^{\frac{1}{2\alpha}-2}\,\mathrm{d}s.
\end{split}
\end{equation}
Thus we have that for  each $\alpha\in[5/6,1),$
\begin{align*}
\big|\ww_{kS}^\natural\big|\leq& \frac{C}{|\xx|^{3}}\sup_{\xx\in\RR^3}\left(|\xx|^{4\alpha-1}|\operatorname{div}\mathbf{f}|(\xx)\right)\sum_{q<0}2^{q} \int_0^1 \,\mathrm{d}s\int_{B_{\frac{|\xx|}{2}}(0)}\frac{1}{|\yy|^{4\alpha-1}}\,\mathrm{d}\yy\\
\leq& C\sup_{\xx\in\RR^3}\left(|\xx|^{4\alpha-2}|\operatorname{div}\mathbf{f}|(\xx)\right)\,\frac{1}{|\xx|^{4\alpha-1}}.
\end{align*}
Combining this estimate with \eqref{eq-der-1}, \eqref{eq-der-2} and \eqref{eq-der-4} yields the second desired estimate.

For $\alpha=1,$ we apply the Young inequality and Lemma \ref{lem-exe-decay} to $\ww_{kS}^\sharp$ to obtain
\begin{align*}
&\sum_{q>0}\int_{0}^{1}  \int_{ \RR^3}|\xx-\yy|^{3}\big| \partial_{x_k}\big( O_{q,1}(1-s,\xx-\yy) \big) \big| \operatorname{div}\mathbf{f}\big(\yy/s^{\frac{1}{2 }}\big) s^{\frac{1}{2 }-2}\,\mathrm{d}\yy{\rm d}s\\
\leq&C\sum_{q>0}\int_{0}^{1}  \int_{ \RR^3}\left\||\cdot|^{3} \partial_{x_k}  O_{q,1}(1-s,\cdot)  \right\|_{L^2(\RR^3)}\left \|\operatorname{div}\mathbf{f}\big(\yy/s^{\frac{1}{2 }}\big) \right\|_{L^2(\RR^3)} s^{\frac{1}{2 }-2}\, {\rm d}s\\
\leq&C\|\operatorname{div}\mathbf{f}\|_{L^2(\RR^3)} \sum_{q>0}2^{-\frac{q}{2}}\int_0^1s^{-\frac{3}{4 } }\,\mathrm{d}s\leq C\|\operatorname{div}\mathbf{f}\|_{L^2(\RR^3)},
\end{align*}
which implies that for $\alpha=1,$
\begin{equation}\label{eq-der-33}
\big|\ww_{kS}^\sharp\big|(\xx)\leq  \frac{ C}{|\xx|^{3}}\|\operatorname{div}\mathbf{f}\|_{L^2(\RR^3)}.
\end{equation}
Similarly, we have from  \eqref{eq-der-3} that for $\alpha=1,$
\begin{equation}\label{eq-der-5}
\begin{split}
\big|\ww_{kS}^\natural\big|\leq& \frac{C}{|\xx|^{3}}\sum_{q\leq0} \int_0^1  \left\||\cdot|^{3} |\partial_{x_k} \left( O_{q,1}(1-s,\cdot)\right)|\right\|_{L^5(\RR^3)}\times\\
&\qquad\qquad\qquad\times\left\|\operatorname{div}\mathbf{f}\big(\cdot/s^{\frac{1}{2 }}\big)\right\|_{L^{\frac54}(\RR^3)} \,s^{\frac{1}{2 }-2}\,\mathrm{d}s\\
\leq& \frac{C}{|\xx|^{3}}\sum_{q\leq0}2^{q\frac25}  \left\|\operatorname{div}\mathbf{f}\right\|_{L^{\frac54}(\RR^3)} \, \int_0^1 s^{-\frac{3}{10}}\,\mathrm{d}s
\leq \frac{C}{|\xx|^{3}}\left\|\operatorname{div}\mathbf{f}\right\|_{L^{\frac54}(\RR^3)}.
\end{split}
\end{equation}
Collecting estimates \eqref{eq-der-1}, \eqref{eq-der-2}, \eqref{eq-der-33} and \eqref{eq-der-5} give the third desired estimate, and then we finish the proof of Proposition \ref{Prop-optimal-decay-low}.
\end{proof}
We apply Proposition \ref{Prop-optimal-decay-low} to $\vv_k$ to obtain
\begin{equation}\label{eq-est-decay-2-k}
\sup_{\xx\in\RR^3}\langle\xx\rangle^{4\alpha-2}|\vv_k|(\xx)\leq  C\| \mathbf{f}_k\|_{L^6(\RR^3)}+C\sup_{\xx\in\RR^3}\left(|\xx|^{4\alpha-2}|\mathbf{f}_k|(\xx)\right).
\end{equation}
Since $\sigma\in C^{1,0}(\mathbb{S}^2)$, we get by using Theorem \ref{thm-Lp-U-G} that for all $p\in[2,\infty)$
\begin{equation}\label{eq.vkh2}
 \|\vv_k\|_{B_{p,p}^{2\alpha}(\RR^3)}\leq C(\|\sigma\|_{W^{1,\infty}(\mathbb{S}^2)}).
\end{equation}
This estimate enables us to conclude by the H\"older inequality, \eqref{eq.est-u0} and \eqref{eq.est-nablau0} that
\begin{align*}
\|\mathbf{f}_k\|_{L^6(\RR^3)}\leq &C\left(\|\vv_k\|_{L^\infty(\RR^3)}+\|\mathbf{U}_0\|_{L^\infty(\RR^3)}\right)\|\nabla \vv\|_{L^2(\RR^3)}\\
&+C\left(\|\vv \|_{L^\infty(\RR^3)}+\|\mathbf{U}_0\|_{L^\infty(\RR^3)}\right)\|\nabla  \mathbf{U}_0\|_{L^2(\RR^3)}\\
\leq&C\|\vv_k\|_{H^{2\alpha}(\RR^3)}^2+C\|\vv\|^2_{H^{2\alpha}(\RR^3)}+C\|\mathbf{U}_0\|_{L^\infty(\RR^3)}^2+C\|\nabla \mathbf{U}_0\|_{L^2(\RR^3)}^2\\
\leq& C(\|\sigma\|_{W^{1,\infty}( \mathbb{S}^2)}).
\end{align*}
By the H\"older inequality, one has
\begin{equation}\label{eq-R-1}
\begin{split}
&\sup_{\xx\in\RR^3} \langle\xx\rangle^{4\alpha-2}| \vv\otimes\vv_k|(\xx)\\
\leq &R^{4\alpha-2}\|\vv\|_{L^\infty(\RR^3)}\|\vv_k\|_{L^\infty(\RR^3)}+\|\vv\|_{L^\infty(\RR^3\backslash B_R(0))}\sup_{|\xx|\geq R} \langle\xx\rangle^{4\alpha-2}|\vv_k|(\xx),
\end{split}
\end{equation}
where $R$ to be fixed later.

By the H\"older inequality, the interpolation inequality and  Proposition 2.8 in \cite{LMZ19}, we get
\begin{align*}
&\sup_{\xx\in\RR^3} \langle\xx\rangle^{4\alpha-2}| \mathbf{U}_0\otimes\vv_k+\mathbf{U}_0\otimes\vv_k+ \mathbf{U}_0\otimes\mathbf{U}_0|(\xx)\\
\leq &\sup_{\xx\in\RR^3} \langle\xx\rangle^{2\alpha-1}|\mathbf{U}_0|(\xx)\sup_{\xx\in\RR^3} \langle\xx\rangle^{2\alpha-1}|\vv_k|(\xx)+\left(\sup_{\xx\in\RR^3} \langle\xx\rangle^{2\alpha-1}|\mathbf{U}_0|(\xx)\right)^2\\
\leq& C\|\vv_k\|^{\frac12}_{L^\infty(\RR^3)}\left(\sup_{\xx\in\RR^3} \langle\xx\rangle^{4\alpha-2}|\vv_k|(\xx)\right)^{\frac12}+C.
\end{align*}
Moreover,  we get by the Young inequality and $\sigma\in C^{1,0}(\mathbb{S}^2)$ that
\begin{equation}\label{eq-R-2}
\sup_{\xx\in\RR^3} \langle\xx\rangle^{4\alpha-2}| \mathbf{U}_0\otimes\vv_k+\mathbf{U}_0\otimes\vv_k+ \mathbf{U}_0\otimes\mathbf{U}_0|(\xx)\leq C+\frac14\sup_{\xx\in\RR^3} \langle\xx\rangle^{4\alpha-2}|\vv_k|(\xx).
\end{equation}
Inserting \eqref{eq-R-1} and \eqref{eq-R-2} into \eqref{eq-est-decay-2-k} gives
 \begin{equation}\label{eq-est-decay-2-k-2}
 \begin{split}
\sup_{\xx\in\RR^3}\langle\xx\rangle^{4\alpha-2}|\vv_k|(\xx)\leq & C+\frac14\sup_{\xx\in\RR^3} \langle\xx\rangle^{4\alpha-2}|\vv_k|(\xx) \\
&+\|\vv\|_{L^\infty(\RR^3\backslash B_R(0))}\sup_{|\xx|\geq R} \langle\xx\rangle^{4\alpha-2}|\vv_k|(\xx).
\end{split}
\end{equation}
With the help of the interpolation theorem, we have
\begin{equation}\label{eq-inter-1}
\begin{split}
\|\vv\|_{L^\infty(\RR^3\backslash B_R(0))}\leq&\|(1-\phi_R)\vv\|_{L^\infty(\RR^3) }\\
\leq&C \|(1-\phi_R)\vv\|^{\frac{1}{10}}_{L^2(\RR^3) }\big\|\nabla\big((1-\phi_R)\vv\big)\big\|^{\frac{9}{10}}_{L^{\frac{18}{5}}(\RR^3) }\\
\leq&C \|\vv\|^{\frac{1}{10}}_{L^2(\RR^3) }\big\|\nabla\big((1-\phi_R)\vv\big)\big\|^{\frac{9}{10}}_{L^{\frac{18}{5}}(\RR^3) }.
\end{split}
\end{equation}
The Leibniz estimate allows us to conclude
\begin{align*}
\big\|\nabla\big((1-\phi_R)\vv\big)\big\|_{L^{\frac{18}{5}}(\RR^3) }\leq &\frac{C}{R}\big\| \vv \big\|_{L^{\frac{18}{5}}(\RR^3) }+C\big\| \nabla\vv \big\|_{L^{\frac{18}{5}}(\RR^3\backslash B_{\frac{R}{2}(0)}) }.
\end{align*}
Plugging this estimate into \eqref{eq-inter-1}, we readily have
\begin{equation}\label{eq-inter-2}
\|\vv\|_{L^\infty(\RR^3\backslash B_R(0))} \leq\frac{C}{R^{\frac{9}{10}}} +C\big\| \nabla\vv \big\|^{\frac{9}{10}}_{L^\frac{18}{5} (\RR^3\backslash B_{\frac{R}{2}(0)}) }.
\end{equation}
 Since  $\vv\in H^{2\alpha}(\RR^3)$ with $\alpha\in[5/6,1]$, we get by the inclusion $$H^{2\alpha}(\RR^3)\hookrightarrow \dot{W}^{1,\frac{18}{5}}(\RR^3)$$ and \eqref{eq-inter-2} that
 \[\lim_{R\to+\infty}\|\nabla \vv\|_{L^\infty(\RR^3\backslash B_R(0))}=0.\]
 Thus, there exists a number $R_0>0$ such that for all $R>R_0$,
 \[\|\nabla\vv\|_{L^\infty(\RR^3\backslash B_R(0))}<\frac14.\]
Then \eqref{eq-est-decay-2-k-2} becomes
  \begin{equation*}\label{eq-est-decay-2-k-2}
\sup_{\xx\in\RR^3}\langle\xx\rangle^{4\alpha-2}|\vv_k|(\xx)\leq   C(\|\sigma\|_{W^{1,\infty}(\mathbb{S}^2)})+\frac12\sup_{\xx\in\RR^3} \langle\xx\rangle^{4\alpha-2}|\vv_k|(\xx),
\end{equation*}
which implies
  \begin{equation}\label{eq-est-decay-2-k-f}
\sup_{\xx\in\RR^3}\langle\xx\rangle^{4\alpha-2}|\vv_k|(\xx)\leq   C(\|\sigma\|_{W^{1,\infty}(\mathbb{S}^2)}).
\end{equation}
Since $\vv$ is the limit of $\vv_k$, we get from Fatou's lemma and \eqref{eq-est-decay-2-k-f} that
  \begin{equation}\label{eq-est-decay-2-f}
\sup_{\xx\in\RR^3}\langle\xx\rangle^{4\alpha-2}|\vv|(\xx)\leq   C(\|\sigma\|_{W^{1,\infty}(\mathbb{S}^2)}).
\end{equation}
Next, we begin to  show the decay estimate for $\nabla \vv_k $ which fulfills
\begin{equation*}
\nabla\vv_k(\xx) =\int_{0}^{1}\int_{\RR^3}\nabla O_{\alpha}(1-s,\xx-\yy){\rm div}_{\xx}\left(s^{\frac{1}{\alpha}-2}\mathbf{f}\big(\yy/s^{\frac{1}{2\alpha}}\big)\right)\,\mathrm{d}\yy{\rm d}s.
\end{equation*}
Based on this equality, we readily have by Proposition \ref{Prop-optimal-decay-low} that for each $\alpha\in[5/6,1),$
\begin{equation}\label{eq-est-decay-2-kkk}
\sup_{\xx\in\RR^3}\langle\xx\rangle^{4\alpha-1}|\nabla\vv_k|(\xx)\leq  C\|\operatorname{div}\mathbf{f}_k\|_{L^2(\RR^3)\cap L^6(\RR^3)}+C\sup_{\xx\in\RR^3}\left(|\xx|^{4\alpha-1}|\operatorname{div}\mathbf{f}_k|(\xx)\right)
\end{equation}
and for $\alpha=1$
\begin{equation}\label{eq-est-decay-2-kkkk}
\sup_{\xx\in\RR^3}\langle\xx\rangle^{3}|\nabla\vv_k|(\xx)\leq  C\|\operatorname{div}\mathbf{f}_k\|_{L^{\frac54}(\RR^3)\cap L^6(\RR^3)}+C\sup_{\xx\in\RR^3}\left(|\xx|^{3}|\operatorname{div}\mathbf{f}_k|(\xx)\right).
\end{equation}
By the triangle inequality, \eqref{eq.est-u0}, \eqref{eq.est-nablau0}, \eqref{eq.pp-est} and \eqref{eq-est-decay-2-f}, we easily find that
\begin{align*}
&\sup_{\xx\in\RR^3}\left(|\xx|^{4\alpha-1}|\operatorname{div}\mathbf{f}_k|(\xx)\right)\\
\leq&\sup_{\xx\in\RR^3} |\xx|^{4\alpha-1}| \vv\cdot\nabla\vv_k+\vv_k\cdot\nabla\mathbf{U}_0+ \mathbf{U}_0\cdot\nabla \mathbf{U}_0|(\xx)+\sup_{\xx\in\RR^3} |\xx|^{4\alpha-1}| \mathbf{U}_0\cdot\nabla\vv_k|\\
\leq &C+C \sup_{\xx\in\RR^3} \langle\xx\rangle^{2\alpha}|\nabla\vv_k|(\xx).
\end{align*}
Since
\[\sup_{\xx\in\RR^3} \langle\xx\rangle^{2\alpha}|\nabla\vv_k|(\xx)\leq \|\nabla\vv_k\|_{L^\infty(\RR^3)}\left(\sup_{\xx\in\RR^3} \langle\xx\rangle^{4\alpha-1}|\nabla\vv_k|(\xx)\right)^{\frac{2\alpha}{4\alpha-1}},\]
we have by the Young inequality that
\begin{align*}
 \sup_{\xx\in\RR^3}\left(|\xx|^{4\alpha-1}|\operatorname{div}\mathbf{f}_k|(\xx)\right)
\leq &C+C\left(\sup_{\xx\in\RR^3} \langle\xx\rangle^{2\alpha-1}|\nabla\vv_k|(\xx)\right)^{\frac{2\alpha}{4\alpha-1}}\\
\leq &C+\frac14 \sup_{\xx\in\RR^3} \langle\xx\rangle^{4\alpha-1}|\nabla\vv_k|(\xx).
\end{align*}
Inserting this estimate into \eqref{eq-est-decay-2-kkk}  and \eqref{eq-est-decay-2-kkkk}, respectively, yields that for each $\alpha\in[5/6,1),$
\begin{equation}\label{eq-est-decay-2-kkkkkk}
\sup_{\xx\in\RR^3}\langle\xx\rangle^{4\alpha-1}|\nabla\vv_k|(\xx)\leq  C\|\operatorname{div}\mathbf{f}_k\|_{L^2(\RR^3)\cap L^6(\RR^3)}+C
\end{equation}
and for $\alpha=1$,
\begin{equation}\label{eq-est-decay-2-kkkkkk-22}
\sup_{\xx\in\RR^3}\langle\xx\rangle^{3}|\nabla\vv_k|(\xx)\leq  C\|\operatorname{div}\mathbf{f}_k\|_{L^\frac{4}{5}(\RR^3)\cap L^6(\RR^3)}+C.
\end{equation}
By the H\"older inequality, \eqref{eq.est-u0}, \eqref{eq.est-nablau0}  and \eqref{eq.pp-est}, one has that for each $\alpha\in[5/6,1],$
\begin{align*}
\|\operatorname{div}\mathbf{f}_k\|_{L^6(\RR^3)}\leq& \|\vv\|_{L^\infty(\RR^3)}\|\nabla \vv_k\|_{L^6(\RR^3)}+\|\vv_k\|_{L^\infty(\RR^3)}\|\nabla \uu_0\|_{L^6(\RR^3)}\\
&+\|\uu_0\|_{L^\infty(\RR^3)}\|\nabla \vv_k\|_{L^6(\RR^3)}+\|\uu_0\|_{L^\infty(\RR^3)}\|\nabla \uu_0\|_{L^6(\RR^3)}\\
\leq &  C(\|\sigma\|_{W^{1,\infty}(\mathbb{S}^2)})
\end{align*}
and
\begin{align*}
\|\operatorname{div}\mathbf{f}_k\|_{L^2(\RR^3)}\leq& \|\vv\|_{L^\infty(\RR^3)}\|\nabla \vv_k\|_{L^2(\RR^3)}+\|\vv_k\|_{L^\infty(\RR^3)}\|\nabla \uu_0\|_{L^2(\RR^3)}\\
&+\|\uu_0\|_{L^\infty(\RR^3)}\|\nabla \vv_k\|_{L^2(\RR^3)}+\|\uu_0\|_{L^\infty(\RR^3)}\|\nabla \uu_0\|_{L^2(\RR^3)}\\
\leq &  C(\|\sigma\|_{W^{1,\infty}(\mathbb{S}^2)}).
\end{align*}
Plugging both estimates into \eqref{eq-est-decay-2-kkkkkk} gives that for each $\alpha\in[5/6,1),$
  \begin{equation*}
\sup_{\xx\in\RR^3}\langle\xx\rangle^{4\alpha-1}|\nabla\vv_k|(\xx)\leq   C(\|\sigma\|_{W^{1,\infty}(\mathbb{S}^2)}).
\end{equation*}
For $\alpha=1,$ we can show by  the H\"older inequality \eqref{eq.est-u0}, \eqref{eq.est-nablau0} and \eqref{eq.pp-est} again that
\begin{align*}
\|\operatorname{div}\mathbf{f}_k\|_{L^{\frac54}(\RR^3)}\leq& \|\vv\|_{L^\frac52(\RR^3)}\|\nabla \vv_k\|_{L^\frac52(\RR^3)}+\|\vv_k\|_{L^\frac52(\RR^3)}\|\nabla \uu_0\|_{L^\frac52(\RR^3)}\\
&+\|\uu_0\|_{L^\frac{10}{3}(\RR^3)}\|\nabla \vv_k\|_{L^2(\RR^3)}+\|\uu_0\|_{L^5(\RR^3)}\|\nabla \uu_0\|_{L^{\frac53}(\RR^3)}\\
\leq &  C(\|\sigma\|_{W^{1,\infty}(\mathbb{S}^2)}).
\end{align*}
Inserting this estimate into \eqref{eq-est-decay-2-kkkkkk-22} implies that for $\alpha=1,$
  \begin{equation*}
\sup_{\xx\in\RR^3}\langle\xx\rangle^{3}|\nabla\vv_k|(\xx)\leq   C(\|\sigma\|_{W^{1,\infty}(\mathbb{S}^2)}).
\end{equation*}
Finally, we get   by taking $k\to\infty$ that  for each $\alpha\in[5/6,1],$
  \begin{equation}\label{eq-est-nabla-decay-2-k-f}
\sup_{\xx\in\RR^3}\langle\xx\rangle^{4\alpha-1}|\nabla\vv|(\xx)\leq   C(\|\sigma\|_{W^{1,\infty}(\mathbb{S}^2)}).
\end{equation}
\subsection[The decay estimate   with  logarithmic loss]{The decay estimate   with  logarithmic loss} In this subsection, we are going to show the almost optimal estimate   with  logarithmic loss which is caused by   the Leray projector $\mathbb{P}$.
\begin{proposition}\label{Prop-optimal-decay}
Let $\alpha\in[5/6,1]$ and $\mathbf{f}(\xx)\in \dot{H}^1(\RR^3)\cap W^{1,\infty}(\RR^3)$ satisfies
\begin{equation}\label{eq-con-optimal}
\sup_{\xx\in\RR^3}\left|| \xx|^{4\alpha-2}\mathbf{f}(\xx)\right|+\sup_{\xx\in\RR^3}\left||\xx|^{4\alpha-1}\nabla \mathbf{f}(\xx)\right|<\infty,
\end{equation}
and
\begin{equation}\label{eq.fundamental-solution-2}
\ww(\xx)=\int_{0}^{1}e^{-(-\Delta)^{\alpha}(1-s)}\mathbb{P}{\rm div}_{\xx}\left(s^{\frac{1}{\alpha}-2}\mathbf{f}\big(\cdot/s^{\frac{1}{2\alpha}}\big)\right)\,{\rm d}s.
\end{equation}
Then $\ww$ enjoys the decay estimate
\begin{equation}\label{eq-0830-1}
|\ww|(\xx)\leq C\langle \xx\rangle^{-(4\alpha-1)}\log\langle \xx\rangle\ \  \mbox{for each}\ \,\,  \xx\in\RR^{3}.
\end{equation}
\end{proposition}
\begin{proof}
Thanks to \eqref{eq-est-infty}, it suffices to show the decay estimate of $\ww$ for large $|\xx|.$ So we always assume that $|\xx|\geq200$ in the remaining part of the proof. In terms  of the low-high frequency decomposition, one splits $\ww$ into three parts as follows
\begin{align*}
 \ww=&\left(\sum_{q\geq0}+\sum_{N_{\xx}\leq q<0}+\sum_{q<N_{\xx}}\right)\int_{0}^{1}\dot{\Delta}_q\mathbb{P} e^{-(-\Delta)^{\alpha}(1-s)}s^{\frac{1}{2 \alpha}-2}\left(\operatorname{div}\mathbf{f}\right)(\cdot/s^{\frac{1}{2\alpha}})\,{\rm d}s\\
=:&{\rm I}+{\rm II}+{\rm III},
\end{align*}
where $N_{\xx}=-[\log_2|\xx|]+1.$

The first term $I $  which can be rewritten as
\begin{align*}
{\rm I}=&\sum_{q\geq0}\int_{0}^{1}\int_{B_{\frac{|\xx|}{2}}(0)}O_{\alpha,q}(1-s,\xx-\yy)s^{\frac{1}{ 2\alpha}-2}\left(\operatorname{div}\mathbf{f}\right)(\yy/s^{\frac{1}{2\alpha}})\,{\rm d}\yy{\rm d}s\\
&+\sum_{q\geq0}\int_{0}^{1}\int_{\mathbf{R}^3\backslash B_{\frac{| \xx|}{2}}(0)}O_{\alpha,q}(1-s,\xx-\yy)s^{\frac{1}{ 2\alpha}-2}\left(\operatorname{div}\mathbf{f}\right)(\yy/s^{\frac{1}{2\alpha}})\,{\rm d}\yy{\rm d}s\\
=:&{\rm I}_1+{\rm I}_2.
\end{align*}
The term $I_2$ can be bounded as follows
\begin{align*}
{\rm I}_2=&\sum_{q\geq0}\int_{0}^{1}\int_{\mathbf{R}^3\backslash B_{\frac{| \xx|}{2}}(0)}O_{\alpha,q}(1-s,\xx-\yy)\frac{1}{|\yy|^{4\alpha-1}} \left(\frac{|\yy| }{s^{\frac{1}{ 2\alpha}}}\right)^{4\alpha-1}\left(\operatorname{div}\mathbf{f}\right)(\yy/s^{\frac{1}{2\alpha}})\,{\rm d}\yy{\rm d}s\\
\leq&\frac{4}{|\xx|^{4\alpha-1}}\sum_{q\geq0}\int_{0}^{1}\int_{\mathbf{R}^3 }\big|O_{\alpha,q}\big|(1-s,\xx-\yy) \left(\frac{|\yy| }{s^{\frac{1}{ 2\alpha}}}\right)^{4\alpha-1}\left|\operatorname{div}\mathbf{f}\right|(\yy/s^{\frac{1}{2\alpha}})\,{\rm d}\yy{\rm d}s.
\end{align*}
Lemma \ref{lem-exe-decay} enables us to conclude that
\begin{equation}\label{eq-d-I2}
\begin{split}
|{\rm I}_2|\leq &\frac{4}{|\xx|^{4\alpha-1}}\sup_{\xx\in\RR^3}\left||\xx|^{4\alpha-1}\operatorname{div}\mathbf{f}(\xx)\right|\sum_{q\geq0}\int_{0}^{1}e^{-c(1-s)2^{2q\alpha}}\,\mathrm{d}s\\
\leq&\frac{C}{|\xx|^{4\alpha-1}}\sup_{\xx\in\RR^3}\left||\xx|^{4\alpha-1}\operatorname{div}\mathbf{f}(\xx)\right|\sum_{q\geq0}2^{-2q\alpha}
\leq C\sup_{\xx\in\RR^3}\left||\xx|^{4\alpha-1}\operatorname{div}\mathbf{f}(\xx)\right|\frac{1}{|\xx|^{4\alpha-1}}.
\end{split}
\end{equation}
As for ${\rm I}_1$, we  see that
\begin{align*}
{\rm I}_1=&\frac{1}{|\xx|^{4\alpha-1}}\sum_{q\geq0}\int_{0}^{1}\int_{B_{\frac{\langle \xx\rangle}{2}}(0)}|\xx|^{4\alpha-1}O_{\alpha,q}(1-s,\xx-\yy)s^{\frac{1}{ 2\alpha}-2}\left(\operatorname{div}\mathbf{f}\right)(\yy/s^{\frac{1}{2\alpha}})\,{\rm d}\yy{\rm d}s\\
\leq&\frac{C}{|\xx|^{4\alpha-1}}\sum_{q\geq0}\int_{0}^{1}\int_{\RR^3}|\xx-\yy|^{4\alpha-1}\big|O_{\alpha,q}\big|(1-s,\xx-\yy)s^{\frac{1}{ 2\alpha}-2}\left|\operatorname{div}\mathbf{f}\right|(\yy/s^{\frac{1}{2\alpha}})\,{\rm d}\yy{\rm d}s.
\end{align*}
Moreover, we have by Lemma \ref{lem-exe-decay} that
\begin{equation}\label{eq-d-I1}
\begin{split}
|{\rm I}_1|\leq &\frac{C}{|\xx|^{4\alpha-1}}\sum_{q\geq0}2^{-q(4\alpha-1-\frac32)}\int_{0}^{1}e^{-c(1-s)2^{2q\alpha}}s^{\frac{1}{ 2\alpha}-2}\left\|\operatorname{div}\mathbf{f}(\cdot/s^{\frac{1}{2\alpha}})\right\|_{L^2(\RR^3)}\,\mathrm{d}s\\
\leq &C\|\operatorname{div}\mathbf{f}\|_{ L^2(\RR^3) }\frac{1}{|\xx|^{4\alpha-1}}\int_{0}^{1} s^{\frac{1}{ 2\alpha}-2+\frac{3}{4\alpha}}  \,\mathrm{d}s\leq C\| \mathbf{f}\|_{ \dot{H}^1(\RR^3) }\frac{1}{|\xx|^{4\alpha-1}}.
\end{split}
\end{equation}
For the term ${\rm II}$, we need to use the smooth cut-off function $\phi_{\frac{|\xx|}{2}}(\yy)$ to decompose  the force $\mathbf{f}(\yy/s^{\frac{1}{2\alpha}})$ into two parts as follows
\begin{align*}
\mathbf{f}(\yy/s^{\frac{1}{2\alpha}})=\phi_{\frac{|\xx|}{2}}(\yy)\mathbf{f}(\yy/s^{\frac{1}{2\alpha}})+\left(1-\phi_{\frac{|\xx|}{2}}(\yy)\right)\mathbf{f}(\yy/s^{\frac{1}{2\alpha}})=:  \mathbf{f}^\natural(\yy/s^{\frac{1}{2\alpha}})+  \mathbf{f}^\sharp (\yy/s^{\frac{1}{2\alpha}}).
\end{align*}
Thus,  we can  rewrite  in terms of the support of $\phi_{\frac{|\xx|}{2}}(\yy)$ the second term ${\rm II}$  as
\begin{equation}\label{eq-decom-II}
\begin{split}
{\rm II}=&\sum_{N_{\xx}\leq q<0}\int_{0}^{1}\int_{B_{\frac{|\xx|}{2}}(0)}\nabla O_{\alpha,q}(1-s,\xx-\yy)s^{\frac{1}{ \alpha}-2}\cdot\mathbf{f}^\natural (\yy/s^{\frac{1}{2\alpha}})\,{\rm d}\yy{\rm d}s\\
&+\sum_{N_{\xx}\leq q<0}\int_{0}^{1}\int_{\mathbf{R}^3\backslash B_{\frac{|\xx|}{4}}(0)}O_{\alpha,q}(1-s,\xx-\yy)s^{\frac{1}{ 2\alpha}-2}\left(\operatorname{div}\mathbf{f}^\sharp\right)(\yy/s^{\frac{1}{2\alpha}})\,{\rm d}\yy{\rm d}s\\
=:&{\rm II}_1+{\rm II}_2.
\end{split}
\end{equation}
Following the same method as used in the proof of \eqref{eq-d-I1}, we can show that $|{\rm II}_1|$ can be bounded by
\begin{align*}
&\frac{C}{|\xx|^{4\alpha+2}}\sum_{N_{\xx}\leq q<0}\int_{0}^{1}\int_{B_{\frac{|\xx|}{2}}(0)}|\xx-\yy|^{4\alpha+2}\big|\nabla O_{\alpha,q}\big|(1-s,\xx-\yy)s^{\frac{1}{ \alpha}-2}\left|\mathbf{f}\right|(\yy/s^{\frac{1}{2\alpha}})\,{\rm d}\yy{\rm d}s\\
\leq& C\sup_{\xx\in\RR^3}\left||\xx|^{4\alpha-2} \mathbf{f}(\xx)\right|\frac{1}{|\xx|^{4\alpha+2}}\sum_{N_{\xx}\leq q<0}2^{-q(4\alpha+2-4)}\int_{B_{\frac{|\xx|}{2}}(0)} \frac{1}{|\yy|^{4\alpha-2}}\,\mathrm{d}\yy \\
\leq& C\sup_{\xx\in\RR^3}\left||\xx|^{4\alpha-2} \mathbf{f}(\xx)\right|\frac{1}{|\xx|^{4\alpha+2}}|\xx|^{4\alpha+2-4}|\xx|^{3-(4\alpha-2)}
\leq  C\sup_{\xx\in\RR^3}\left||\xx|^{4\alpha-2} \mathbf{f}(\xx)\right|\frac{1}{|\xx|^{4\alpha-1}},
\end{align*}
which implies
\begin{equation}\label{eq-est-ii1}
|{\rm II}_1|\leq  C\sup_{\xx\in\RR^3}\left||\xx|^{4\alpha-2} \mathbf{f}(\xx)\right|\frac{1}{|\xx|^{4\alpha-1}}.
\end{equation}
In the same fashion in \eqref{eq-d-I2}, it is easy to show by the fact $|\nabla\phi_{\frac{\xx}{2}} |\leq \frac{C}{|\xx|}$ that
\begin{equation}\label{eq-est-ii2}
\begin{split}
|{\rm II}_2|\leq &\frac{C}{|\xx|^{C\alpha-1}}\sup_{\xx\in\RR^3}\left(\left||\xx|^{4\alpha-1}\operatorname{div}\mathbf{f}(\xx)\right|+\left||\xx|^{4\alpha-2}\mathbf{f}(\xx)\right|\right)\sum_{N_{\xx}\leq q<0}\int_{0}^{1}e^{-c(1-s)2^{2q\alpha}}\,\mathrm{d}s\\
\leq&\frac{C}{|\xx|^{4\alpha-1}}\sup_{\xx\in\RR^3}\left(\left||\xx|^{4\alpha-1}\operatorname{div}\mathbf{f}(\xx)\right|+\left||\xx|^{4\alpha-2}\mathbf{f}(\xx)\right|\right)\sum_{N_{\xx}\leq q<0}1\\
\leq &C\sup_{\xx\in\RR^3}\left(\left||\xx|^{4\alpha-1}\operatorname{div}\mathbf{f}(\xx)\right|+\left||\xx|^{4\alpha-2}\mathbf{f}(\xx)\right|\right)\frac{1}{|\xx|^{4\alpha-1}}\log\langle\xx\rangle.
\end{split}
\end{equation}
For the last term ${\rm III}$, we decompose it in terms of \eqref{eq-decom-II} that
\begin{align*}
{\rm III}=&\sum_{q<N_{\xx} }\int_{0}^{1}\int_{B_{\frac{| \xx|}{2}}(0)}\nabla O_{\alpha,q}(1-s,\xx-\yy)s^{\frac{1}{  \alpha}-2}\cdot\mathbf{f}^\natural (\yy/s^{\frac{1}{ 2\alpha}})\,{\rm d}\yy{\rm d}s\\
&+\sum_{  q<N_{\xx}}\int_{0}^{1}\int_{\mathbf{R}^3\backslash B_{\frac{| \xx|}{4}}(0)}\nabla O_{\alpha,q}(1-s,\xx-\yy)s^{\frac{1}{  \alpha}-2} \cdot\mathbf{f}^\sharp (\yy/s^{\frac{1}{2 \alpha}})\,{\rm d}\yy{\rm d}s\\
=:&{\rm III}_1+{\rm III}_2.
\end{align*}
We perform the process as used in deriving \eqref{eq-d-I1} to get that $|{\rm III}_1|$ can be bounded by
\begin{align*}
&\frac{C}{|\xx|^{4\alpha-1}}\sum_{ q<  N_{\xx}}\int_{0}^{1}\int_{B_{\frac{|\xx|}{2}}(0)}|\xx-\yy|^{4\alpha-1}\big|\nabla O_{\alpha,q}\big|(1-s,\xx-\yy)s^{\frac{1}{ \alpha}-2}\left|\mathbf{f}\right|(\yy/s^{\frac{1}{2\alpha}})\,{\rm d}\yy{\rm d}s\\
\leq& C\sup_{\xx\in\RR^3}\left||\xx|^{4\alpha-2} \mathbf{f}(\xx)\right|\frac{1}{|\xx|^{4\alpha-1}}\sum_{ q<N_{\xx} }2^{-q(4\alpha-1-4)}\int_{B_{\frac{|\xx|}{2}}(0)} \frac{1}{|\yy|^{4\alpha-2}}\,\mathrm{d}\yy \\
\leq& C\sup_{\xx\in\RR^3}\left||\xx|^{4\alpha-2} \mathbf{f}(\xx)\right|\frac{1}{|\xx|^{4\alpha-1}}|\xx|^{4\alpha-1-4}|\xx|^{3-(4\alpha-2)}
\leq  C\sup_{\xx\in\RR^3}\left||\xx|^{4\alpha-2} \mathbf{f}(\xx)\right|\frac{1}{|\xx|^{4\alpha-1}}.
\end{align*}
This estimate guarantees that
\begin{equation}\label{eq-est-iii1}
|{\rm III}_1|\leq  C\sup_{\xx\in\RR^3}\left||\xx|^{4\alpha-2} \mathbf{f}(\xx)\right|\frac{1}{|\xx|^{4\alpha-1}}.
\end{equation}
Now it remains  to deal with $ {\rm III}_2$. Since  $|\yy|\geq\frac14|\xx|$, we have
\begin{align*}
{\rm III}_2
\leq&\frac{4}{|\xx|^{4\alpha-2}}\sum_{q<  N_{\xx}}\int_{0}^{1}\int_{\mathbf{R}^3 }\big|\nabla O_{\alpha,q}\big|(1-s,\xx-\yy) \left(\frac{|\yy| }{s^{\frac{1}{ 2\alpha}}}\right)^{4\alpha-2}\left|\mathbf{f}\right|(\yy/s^{\frac{1}{2\alpha}})\,{\rm d}\yy{\rm d}s\\
\leq&  C\sup_{\xx\in\RR^3}\left||\xx|^{4\alpha-2} \mathbf{f}(\xx)\right|\frac{1}{|\xx|^{4\alpha-2}}\sum_{ q<N_{\xx} }2^{q}\leq C\sup_{\xx\in\RR^3}\left||\xx|^{4\alpha-2} \mathbf{f}(\xx)\right|\frac{1}{|\xx|^{4\alpha-1}}.
\end{align*}
Combining this estimate with \eqref{eq-d-I2}, \eqref{eq-d-I1}, \eqref{eq-est-ii1},   \eqref{eq-est-ii2} and \eqref{eq-est-iii1} yields
\begin{equation*}
 \sup_{|\xx|\geq2}|\ww(\xx)|\leq C\frac{1}{|\xx|^{4\alpha-1}}\log\langle\xx\rangle.
\end{equation*}
This inequality together with \eqref{eq-est-infty} implies the estimate \eqref{eq-0830-1}.
\end{proof}
By Proposition \ref{Prop-optimal-decay}, we immediately get the decay estimate in Theorem \ref{thm-leray}. Now we turn to show the decay estimate in Theorem \ref{thm-leray-per}.
With both estimates \eqref{eq-est-decay-2-kkkkkk} and \eqref{eq-est-nabla-decay-2-k-f} in hand,  we write  $\vv$ in the integral form in terms of Proposition~4.1 established in \cite{LMZ21} that
\begin{equation}\label{eq.fundamental-solution-l}
\vv(\xx)= \int_{0}^{1}e^{-(-\Delta)^{\alpha}(1-s)}\mathbb{P}{\rm div}_{\xx}\left(s^{\frac{1}{\alpha}-2}\mathbf{f}\big(\cdot/s^{\frac{1}{2\alpha}}\big)\right)\,{\rm d}s,
\end{equation}
where \[\mathbf{f}(\xx)=-\Big(\vv\otimes \vv+\mathbf{U}_{0}\otimes \vv+\vv\otimes \mathbf{U}_{0}+\mathbf{U}_{0}\otimes \mathbf{U}_{0}\Big)(\xx).\]
Since $\alpha\in[5/6,1]$ and $\sigma\in C^{1,0}(\mathbb{S}^2)$, it is easy to verify  by using both estimates \eqref{eq-est-decay-2-kkkkkk} and \eqref{eq-est-nabla-decay-2-k-f} again that
\begin{align*}
\sup_{\xx\in\RR^3}\langle\xx \rangle^{4\alpha-1}\operatorname{div}\Big(\vv\otimes \vv+\mathbf{U}_{0}\otimes \vv+\vv\otimes \mathbf{U}_{0}+\mathbf{U}_{0}\otimes \mathbf{U}_{0}\Big)(\xx)<\infty,
\end{align*}
which meets the condition \eqref {eq-con-optimal} in Proposition \ref{Prop-optimal-decay}. Thus we have
\begin{equation}\label{eq-decay-lossssss}
 \left|\vv(\xx)\right|\leq C( \|\sigma\|_{W^{1,\infty}(\mathbb{S}^2)})(1+|\xx|)^{4\alpha-1}\log^{-1}(1+|\xx|),
\end{equation}
which is the decay estimate in Theorem \ref{thm-leray-per}.

%%%%%%%%%%%%%%%%%%%%%%%%%%%%%%%%%%%%%%%%%%%%%%%%%%%%%%%%%%%%%%%%%%%%%%%%%%%%%%%%%%%%%%%%%%%%%%%%%%%%%%%%%%%%%%%%%%%%%%%%%%%%%%%%
\subsection[The optimal decay estimate]{The optimal decay estimate}
%%%%%%%%%%%%%%%%%%%%%%%%%%%%%%%%%%%%%%%%%%%%%%%%%%%%%%%%%%%%%%%%%%%%%%%%%%%%%%%%%%%%%%%%%%%%%%%%%%%%%%%%%%%%%%%%%%%%%%%%%%%%%%%%%%%%%

In this section, we are going to show the optimal decay estimate for $\vv$ with the special force $\mathbf{f}$. Since  the decay rate for each term of $\mathbf{f}$  is different, we decompose it into two parts as follows
\begin{equation*}
\mathbf{f}(\xx)= -\Big(\vv\otimes \vv+\mathbf{U}_{0}\otimes \vv+\vv\otimes \mathbf{U}_{0}\Big)(\xx)-\left(\mathbf{U}_{0}\otimes \mathbf{U}_{0}\right)(\xx),
 \end{equation*}
 where  $\mathbf{f}_1=-\Big(\vv\otimes \vv+\mathbf{U}_{0}\otimes \vv+\vv\otimes \mathbf{U}_{0}\Big)(\xx).$
\begin{proposition}\label{prop-decay-begin-third}
Let $\alpha\in[\frac56,1]$, $\sigma\in C^{1,0}(\mathbb{S}^2)$, and  $\vv$ be a weak solution to the system~\eqref{E}.
Assume $\mathbf{f}_2(\xx)\in \dot{H}^1(\RR^3)\cap W^{1,\infty}(\RR^3)$ satisfies
\begin{equation}\label{eq-con-optimal-last}
\sup_{\xx\in\RR^3}\left|| \xx|^{4\alpha-2}\mathbf{f}_2(\xx)\right|+\sup_{\xx\in\RR^3}\left||\xx|^{4\alpha-1}\nabla \mathbf{f}_2(\xx)\right|<\infty,
\end{equation}
and
\begin{equation}\label{eq.fundamental-solution-last}
\ww(\xx)=\int_{0}^{1}e^{-(-\Delta)^{\alpha}(1-s)}\mathbb{P}{\rm div}_{\xx}\left(s^{\frac{1}{\alpha}-2}(\mathbf{f}_1+ \mathbf{f}_2)\big(\cdot/s^{\frac{1}{2\alpha}}\big)\right)\,{\rm d}s.
\end{equation}
If, moreover, there exists  some  $\gamma\in(0,1)$ such that
\begin{equation}\label{eq.con-cr}
 \left\||\cdot|^{4\alpha-1+\gamma}\operatorname{div}\mathbf{f}_2\right\|_{\dot{C}^\gamma (\RR^3)}<\infty,
 \end{equation}
then we have
\begin{equation*}
 \left|\ww(\xx)\right|\leq C\left(\|\sigma\|_{W^{1,\infty}(\RR^3)},\left\||\cdot|^{4\alpha-1+\gamma}\operatorname{div}\mathbf{f}_2\right\|_{\dot{C}^\gamma (\RR^3)}\right)(1+|\xx|)^{4\alpha-1}.
\end{equation*}
\end{proposition}
\begin{proof}
From \eqref{eq-decay-lossssss}, we have
\begin{equation}\label{eq-20220819}
 \left|\ww(\xx)\right|\leq C( \|\sigma\|_{W^{1,\infty}(\mathbb{S}^2)})(1+|\xx|)^{4\alpha-1}\log^{-1}(1+|\xx|).
\end{equation}
This decay estimate allows us to write
\begin{align*}
\ww(\xx)=&\int_{0}^{1}e^{-(-\Delta)^{\alpha}(1-s)}\mathbb{P}{\rm div}_{\xx}\left(s^{\frac{1}{\alpha}-2}\left(\mathbf{f}_1+\mathbf{f}_2\right)\big(\cdot/s^{\frac{1}{2\alpha}}\big)\right)\,\mathrm{d}s\\
=:&\ww^\natural+\ww^\sharp.
\end{align*}
For $|\xx|\geq200,$ as in the proof of Proposition \ref{Prop-optimal-decay}, we decompose $ \ww^\natural$ as follows
\begin{align*}
 \ww^\natural=&\left(\sum_{q\geq0}+\sum_{N_{\xx}\leq q<0}+\sum_{q<N_{\xx}}\right)\int_{0}^{1}\dot{\Delta}_q\mathbb{P} e^{-(-\Delta)^{\alpha}(1-s)}s^{\frac{1}{2 \alpha}-2}\left(\operatorname{div}\mathbf{f}\right)(\cdot/s^{\frac{1}{2\alpha}})\,{\rm d}s\\
=:&{\rm I}+{\rm II}+{\rm III},
\end{align*}
where $N_{\xx}=-[\log_2|\xx|]+1.$

From the proof of Proposition \ref{Prop-optimal-decay},  we know that
 \begin{equation}\label{eq-f-1}
  \left(|{\rm I}|+|{\rm II}_1|+|{\rm III}|\right)(\xx) |\leq C (1+|\xx|)^{4\alpha-1},
 \end{equation}
 where ${\rm II}_1$ comes from the following decomposition
 \begin{align*}
{\rm II}=&\sum_{N_{\xx}\leq q<0}\int_{0}^{1}\int_{B_{\frac{|\xx|}{2}}(0)}\nabla O_{\alpha,q}(1-s,\xx-\yy)s^{\frac{1}{ \alpha}-2}\cdot\left(\phi_{\frac{ |\xx|}{2}}\mathbf{f}  \right)(\yy/s^{\frac{1}{2\alpha}})\,{\rm d}\yy{\rm d}s\\
&+\sum_{i=1}^2\sum_{N_{\xx}\leq q<0}\int_{0}^{1}\int_{\mathbf{R}^3\backslash B_{\frac{|\xx|}{4}}(0)}O_{\alpha,q}(1-s,\xx-\yy)s^{\frac{1}{ 2\alpha}-2}\operatorname{div}\left(\phi^c_{\frac{ |\xx|}{2}}\mathbf{f}_i \right)(\yy/s^{\frac{1}{2\alpha}})\,{\rm d}\yy{\rm d}s\\
=:&{\rm II}_1+{\rm II}_{2,1}+{\rm II}_{2,2}
\end{align*}
with $\phi^c_{\frac{ |\xx|}{2}}=1-\phi_{\frac{ |\xx|}{2}}.$

So our task is now to bound ${\rm II}_{2,1}$ and ${\rm II}_{2,2}$. To begin with, we estimate ${\rm II}_{2,1}$. In terms of estimates \eqref{eq.est-u0}, \eqref{eq.est-nablau0} and  \eqref{eq-20220819}, it is easy to check that there exist a positive constant $C$ such that
\begin{equation}\label{eq-decay-f1}
\sup_{\xx\in\RR^3}|\xx|^{4\alpha-1}| \mathbf{f}_1|(\xx)+\sup_{\xx\in\RR^3}|\xx|^{4\alpha}|\operatorname{div}\mathbf{f}_1|(\xx)\leq C.
\end{equation}
With the decay estimate \eqref{eq-decay-f1} in hand, by using the triangle inequality, we readily have  that for all $|\xx|\geq 200,$
\begin{equation*}
|{\rm II}_{2,1}|\leq \frac{C}{|\xx|^{4\alpha}}\sum_{N_{\xx}\leq q<0}\int_{0}^{1}\int_{\mathbf{R}^3 }|O_{\alpha,q}|(1-s,\xx-\yy)s^{\frac{1}{ 2\alpha}-2}\left||\yy|^{4\alpha}\operatorname{div}\left(\phi^c_{\frac{ |\xx|}{2}}\mathbf{f}_1\right)\right|(\yy/s^{\frac{1}{2\alpha}})\,{\rm d}\yy{\rm d}s.
\end{equation*}
By the Young inequality  and \eqref{eq-decay-f1}, we obtain
\begin{equation}\label{eq-f-2}
\begin{split}
|{\rm II}_{2,1}|\leq &\frac{C}{|\xx|^{4\alpha}}\bigg(\sup_{\xx\in\RR^3}|\xx|^{4\alpha-1}| \mathbf{f}_1|(\xx)+\sup_{\xx\in\RR^3}|\xx|^{4\alpha}|\operatorname{div}\mathbf{f}_1|(\xx)\bigg)\sum_{N_{\xx}\leq q<0}\int_0^1s^{\frac{1}{2\alpha}}\,\mathrm{d}s\\
\leq&\frac{C}{|\xx|^{4\alpha}}\log\langle \xx\rangle.
\end{split}
\end{equation}
Collecting \eqref{eq-f-1} and \eqref{eq-f-2} yields that for all $|\xx|\geq 200,$
 \begin{equation}\label{eq-f-3}
  \left|\vv^\natural\right|(\xx)  \leq C (1+|\xx|)^{4\alpha-1}.
 \end{equation}
  It remains for us to bound ${\rm II}_{2,2}$ for large $|\xx|$. Firstly, we note that
\begin{equation}\label{sharp-1}
\begin{split}
{\rm II}_{2,2}=&\sum_{N_{\xx}\leq q<0}\int_{0}^{1}\int_{\mathbf{R}^3}O_{\alpha,q}(1-s,\xx-\yy)s^{\frac{1}{ 2\alpha}-2}\widetilde{\dot{\Delta}}_q\left(\operatorname{div}\mathbf{f}_2^\sharp\right)(\yy/s^{\frac{1}{2\alpha}})\,{\rm d}\yy{\rm d}s\\
\leq&C  \sum_{N_{\xx}\leq q<0}2^{-q\gamma}\int_{0}^{1}s^{\frac{1}{2\alpha}-2}\|O_{\alpha}\|_{L^1(\RR^3)} \left\|\left(\operatorname{div}\mathbf{f}_2^\sharp\right)(\cdot/s^{\frac{1}{2\alpha}})\right\|_{\dot{B}^\gamma_{\infty,\infty}(\RR^3)}{\rm d}s,
\end{split}
\end{equation}
where
\[\mathbf{f}_2^\sharp (\yy/s^{\frac{1}{2\alpha}})=\left(1-\phi_{\frac{|\xx|}{2}}(\yy)\right)\mathbf{f}_2(\yy/s^{\frac{1}{2\alpha}}). \]
Since
\begin{equation*}
\left(\operatorname{div}\mathbf{f}^\sharp\right)(\yy/s^{\frac{1}{2\alpha}})=-\nabla\left(\phi_{\frac{|\xx|}{2}}(\yy)\right)\cdot\mathbf{f}_2(\yy/s^{\frac{1}{2\alpha}})
+\left(1-\phi_{\frac{|\xx|}{2}}(\yy)\right)\left(\operatorname{div}\mathbf{f}_2\right)(\yy/s^{\frac{1}{2\alpha}}),
\end{equation*}
we further obtain by the fact the H\"older space $\dot{C}^\gamma(\RR^3)=\dot{B}^\gamma_{\infty,\infty}(\RR^3)$ for all $\gamma\in(0,1)$ that
\begin{align*}
&\left\|\left(\operatorname{div}\mathbf{f}_2^\sharp\right)(\cdot/s^{\frac{1}{2\alpha}})\right\|_{\dot{B}^\gamma_{\infty,\infty}(\RR^3)}\\
\leq&\left\|\nabla\left(\phi_{\frac{|\xx|}{2}}(\cdot)\right)\cdot\mathbf{f}_2(\cdot/s^{\frac{1}{2\alpha}})\right\|_{\dot{C}^\gamma(\RR^3)}
+\left\|\left(1-\phi_{\frac{|\xx|}{2}}(\cdot)\right)\left(\operatorname{div}\mathbf{f}_2\right)(\cdot/s^{\frac{1}{2\alpha}})\right\|_{\dot{C}^\gamma(\RR^3)}=:{\rm L_1}+{\rm L_2}.
\end{align*}
According to the definition of H\"older space and the support of $\phi$, we have
\begin{align*}
{\rm L_1}=&\sup_{\substack{\yy,\,\zz\in\RR^3\\ \yy\neq\zz }}\frac{\left|\nabla\left(\phi_{\frac{|\xx|}{2}}(\yy)\right)\cdot\mathbf{f}_2(\yy/s^{\frac{1}{2\alpha}})
-\nabla\left(\phi_{\frac{|\xx|}{2}}(\zz)\right)\cdot\mathbf{f}_2(\zz/s^{\frac{1}{2\alpha}})\right|}{|\yy-\zz|^\gamma}\\
=&\sup_{\substack{\yy,\,\zz\in\RR^3\backslash B_{\frac{|\xx|}{2}}(0)\\ \yy\neq\zz }}\frac{\left|\nabla\left(\phi_{\frac{|\xx|}{2}}(\yy)\right)\cdot\mathbf{f}_2(\yy/s^{\frac{1}{2\alpha}})
-\nabla\left(\phi_{\frac{|\xx|}{2}}(\zz)\right)\cdot\mathbf{f}_2(\zz/s^{\frac{1}{2\alpha}})\right|}{|\yy-\zz|^\gamma}.
\end{align*}
Moreover, one has by the triangle inequality that
\begin{align*}
{\rm L_1}\leq&\sup_{\substack{\yy,\,\zz\in\RR^3\backslash B_{\frac{|\xx|}{2}}(0)\\ \yy\neq\zz }}\frac{\left|\nabla\left(\phi_{\frac{|\xx|}{2}}(\yy)\right)\cdot\mathbf{f}_2(\yy/s^{\frac{1}{2\alpha}})
-\nabla\left(\phi_{\frac{|\xx|}{2}}(\zz)\right)\cdot\mathbf{f}_2(\yy/s^{\frac{1}{2\alpha}})\right|}{|\yy-\zz|^\gamma}\\
&+\sup_{\substack{\yy,\,\zz\in\RR^3\backslash B_{\frac{|\xx|}{2}}(0)\\ \yy\neq\zz }}\frac{\left|\nabla\left(\phi_{\frac{|\xx|}{2}}(\zz)\right)\cdot\mathbf{f}_2(\yy/s^{\frac{1}{2\alpha}})
-\nabla\left(\phi_{\frac{|\xx|}{2}}(\zz)\right)\cdot\mathbf{f}_2(\zz/s^{\frac{1}{2\alpha}})\right|}{|\yy-\zz|^\gamma}.
\end{align*}
Thus, we have
\begin{align*}
{\rm L_1}\leq&\left\|\nabla\left(\phi_{\frac{|\xx|}{2}}(\cdot)\right)\right\|_{\dot{C}^\gamma(\RR^3)}\sup_{\yy\in\RR^3\backslash B_{\frac{|\xx|}{2}}(0)}|\mathbf{f}_2|(\yy/s^{\frac{1}{2\alpha}})\\
&+\left\|\nabla\left(\phi_{\frac{|\xx|}{2}}(\cdot)\right)\right\|_{L^\infty(\RR^3)}
\sup_{\substack{\yy,\,\zz\in\RR^3\backslash B_{\frac{|\xx|}{2}}(0)\\ \yy\neq\zz }}\frac{\left|\mathbf{f}_2(\yy/s^{\frac{1}{2\alpha}})
-\mathbf{f}_2(\zz/s^{\frac{1}{2\alpha}})\right|}{|\yy-\zz|^\gamma}=:\mathrm{K}_1+\mathrm{K}_2.
\end{align*}
On one hand, using the property of $\phi_{\frac{|\xx|}{2}}(\yy),$ one has
\begin{align*}
 \mathrm{K}_1
\leq&\frac{Cs^{\frac{2\alpha-1}{\alpha}}}{|\xx|^{1+\gamma}}\sup_{\yy\in\RR^3\backslash B_{\frac{|\xx|}{2}}(0)}\frac{1}{|\yy|^{4\alpha-2}} \sup_{\yy\in\RR^3\backslash B_{\frac{|\xx|}{2}}(0)}\left|\yy/s^{\frac{1}{2\alpha}}\right|^{4\alpha-2}|\mathbf{f}_2|(\yy/s^{\frac{1}{2\alpha}})\\
\leq&\frac{Cs^{\frac{2\alpha-1}{\alpha}}}{|\xx|^{4\alpha-1+\gamma}} \sup_{\yy\in\RR^3 }\left|\yy \right|^{4\alpha-2}|\mathbf{f}_2|(\yy ).
\end{align*}
On the other hand, we split $\mathrm{K}_2$ into two parts as follows:
\begin{align*}
\mathrm{K}_2
 =&\left\|\nabla\left(\phi_{\frac{|\xx|}{2}}(\cdot)\right)\right\|_{L^\infty(\RR^3)}
\sup_{\substack{\yy,\,\zz\in\RR^3\backslash B_{\frac{|\xx|}{2}}(0)\\ |\yy-\zz|\geq|\frac{\xx|}{100}}}\frac{\left|\mathbf{f}_2(\yy/s^{\frac{1}{2\alpha}})
-\mathbf{f}_2(\zz/s^{\frac{1}{2\alpha}})\right|}{|\yy-\zz|^\gamma}\\
&+\left\|\nabla\left(\phi_{\frac{|\xx|}{2}}(\cdot)\right)\right\|_{L^\infty(\RR^3)}
\sup_{\substack{\yy,\,\zz\in\RR^3\backslash B_{\frac{|\xx|}{2}}(0)\\ |\yy-\zz|<\frac{|\xx|}{100}}}\frac{\left|\mathbf{f}_2(\yy/s^{\frac{1}{2\alpha}})
-\mathbf{f}_2(\zz/s^{\frac{1}{2\alpha}})\right|}{|\yy-\zz|^\gamma}:=\mathrm{K}_2^1+\mathrm{K}_2^2.
\end{align*}
Obviously, we have
\begin{align*}
\mathrm{K}_2^1\leq &\frac{C}{|\xx|^{1+\gamma}}\sup_{\yy,\,\zz\in\RR^3\backslash B_{\frac{|\xx|}{2}}(0)} {\left|\mathbf{f}_2(\yy/s^{\frac{1}{2\alpha}})
-\mathbf{f}_2(\zz/s^{\frac{1}{2\alpha}})\right|} \\
\leq& \frac{C}{|\xx|^{1+\gamma}}\sup_{\yy \in\RR^3\backslash B_{\frac{|\xx|}{2}}(0)}\frac{1}{|\yy|^{4\alpha-2}} \left(|\yy|^{4\alpha-2}\left|\mathbf{f}_2(\yy/s^{\frac{1}{2\alpha}})
 \right| \right)
 \leq  \frac{Cs^{\frac{2\alpha-1}{\alpha}}}{|\xx|^{4\alpha-1+\gamma}}\sup_{\yy\in\RR^3} |\yy|^{4\alpha-2}\left|\mathbf{f}_2\right|(\yy ).
\end{align*}
For the term $\mathrm{K}_2^2,$ by the classical mean value theorem of the differential calculus, we know there exist $\theta\in(0,1)$ such  that
\begin{align*}
\mathrm{K}_2^2\leq&\frac{C}{|\xx|}
\sup_{\substack{\yy,\,\zz\in\RR^3\backslash B_{\frac{|\xx|}{2}}(0)\\ |\yy-\zz|<\frac{|\xx|}{100} }} \left|D\mathbf{f}_2\right|(\yy+\theta(\zz-\yy)/s^{\frac{1}{2\alpha}})  |\yy-\zz|^{1-\gamma}  \\
\leq& \frac{C}{|\xx|^\gamma}
\sup_{\substack{\yy,\,\zz\in\RR^3\backslash B_{\frac{|\xx|}{2}}(0)\\ |\yy-\zz|<\frac{|\xx|}{100} }}\frac{1}{\left|\yy+\theta(\zz-\yy)\right|^{4\alpha-1}} \left|\yy+\theta(\zz-\yy)\right|^{4\alpha-1}\left|D\mathbf{f}_2\right|(\yy+\theta(\zz-\yy)/s^{\frac{1}{2\alpha}})\\
\leq& \frac{Cs^{\frac{2\alpha-1}{\alpha}}}{|\xx|^{4\alpha-1+\gamma}}
\sup_{ \yy  \in\RR^3 } \left|\yy \right|^{4\alpha-1}\left|D\mathbf{f}_2\right|(\yy ).
\end{align*}
Collecting estimates for $\mathrm{K}_1$, $\mathrm{K}_2^1$ and $\mathrm{K}_2^2$, we readily have
\begin{equation}\label{eq-estimate-L1}
{\rm L_1}\leq \frac{Cs^{\frac{2\alpha-1}{\alpha}}}{|\xx|^{4\alpha-1+\gamma}}\bigg(\sup_{\yy\in\RR^3} |\yy|^{4\alpha-2}\left|\mathbf{f}_2\right|(\yy )+
\sup_{ \yy  \in\RR^3 } \left|\yy \right|^{4\alpha-1}\left|D\mathbf{f}_2\right|(\yy )\bigg).
\end{equation}
Next, we deal with the term ${\rm L}_2.$  By the triangle inequality, we obtain
\begin{align*}
{\rm L}_2\leq&\sup_{\substack{\yy,\,\zz\in\RR^3 \backslash B_{\frac{|\xx|}{2}}(0)\\ \yy\neq\zz }}\frac{\big|\phi_{\frac{|\xx|}{2}}^c(\yy)
-\phi_{\frac{|\xx|}{2}}^c(\zz)\big|\,|\operatorname{div}\mathbf{f}_2(\yy /s^{\frac{1}{2\alpha}})|}{|\yy-\zz|^\gamma} \\
&+\sup_{\substack{\yy,\,\zz\in\RR^3 \backslash B_{\frac{|\xx|}{2}}(0)\\ \yy\neq\zz }}\frac{\phi_{\frac{|\xx|}{2}}^c(\zz)| \operatorname{div}\mathbf{f}_2(\yy /s^{\frac{1}{2\alpha}})
-\operatorname{div}\mathbf{f}_2(\zz /s^{\frac{1}{2\alpha}} )\big|}{|\yy-\zz|^\gamma}=:{\rm L}_2^1+{\rm L}_2^2.
\end{align*}
Since
\[\|\phi_{\frac{|\xx|}{2}}\|_{\dot{C}^\gamma(\RR^3)}\leq C\|\phi_{\frac{|\xx|}{2}}\|^{1-\gamma}_{L^\gamma(\RR^3)}\|\nabla\phi_{\frac{|\xx|}{2}}\|^{\gamma}_{L^\gamma(\RR^3)}\leq \frac{C}{|\xx|^\gamma},\]
we have
\begin{equation}\label{eq-ii-1}
\begin{split}
{\rm L}_2^1=&\sup_{\substack{\yy,\,\zz\in\RR^3 \backslash B_{\frac{|\xx|}{2}}(0)\\ \yy\neq\zz }}\frac{\big|\phi_{\frac{|\xx|}{2}}(\yy)
-\phi_{\frac{|\xx|}{2}}(\zz)\big|\,|\operatorname{div}\mathbf{f}_2(\yy /s^{\frac{1}{2\alpha}})|}{|\yy-\zz|^\gamma}\\
\leq& \frac{Cs^{\frac{4\alpha-1}{2\alpha}}}{|\xx|^{4\alpha-1+\gamma}} \sup_{\yy\in\RR^3 }\left|\yy \right|^{4\alpha-1}|\operatorname{div}\mathbf{f}_2|(\yy ).
\end{split}
\end{equation}
As for the last term ${\rm L}_2^2$,  we note that
\begin{align*}
{\rm L}_2^2\leq&\sup_{\substack{\yy,\,\zz\in\RR^3 \backslash B_{\frac{|\xx|}{2}}(0)\\ |\yy-\zz|\geq\frac{|\xx|}{100} }}\frac{ | \operatorname{div}\mathbf{f}_2(\yy /s^{\frac{1}{2\alpha}})
-\operatorname{div}\mathbf{f}_2(\zz /s^{\frac{1}{2\alpha}} )\big|}{|\yy-\zz|^\gamma}\\
&+\sup_{\substack{\yy,\,\zz\in\RR^3 \backslash B_{\frac{|\xx|}{2}}(0)\\ |\yy-\zz|<\frac{|\xx|}{100} }}\frac{ | \operatorname{div}\mathbf{f}_2(\yy /s^{\frac{1}{2\alpha}})
-\operatorname{div}\mathbf{f}_2(\zz /s^{\frac{1}{2\alpha}} )\big|}{|\yy-\zz|^\gamma}=:{\rm L}_2^{2,1}+{\rm L}_2^{2,2}.
\end{align*}
In the same way as in the proof of ${\rm K}_2^1,$ it is obvious
\begin{equation}\label{eq-ii-2-1}
{\rm L}_2^{2,1}\leq\frac{Cs^{\frac{4\alpha-1}{2\alpha}}}{|\xx|^{4\alpha-1+\gamma}} \sup_{\yy\in\RR^3 }\left|\yy \right|^{4\alpha-1}|\operatorname{div}\mathbf{f}_2|(\yy ).
\end{equation}
For the term ${\rm L}_2^{2,2}$, we observe that
\begin{align*}
{\rm L}_2^{2,2}\leq &\frac{C }{|\xx|^{4\alpha-1+\gamma}}\sup_{\substack{\yy,\,\zz\in\RR^3 \backslash B_{\frac{|\xx|}{2}}(0)\\ |\yy-\zz|<\frac{|\xx|}{100} }}\frac{|\yy|^{4\alpha-1+\gamma} | \operatorname{div}\mathbf{f}_2(\yy /s^{\frac{1}{2\alpha}})
-\operatorname{div}\mathbf{f}_2(\zz /s^{\frac{1}{2\alpha}} )\big|}{|\yy-\zz|^\gamma}
\end{align*}
By the triangle inequality, we have
\begin{align*}
&\sup_{\substack{\yy,\,\zz\in\RR^3 \backslash B_{\frac{|\xx|}{2}}(0)\\ |\yy-\zz|<\frac{|\xx|}{100} }}\frac{|\yy|^{4\alpha-1+\gamma} | \operatorname{div}\mathbf{f}_2(\yy /s^{\frac{1}{2\alpha}})
-\operatorname{div}\mathbf{f}_2(\zz /s^{\frac{1}{2\alpha}} )\big|}{|\yy-\zz|^\gamma}\\
\leq&\sup_{\substack{\yy,\,\zz\in\RR^3 \backslash B_{\frac{|\xx|}{2}}(0)\\ |\yy-\zz|<\frac{|\xx|}{100} }}\frac{\big| |\yy|^{4\alpha-1+\gamma} \operatorname{div}\mathbf{f}_2(\yy /s^{\frac{1}{2\alpha}})
-|\zz|^{4\alpha-1+\gamma} \operatorname{div}\mathbf{f}_2(\zz /s^{\frac{1}{2\alpha}} )\big|}{|\yy-\zz|^\gamma}\\
&+\sup_{\substack{\yy,\,\zz\in\RR^3 \backslash B_{\frac{|\xx|}{2}}(0)\\ |\yy-\zz|<\frac{|\xx|}{100} }}\frac{\left||\yy|^{4\alpha-1+\gamma}-|\zz|^{4\alpha-1+\gamma} \right| \big| \operatorname{div}\mathbf{f}_2(\yy /s^{\frac{1}{2\alpha}})
-\operatorname{div}\mathbf{f}_2(\zz /s^{\frac{1}{2\alpha}} )\big|}{|\yy-\zz|^\gamma}\\
=:&{\rm A}+{\rm B}.
\end{align*}
According to the definition of the H\"older space,  we find that
\begin{align*}
{\rm A}\leq Cs^{\frac{4\alpha-1}{2\alpha}}\left\||\cdot|^{4\alpha-1+\gamma}  \operatorname{div}\mathbf{f}_2\right\|_{\dot{C}^\gamma(\RR^3)}.
\end{align*}
On the other hand,  the symmetry of ${\rm B}$ helps us to conclude
\begin{align*}
{\rm B}\leq&2\sup_{\substack{\yy,\,\zz\in\RR^3 \backslash B_{\frac{|\xx|}{2}}(0)\\ |\yy-\zz|<\frac{|\xx|}{100} }}\frac{\left||\yy|^{4\alpha-1+\gamma}-|\zz|^{4\alpha-1+\gamma} \right| }{|\yy-\zz|^\gamma}\big | \operatorname{div}\mathbf{f}_2 \big|(\yy /s^{\frac{1}{2\alpha}})\\
\leq &C \sup_{\substack{\yy,\,\zz\in\RR^3 \backslash B_{\frac{|\xx|}{2}}(0)\\ |\yy-\zz|<\frac{|\xx|}{100} }} \max\{|\yy|^{4\alpha-2+\gamma},|\zz|^{4\alpha-2+\gamma} \} |\yy-\zz|^{1-\gamma} \big|\operatorname{div}\mathbf{f}_2 \big|(\yy /s^{\frac{1}{2\alpha}})\\
\leq &C \sup_{\substack{\yy,\,\zz\in\RR^3 \backslash B_{\frac{|\xx|}{2}}(0)\\ |\yy-\zz|<\frac{|\xx|}{100} }}  |\yy|^{4\alpha-1}  \big|\operatorname{div}\mathbf{f}_2 \big|(\yy /s^{\frac{1}{2\alpha}})
\leq Cs^{\frac{4\alpha-1}{2\alpha}}\sup_{ \yy \in\RR^3   }|\yy|^{4\alpha-1}| \operatorname{div}\mathbf{f}_2 |(\yy ).
\end{align*}
Collecting both estimates for $\rm A$ and $\rm B$, we have
\begin{equation*}
{\rm L}_2^{2,2}\leq  \frac{C s^{\frac{4\alpha-1}{2\alpha}}}{|\xx|^{4\alpha-1+\gamma}} \bigg(\left\||\cdot|^{4\alpha-1+\gamma}  \operatorname{div}\mathbf{f}_2\right\|_{\dot{C}^\gamma(\RR^3)}+\sup_{ \yy \in\RR^3   }|\yy|^{4\alpha-1}| \operatorname{div}\mathbf{f}_2 |(\yy )\bigg).
\end{equation*}
This estimate together with \eqref{eq-ii-1} and \eqref{eq-ii-2-1} gives
\begin{equation}\label{eq-estimate-L2}
{\rm L}_2\leq  \frac{C s^{\frac{4\alpha-1}{2\alpha}}}{|\xx|^{4\alpha-1+\gamma}} \bigg(\left\||\cdot|^{4\alpha-1+\gamma}  \operatorname{div}\mathbf{f}_2\right\|_{\dot{C}^\gamma(\RR^3)}+\sup_{ \yy \in\RR^3   }|\yy|^{4\alpha-1}| \operatorname{div}\mathbf{f}_2 |(\yy )\bigg).
\end{equation}
Inserting \eqref{eq-estimate-L1} and \eqref{eq-estimate-L2} into \eqref{sharp-1} and then using the conditions \eqref{eq-con-optimal-last} and \eqref{eq.con-cr}, we eventually obtain
\begin{align*}
{\rm II}_{2,2}
\leq&C \frac{1}{|\xx|^{4\alpha-1+\gamma}} \sum_{N_{\xx}\leq q<0}2^{-q\gamma}\int_{0}^{1}\left(1+s^{-\frac{1}{2\alpha}}\right) {\rm d}s\\
\leq&  \frac{C}{|\xx|^{4\alpha-1+\gamma}}2^{-\gamma N_{\xx}}\leq \frac{C}{|\xx|^{4\alpha-1}}.
\end{align*}
Thus, we complete the proof of Proposition \ref{prop-decay-begin-third}.
\end{proof}
With Proposition \ref{prop-decay-begin-third} in hand, we turn to show the optimal estimate in Theorem \ref{thm-leray-per}. To do this, we just need to prove \eqref{eq-con-optimal-last} and \eqref{eq.con-cr}. Since $\sigma\in C^{1,0}(\mathbb{S}^2)$ and $\mathbf{f}_2=\Div(\uu_0\otimes\uu_0)$,  we have in terms of Proposition~4.1 established in \cite{LMZ21} that
\begin{equation}\label{eq-0824-1}
\begin{split}
&\sup_{\xx\in\RR^3}\left|| \xx|^{4\alpha-2}\mathbf{f}_2(\xx)\right|+\sup_{\xx\in\RR^3}\left||\xx|^{4\alpha-1}\nabla \mathbf{f}_2(\xx)\right|\\
\leq& C\bigg(\sup_{\xx\in\RR^3}|\xx|^{2\alpha-1}|\uu_0|(\xx)\bigg)^{2}+\sup_{\xx\in\RR^3}|\xx|^{2\alpha-1}|\uu_0|(\xx)\sup_{\xx\in\RR^3}|\xx|^{2\alpha }|\Div\uu_0|(\xx)<\infty.
\end{split}
\end{equation}
Thus, our task is now to show
\begin{equation}\label{eq.claim-2}
 \left\||\cdot|^{4\alpha-1+\gamma}\operatorname{div}\mathbf{f}_2\right\|_{\dot{B}^\gamma_{\infty,\infty}(\RR^3)}
<\infty.
\end{equation}
In view of the Bony paraproduct, we write
\begin{equation*}
|\xx|^{4\alpha-1+\gamma}\Div \mathbf{f}_2=\dot{T}_{|\xx|^{4\alpha-3+\gamma} }|\xx|^{2}\Div \mathbf{f}_2+\dot{T}_{|\xx|^{2} \Div\mathbf{f}_2}{|\xx|^{4\alpha-3+\gamma} }+\dot{R}\left(|\xx|^{4\alpha-3+\gamma} ,|\xx|^{2} \Div\mathbf{f}_2\right).
\end{equation*}
In terms of the property of support, we can show by using $\alpha\in[5/6,1]$ and $\gamma\in(0,1)$ that
\begin{align*}
&\left\|\dot{T}_{|\cdot|^{2}\Div \mathbf{f}_2}{|\cdot|^{4\alpha-3+\gamma} }\right\|_{\dot{B}^{ \gamma}_{\infty,\infty}(\RR^3)}\\
\leq&C\sup_{q\in\ZZ}2^{q\gamma}\left\|\dot{S}_{q-1}\big(|\cdot|^{2}\Div \mathbf{f}_2\big)\right\|_{L^\infty(\RR^3)}\left\|\dot{\Delta}_{q}\big(|\cdot|^{4\alpha-3+\gamma}\big)\right\|_{L^\infty(\RR^3)}\\
\leq &C\left\| |\cdot|^{2}\Div \mathbf{f}_2\right\|_{\dot{B}^{3-4\alpha}_{\infty,\infty}(\RR^3)}\left\| |\cdot|^{4\alpha-3+\gamma} \right\|_{\dot{B}^{4\alpha-3+\gamma}_{\infty,\infty}(\RR^3)}\leq C\left\| |\cdot|^{2}\Div \mathbf{f}_2\right\|_{\dot{B}^{3-4\alpha}_{\infty,\infty}(\RR^3)},
\end{align*}
and
 \begin{align*}
&\left\|\dot{R}\left(|\cdot|^{4\alpha-3+\gamma} ,|\cdot|^{2} \Div\mathbf{f}_2\right)\right\|_{\dot{B}^{ \gamma}_{\infty,\infty}(\RR^3)}\\
\leq&C\sup_{q\in\ZZ}2^{q\gamma}\left\|\dot{\Delta}_{q}\big(|\cdot|^{2}\Div \mathbf{f}_2\big)\right\|_{L^\infty(\RR^3)}\left\|\widetilde{\dot{\Delta}}_{q}\big(|\cdot|^{4\alpha-3+\gamma}\big)\right\|_{L^\infty(\RR^3)}\\
\leq &C\left\| |\cdot|^{2}\Div \mathbf{f}_2\right\|_{\dot{B}^{3-4\alpha}_{\infty,\infty}(\RR^3)}\left\| |\cdot|^{4\alpha-3+\gamma} \right\|_{\dot{B}^{4\alpha-3+\gamma}_{\infty,\infty}(\RR^3)}\leq C\left\| |\cdot|^{2}\Div \mathbf{f}_2\right\|_{\dot{B}^{3-4\alpha}_{\infty,\infty}(\RR^3)}.
\end{align*}
Since $\mathbf{f}_2=-\uu_0\otimes\uu_0$ and $\sigma\in C^{1,0}(\mathbb{S}^2),$ we have by \eqref{eq.bern-g-infty}  that
\begin{equation}\label{eq-22-0825-1}
\begin{split}
&\left\| |\cdot|^{2}\Div \mathbf{f}_2\right\|_{\dot{B}^{3-4\alpha}_{\infty,\infty}(\RR^3)}\\
\leq &C\left\| |\cdot|^{2}\Div \mathbf{f}_2\right\|_{L^{ \frac{3}{4\alpha-3},\infty} (\RR^3)}\\
\leq&C\sup_{\xx\in\RR^3}|\xx|^{2\alpha-1}|\uu_0|(\xx)\sup_{\xx\in\RR^3}|\xx|^{2\alpha }|\Div\uu_0|(\xx)\left\| |\cdot|^{3-4\alpha} \right\|_{L^{ \frac{3}{4\alpha-3},\infty} (\RR^3)}.
\end{split}
\end{equation}
This estimate together with   Proposition~4.1 established in \cite{LMZ21} means that
\begin{equation}\label{eq.20220827-1}
\left\|\dot{T}_{|\cdot|^{2}\Div \mathbf{f}_2}{|\cdot|^{4\alpha-3+\gamma} }\right\|_{\dot{B}^{ \gamma}_{\infty,\infty}(\RR^3)}+\left\|\dot{R}\left(|\cdot|^{4\alpha-3+\gamma} ,|\cdot|^{2} \Div\mathbf{f}_2\right)\right\|_{\dot{B}^{ \gamma}_{\infty,\infty}(\RR^3)}
<\infty.
\end{equation}
Next, we turn to bound the paraproduct term $\dot{T}_{|\xx|^{4\alpha-3+\gamma} }|\xx|^{2}\Div \mathbf{f}_2.$ Since
\[\int_{\RR^3}h_q(\xx-\yy)\,\mathrm{d}\yy=\int_{\RR^3}h (\xx-\yy)\,\mathrm{d}\yy=\hat{h}(0)=1\quad\text{for each}\,\,q\in\ZZ,\]
we have
\begin{align*}
&\dot{S}_{q-1}\big(|\cdot|^{4\alpha-3+\gamma}\big)\\=&\left(\dot{S}_{q-1} |\cdot|^{4\alpha-3+\gamma}-2^{4\alpha-3+\gamma} |\xx|^{4\alpha-3+\gamma}\right)+2^{4\alpha-3+\gamma} |\xx|^{4\alpha-3+\gamma}\\
=&\int_{\RR^3}h_q(\xx-\yy)\left(|\yy|^{4\alpha-3+\gamma}-2^{4\alpha-3+\gamma} |\xx|^{4\alpha-3+\gamma}\right)\,\mathrm{d}\yy+2^{4\alpha-3+\gamma} |\xx|^{4\alpha-3+\gamma}.
\end{align*}
Performing the fact $(a+b)^p\leq 2^{p}(a^p+b^p)$ with $p\geq0,$ one has
\begin{align*}
&\left|\int_{\RR^3}h_q(\xx-\yy)\left(|\yy|^{4\alpha-3+\gamma}-2^{4\alpha-3+\gamma} |\xx|^{4\alpha-3+\gamma}\right)\,\mathrm{d}\yy\right|\\
\leq&2^{4\alpha-3+\gamma}\int_{\RR^3}|\xx-\yy|^{4\alpha-3+\gamma}|h_q|(\xx-\yy)\,\mathrm{d}\yy\leq C2^{-q(4\alpha-3+\gamma)}.
\end{align*}
Thus, we have
\begin{equation*}
\left|\dot{S}_{q-1}\big(|\cdot|^{4\alpha-3+\gamma}\big)\right|\leq C2^{-q(4\alpha-3+\gamma)}+2^{4\alpha-3+\gamma} |\xx|^{4\alpha-3+\gamma}.
\end{equation*}
With the inequality in hand, we can show by \eqref{eq.bern-g-infty} that
\begin{align*}
  \left\|\dot{T}_{|\cdot|^{4\alpha-3+\gamma} }|\cdot|^{2}\Div \mathbf{f}_2\right\|_{\dot{B}^{ \gamma}_{\infty,\infty}(\RR^3)}
\leq&C\sup_{q\in\ZZ}2^{q\gamma}\left\|\dot{S}_{q-1}\big(|\cdot|^{4\alpha-3+\gamma}\big)\dot{\Delta}_{q}\big(|\cdot|^{2}\Div \mathbf{f}_2\big)\right\|_{L^\infty(\RR^3)}\\
\leq &C\sup_{q\in\ZZ}2^{q(4\alpha-3+\gamma)}\left\| |\cdot|^{4\alpha-3+\gamma} \dot{\Delta}_{q}\big(|\cdot|^{2}\Div \mathbf{f}_2\big)\right\|_{L^{\frac{3}{4\alpha-3},\infty}(\RR^3)}\\
 &+C\sup_{q\in\ZZ} \left\| \dot{\Delta}_{q}\big(|\cdot|^{2}\Div \mathbf{f}_2\big)\right\|_{L^{\frac{3}{4\alpha-3},\infty}(\RR^3)}=:{\rm A}_1+{\rm A}_2.
\end{align*}
Since $\sigma\in C^{1,0}(\mathbb{S}^2)$ and $\mathbf{f}_2=\Div(\uu_0\otimes\uu_0)$, we have by Proposition~4.1 established in \cite{LMZ21} that
\begin{equation}\label{eq.20220827-2}
{\rm A}_2
\leq C\sup_{\xx\in\RR^3}|\xx|^{2\alpha-1}|\uu_0|(\xx)\sup_{\xx\in\RR^3}|\xx|^{2\alpha }|\Div\uu_0|(\xx)\left\| |\cdot|^{3-4\alpha} \right\|_{L^{ \frac{3}{4\alpha-3},\infty} (\RR^3)}<\infty.
\end{equation}
It remains for us to bound ${\rm A}_2$.  Since $\mathbf{f}_2=-\uu_0\otimes\uu_0,$ we rewrite it in  terms of the Bony paraproduct that
\begin{equation*}
\mathbf{f}_{2}=-\sum_{i,j=1}^3\left(\dot{T}_{\uu_0^i}\uu_0^j+\dot{T}_{\uu_0^j}\uu_0^i+\dot{R}\left(\uu_0^i,\uu_0^j\right)\right).
\end{equation*}
Moreover, taking Fourier transform and property of support, we observe that
\begin{align*}
&\mathscr{F}\left(\dot{\Delta}_{q}\big(|\cdot|^{2}\Div \mathbf{f}_2\big)\right)(\xi)\\
=&-\varphi\left( {\xi}/{2^q}\right)\Delta_\xi\left(i\xi\widehat{\mathbf{f}}_2(\xi)\right)\\
=&\sum_{i,j =1}^3\sum_{|k-q|\leq5}\varphi\left( {\xi}/{2^q}\right) \Delta_\xi\Big(i\xi_i\widehat{\dot{S}_{k-1}\uu_0^i}(\xi)\widehat{\dot{\Delta}_k\uu_0^j}(\xi)+i\xi_i\widehat{\dot{S}_{k-1}\uu_0^j}(\xi)\widehat{\dot{\Delta}_k\uu_0^i}(\xi)\Big)\\
&+\sum_{i,j =1}^3\sum_{k\geq q-5}\varphi\left( {\xi}/{2^q}\right)\Delta_\xi\Big( i\xi_i\widehat{\dot{\Delta}_{k}\uu_0^i}(\xi)\widehat{\dot{\Delta}_k\uu_0^j}(\xi)\Big).
\end{align*}
Furthermore, we have by the triangle inequality that
\begin{align*}
{\rm A}_1\leq&C \sup_{q\in\ZZ}\sum_{|k-q|\leq5}2^{q(4\alpha-3+\gamma)}\left\| |\cdot|^{4\alpha-3+\gamma} \dot{\Delta}_{q}\big(|\cdot|^{2}\dot{S}_{k-1}\uu_0\cdot\nabla\dot{\Delta}_k\uu_0 \big)\right\|_{L^{\frac{3}{4\alpha-3},\infty}(\RR^3)}\\
&+C \sup_{q\in\ZZ}\sum_{|k-q|\leq5}2^{q(4\alpha-3+\gamma)}\left\| |\cdot|^{4\alpha-3+\gamma} \dot{\Delta}_{q}\big(|\cdot|^{2}\dot{\Delta}_k\uu_0\cdot\nabla\dot{S}_{k-1}\uu_0 \big)\right\|_{L^{\frac{3}{4\alpha-3},\infty}(\RR^3)}\\
&+ C\sup_{q\in\ZZ}\sum_{k\geq q-5}2^{q(4\alpha-3+\gamma)}\left\| |\cdot|^{4\alpha-3+\gamma} \dot{\Delta}_{q}\big(|\cdot|^{2}\dot{\Delta}_k\uu_0\cdot\nabla\widetilde{\dot{\Delta}}_{k}\uu_0 \big)\right\|_{L^{\frac{3}{4\alpha-3},\infty}(\RR^3)}\\
=:&{\rm A}_1^1+{\rm A}_1^2+{\rm A}_1^3.
\end{align*}
To bound these terms, we need the following lemma.
\begin{lemma}\label{lem-last}
Let $m\in(0,\infty),$ $1<r\leq p<\infty$ and $f\in L^{r,\infty}(\RR^3)$ satisfying
\[\||\cdot|^mf\|_{L^{r,\infty}(\RR^3)}<\infty.\]
  Then there holds that
\begin{equation}\label{eq-last-1}
\big\||\cdot|^{m} \dot{\Delta}_{q}f\big\|_{L^{p,\infty}(\RR^3)}\leq C\big\|\dot{\Delta}_{q}\big(|\cdot|^{m} f\big)\big\|_{L^{p,\infty}(\RR^3)}+C2^{-q(m-3( 1/r- 1/p) )}\|f\|_{L^{r,\infty}(\RR^3)}.
\end{equation}
\end{lemma}
\begin{proof}
First of all, we observe that
\begin{equation*}
 |\xx|^{m} \dot{\Delta}_{q}f =2^m\dot{\Delta}_{q}\big(|\cdot|^{m} f\big)+\int_{\RR^3}h_q(\xx-\yy)\left( |\xx|^{m}-2^m|\yy|^m\right)f(\yy)\,\mathrm{d}\yy.
\end{equation*}
It is obvious by using the fact $(a+b)^p\leq 2^p (a^p+b^p)$ with $p\geq 1$ that
\begin{equation}\label{eq.220826-1}
\begin{split}
I:=&\left|\int_{\RR^3}h_q(\xx-\yy)\left( |\xx|^{m}-2^m|\xx-\yy|^m\right)f(\yy)\,\mathrm{d}\yy\right|\\
\leq& 2^m2^{3q}\int_{\RR^3}|\xx-\yy|^m|h|\big(2^q(\xx-\yy)\big)|f|(\yy)\, \mathrm{d}\yy.
\end{split}
\end{equation}
Moreover, by the Young inequality, we obtain
\begin{align*}
 \left\|I\right\|_{L^{p,\infty}(\RR^3)}
\leq& C2^{-q(m-3( 1/r- 1/p) )}\big\||\cdot|^m h\big\|_{L^s(\RR^3)}\|f\|_{L^{r,\infty}(\RR^3)} \\
 \leq&  C2^{-q(m-3( 1/r- 1/p) )}\|f\|_{L^{r,\infty}(\RR^3)},
\end{align*}
where $s$ satisfies $1+\frac1p=\frac1s+\frac1r.$

Thus, it follows the  desired estimate  by the triangle inequality that
\begin{equation*}
\big\||\cdot|^{m} \dot{\Delta}_{q}f\big\|_{L^{p,\infty}(\RR^3)}\leq C\big\|\dot{\Delta}_{q}\big(|\cdot|^{m} f\big)\big\|_{L^{p,\infty}(\RR^3)}+C2^{-q(m-3( 1/r- 1/p) )}\|f\|_{L^{r,\infty}(\RR^3)}.
\end{equation*}
And then we complete the proof of Lemma \ref{lem-last}.
\end{proof}
By \eqref{eq-last-1} in Lemma \ref{lem-last}, we have
\begin{align*}
{\rm A}_1^1\leq&C \sup_{q\in\ZZ}\sum_{|k-q|\leq5}2^{q(4\alpha-3+\gamma)}\left\| \dot{\Delta}_{q}\big(|\cdot|^{4\alpha-1+\gamma}\dot{S}_{k-1}\uu_0\cdot\nabla \dot{\Delta}_k\uu_0 \big)\right\|_{L^{\frac{3}{4\alpha-3},\infty}(\RR^3)}\\
&+C \sup_{q\in\ZZ}\sum_{|k-q|\leq5} \left\| |\cdot|^{2}\dot{S}_{k-1}\uu_0\cdot\nabla \dot{\Delta}_k\uu_0 \right\|_{L^{\frac{3}{4\alpha-3},\infty}(\RR^3)}=:{\rm A}_1^{1,1}+{\rm A}_1^{1,2}.
\end{align*}
Since
\begin{align*}
\big\||\cdot|^{2\alpha-1}\dot{S}_{q-1}\uu_0\big\|_{L^\infty(\RR^3)}\leq &\sup_{k\leq q-1}\big\||\cdot|^{2\alpha-1}\dot{\Delta}_{k}\uu_0\big\|_{L^\infty(\RR^3)}\\
\leq&C\sup_{k\leq q-1}2^{-k(2\alpha-1)}\left\|\dot{\Delta}_{k}\left( {\sigma}/{|\cdot|^{2\alpha-1}}\right)\right\|_{L^\infty(\RR^3)}\\
\leq& C\left\| {\sigma}/{|\cdot|^{2\alpha-1}}\right\|_{L^{\frac{3}{2\alpha-1},\infty}(\RR^3)}
 \leq C\|\sigma\|_{L^\infty(\mathbb{S}^2)},
\end{align*}
we have by  \eqref{eq.bern-g-infty} and Lemma \ref{lem-exe-decay} that
\begin{align*}
{\rm A}_1^{1,2}\leq &C\big\||\cdot|^{2\alpha-1}\dot{S}_{q-1}\uu_0\big\|_{L^\infty(\RR^3)}\sup_{q\in\ZZ} \big\| |\cdot|^{3-2\alpha }\nabla \dot{\Delta}_q\uu_0 \big\|_{L^{\frac{3}{4\alpha-3},\infty}(\RR^3)}\\
\leq&C\|\sigma\|_{L^\infty(\mathbb{S}^2)}\sup_{q\in\ZZ} 2^{q2(\alpha-1)}\left\| \dot{\Delta}_q\big ( {\sigma}/{|\cdot|^{2\alpha-1}} \big)\right\|_{L^{\frac{3}{4\alpha-3},\infty}(\RR^3)}\\
\leq&C\|\sigma\|_{L^\infty(\mathbb{S}^2)} \left\| {\sigma}/{|\cdot|^{2\alpha-1}}\right\|_{L^{\frac{3}{2\alpha-1},\infty}(\RR^3)} \leq C\|\sigma\|^2_{L^\infty(\mathbb{S}^2)}.
\end{align*}
By the H\"older inequality,  \eqref{eq.bern-g-infty} and Lemma \ref{lem-exe-decay}, we find that
\begin{align*}
{\rm A}_1^{1,1}\leq&C\sup_{q\in\ZZ}2^{q(4\alpha-3+\gamma)} \big\||\cdot|^{2\alpha-1}\dot{S}_{q-1}\uu_0\big\|_{L^\infty(\RR^3)}\left\||\cdot|^{2\alpha+\gamma}
\dot{\Delta}_q\nabla\uu_0 \right\|_{L^{\frac{3}{4\alpha-3},\infty}(\RR^3)}\\
\leq&C \sup_{q\in\ZZ}\big\||\cdot|^{2\alpha-1}\dot{S}_{q-1}\uu_0\big\|_{L^{\frac{3}{4\alpha-3}}(\RR^3)}
\sup_{q\in\ZZ}2^{2q(\alpha-1)}\left\|
\dot{\Delta}_q\big ( {\sigma}/{|\cdot|^{2\alpha-1}}\big )\right\|_{L^{\frac{3}{4\alpha-3}}(\RR^3)}\\
\leq&C\|\sigma\|_{L^\infty(\mathbb{S}^2)} \left\| {\sigma}/{|\cdot|^{2\alpha-1}}\right\|_{L^{\frac{3}{2\alpha-1},\infty}(\RR^3)} \leq C\|\sigma\|^2_{L^\infty(\mathbb{S}^2)}.
\end{align*}
Collecting estimates for ${\rm A}_1^{1,1}$ and ${\rm A}_1^{1,2}$, we readily have
\begin{equation}\label{eq-A11}
{\rm A}_1^{1} \leq C\|\sigma\|^2_{L^\infty(\mathbb{S}^2)}.
\end{equation}
Similarly, we can show
\begin{equation}\label{eq-A12}
{\rm A}_1^{2} \leq C\|\sigma\|^2_{L^\infty(\mathbb{S}^2)}.
\end{equation}
Now we focus on the term  ${\rm A}_1^{3}$. By \eqref{eq-last-1} in Lemma \ref{lem-last}, we have
\begin{align*}
 {\rm A}_1^3\leq&C\sup_{q\in\ZZ}\sum_{k\geq q-5}2^{q(4\alpha-3+\gamma)}\left\| \dot{\Delta}_{q}\left(|\cdot|^{4\alpha-1+\gamma} \big(\dot{\Delta}_k\uu_0\cdot\nabla\widetilde{\dot{\Delta}}_{k}\uu_0 \big)\right)\right\|_{L^{\frac{3}{4\alpha-3},\infty}(\RR^3)}\\
 &+C\sup_{q\in\ZZ}\sum_{k\geq q-5}2^{q2(1-\alpha)}\left\|   |\cdot|^{2} \dot{\Delta}_k\uu_0 \cdot\nabla\widetilde{\dot{\Delta}}_{k}\uu_0  \right\|_{L^{\frac{3}{2\alpha-1},\infty}(\RR^3)}\\
=:& {\rm A}_1^{3,1}+ {\rm A}_1^{3,2}.
\end{align*}
By Lemma \ref{lem-exe-decay}, we have that for each $\alpha\in[5/6,1),$
\begin{align*}
{\rm A}_1^{3,2}\leq&C\sup_{q\in\ZZ}\sum_{k\geq q-5}2^{q2(1-\alpha)}\left\| \big( |\cdot|^{2} \dot{\Delta}_k\uu_0\cdot\nabla\widetilde{\dot{\Delta}}_{k}\uu_0 \big)\right\|_{L^{\frac{3}{2\alpha-1},\infty}(\RR^3)}\\
\leq&C\sup_{q\in\ZZ}\sum_{k\geq q-5}2^{q2(1-\alpha)}\left\|  |\cdot|^2  \dot{\Delta}_k\uu_0\right\|_{L^\infty(\RR^3)}\big\|\nabla\widetilde{\dot{\Delta}}_{k}\uu_0  \big\|_{L^{\frac{3}{2\alpha-1},\infty}(\RR^3)}\\
\leq&C\sup_{q\in\ZZ} 2^{q (1-2\alpha)}\big\|    \dot{\Delta}_q\uuu_0\big\|_{L^\infty(\RR^3)}\left\|{\sigma}/{|\cdot|^{2\alpha-1}} \right\|_{L^{\frac{3}{2\alpha-1},\infty}(\RR^3)}\\
\leq&C\left\|{\sigma}/{|\cdot|^{2\alpha-1}} \right\|^2_{L^{\frac{3}{2\alpha-1},\infty}(\RR^3)}\leq C\|\sigma\|^2_{L^\infty(\mathbb{S}^2)}.
\end{align*}
As for $\alpha=1,$  we observe by \eqref{eq-last-1} in Lemma \ref{lem-last} that
\begin{align*}
{\rm A}_1^{3 }=&C\sup_{q\in\ZZ}\sum_{k\geq q-5}2^{q(1+\gamma)}\left\|  |\cdot|^{1+\gamma} \dot{\Delta}_{q} \big( |\cdot|^{2} \dot{\Delta}_k\uu_0 \cdot\nabla\widetilde{\dot{\Delta}}_{k}\uu_0 \big)\right\|_{L^{3,\infty}(\RR^3)}\\
\leq&{\rm A}_1^{3,1}+\sup_{q\in\ZZ}\sum_{k\geq q-5}2^{ q }\left\|  |\cdot|^{2} \dot{\Delta}_k\uu_0 \cdot\nabla\widetilde{\dot{\Delta}}_{k}\uu_0 \right\|_{L^{\frac{3}{2},\infty}(\RR^3)}=: {\rm A}_1^{3,1}+ {\rm A}_1^{3,3}.
\end{align*}
By the discrete Young inequality and the H\"older inequality and Lemma \ref{lem-exe-decay}, we have
\begin{align*}
{\rm A}_1^{3,3}\leq&\sup_{q\in\ZZ} 2^{q}\left\| \big( |\cdot|^{2} \dot{\Delta}_q\uu_0\cdot\nabla\widetilde{\dot{\Delta}}_{q}\uu_0 \big)\right\|_{L^{\frac32,\infty}(\RR^3)}\\
\leq&C\sup_{q\in\ZZ} 2^{q}\left\|  |\cdot|^2  \dot{\Delta}_q\uu_0\right\|_{L^\infty(\RR^3)}\big\|\nabla\widetilde{\dot{\Delta}}_{q}\uu_0  \big\|_{L^{\frac32,\infty}(\RR^3)}\\
\leq&C\sup_{q\in\ZZ} 2^{-q }\big\|    \dot{\Delta}_q\uuu_0\big\|_{L^\infty(\RR^3)} \sup_{q\in\ZZ}  \| \nabla\widetilde{\dot{\Delta}}_{q}\uu_0  \big\|_{L^{\frac32,\infty}(\RR^3)}
\\
\leq&C\left\|{\sigma}/{|\cdot|^{2\alpha-1}} \right\|_{L^{\frac{3}{2\alpha-1},\infty}(\RR^3)}\| \nabla \uuu_0  \big\|_{L^{\frac32,\infty}(\RR^3)} \leq C\|\sigma\|_{L^\infty(\mathbb{S}^2)}\| \nabla \uuu_0  \big\|_{L^{\frac32,\infty}(\RR^3)}.
\end{align*}
Since $\uuu_0(\xx)=\frac{\sigma}{|\xx|}$ as $\alpha=1$ and $\sigma\in C^{1,0}(\mathbb{S}^2)$, we have by the Leibniz formula that
\begin{align*}
\| \nabla\uuu_0  \big\|_{L^{\frac32,\infty}(\RR^3)}\leq &C\|\sigma\|_{L^\infty(\mathbb{S}^2)} \| 1/|\cdot|^2 \big\|_{L^{\frac32,\infty}(\RR^3)}\\
&+C\|\sigma\|_{\dot{W}^{1,\infty}(\mathbb{S}^2)} \| 1/|\cdot|^2 \big\|_{L^{\frac32,\infty}(\RR^3)}.
\end{align*}
Thus, we have
\begin{equation*}
{\rm A}_1^{3,3}\leq C\|\sigma\|^2_{ {W}^{1,\infty}(\mathbb{S}^2)}.
\end{equation*}
In view of \eqref{eq-last-1} in Lemma \ref{lem-last}, we see that
\begin{align*}
{\rm A}_1^{3,1}\leq&C\sup_{q\in\ZZ}\sum_{k\geq q-5}2^{q(4\alpha-3+\gamma)}\left\| \dot{\Delta}_{q}\big(|\cdot|^{4\alpha-1+\gamma} \big(\dot{\Delta}_k\uu_0\cdot\nabla\widetilde{\dot{\Delta}}_{k}\uu_0 \big)\big)\right\|_{L^{\frac{3}{4\alpha-3},\infty}(\RR^3)}\\
\leq&C\sup_{k\in\ZZ}2^{k(4\alpha-3+\gamma)} \big\|  |\cdot|^{4\alpha-1+\gamma} \big(\dot{\Delta}_k\uu_0\cdot\nabla\widetilde{\dot{\Delta}}_{k}\uu_0 \big) \big\|_{L^{\frac{3}{4\alpha-3},\infty}(\RR^3)}\\
\leq&C\sup_{k\in\ZZ}2^{k(4\alpha-3+\gamma)} \big\|  |\cdot|^{4\alpha-1+\gamma}  \dot{\Delta}_k\uu_0\big\|_{L^\infty(\RR^3)}\big\|\nabla\widetilde{\dot{\Delta}}_{k}\uu_0  \big\|_{L^{\frac{3}{4\alpha-3},\infty}(\RR^3)}.
\end{align*}
Moreover, we have by Lemma \ref{lem-exe-decay} and the Bernstein inequality in Lemma \ref{lem-bern} that
\begin{align*}
&\big\|  |\cdot|^{4\alpha-1+\gamma}   \dot{\Delta}_k\uu_0\big\|_{L^\infty(\RR^3)}\big\|\nabla\widetilde{\dot{\Delta}}_{k}\uu_0  \big\|_{L^{\frac{3}{4\alpha-3},\infty}(\RR^3)}\\
\leq&C2^{-k(4\alpha-1+\gamma)}2^{k(3-2\alpha)}\big\|   \dot{\Delta}_k\uuu_0\big\|_{L^\infty(\RR^3)}\big\| \widetilde{\dot{\Delta}}_{k}\uu_0  \big\|_{L^{\frac{3}{2\alpha-1},\infty}(\RR^3)}\\
\leq&C2^{-k(4\alpha-1+\gamma)}2^{k(3-2\alpha)}\big\|   \dot{\Delta}_k\uuu_0\big\|_{L^\infty(\RR^3)}\big\| \widetilde{\dot{\Delta}}_{k}\uu_0  \big\|_{L^{\frac{3}{2\alpha-1},\infty}(\RR^3)}\\
\leq&C2^{-k(4\alpha-3+\gamma)} \left\|{\sigma}/{|\cdot|^{2\alpha-1}} \right\|^2_{L^{\frac{3}{2\alpha-1},\infty}(\RR^3)},
\end{align*}
which helps us to infer that
\begin{equation*}
{\rm A}_1^{3,1}\leq   C\|\sigma\|^2_{L^\infty(\mathbb{S}^2)}.
\end{equation*}
This estimate together with estimates for ${\rm A}_1^{3,2}$ and ${\rm A}_1^{3,3}$ leads to
\begin{equation}\label{eq-A13}
{\rm A}_1^{3}\leq   C\|\sigma\|^2_{L^\infty(\mathbb{S}^2)}.
\end{equation}
Collecting estimates \eqref{eq-A11}, \eqref{eq-A12} and \eqref{eq-A13}, we immediately have
\begin{equation*}
{\rm A}_1\leq   C\|\sigma\|^2_{L^\infty(\mathbb{S}^2)}.
\end{equation*}
This inequality together with \eqref{eq.20220827-1} and \eqref{eq.20220827-2} yields the claim \eqref{eq.claim-2}. So, we complete the proof of Theorem \ref{thm-leray-per}.

\section*{Acknowledgments}   This work is supported in part by the National Natural Science Foundation of China
 under grant    No.11831004,  No. 12026407,  and No.11826005.

%%%%%%%%%%%%%%%%%%%%%%%%%%%%%%%%%%%%%%%%%%%%%%%%%%%%%%%%%%%%%%%%%%%%%%%%%%%%%%%%%%%%%%%%%%%%%%%%%%%%%%%%%%%%%%%%%%%%%%%%%%%%%%%

%%%%%%%%%%%%%%%%%%%%%%%%%%%%%%%%%%%%%%%%%%%%%%%%%%%%%%%%%%%%%%%%%%%%%%%%%%%%%%%%%%%%%%%%%%%%%%%%%%%%
\end{document}